\documentclass[letterpaper,english,preprintnumbers,amsmath,amssymb,superscriptaddress,nofootinbib,prl]{revtex4}
\usepackage{mathpazo}
\usepackage[T1]{fontenc}
\usepackage[latin9]{inputenc}
\usepackage{float}
\usepackage{amsmath}
\usepackage{amsthm}
\usepackage{amssymb}
\usepackage{graphicx}
\usepackage{esint}

\makeatletter

\pdfpageheight\paperheight
\pdfpagewidth\paperwidth

\providecommand{\tabularnewline}{\\}
\floatstyle{ruled}
\newfloat{algorithm}{tbp}{loa}
\providecommand{\algorithmname}{Algorithm}
\floatname{algorithm}{\protect\algorithmname}

\@ifundefined{textcolor}{}
{%
 \definecolor{BLACK}{gray}{0}
 \definecolor{WHITE}{gray}{1}
 \definecolor{RED}{rgb}{1,0,0}
 \definecolor{GREEN}{rgb}{0,1,0}
 \definecolor{BLUE}{rgb}{0,0,1}
 \definecolor{CYAN}{cmyk}{1,0,0,0}
 \definecolor{MAGENTA}{cmyk}{0,1,0,0}
 \definecolor{YELLOW}{cmyk}{0,0,1,0}
}
  \theoremstyle{remark}
  \newtheorem{rem}{\protect\remarkname}
  \theoremstyle{definition}
  \newtheorem{defn}{\protect\definitionname}
\theoremstyle{plain}
\newtheorem{thm}{\protect\theoremname}
  \theoremstyle{plain}
  \newtheorem{prop}{\protect\propositionname}
  \theoremstyle{plain}
  \newtheorem{cor}{\protect\corollaryname}
  \theoremstyle{remark}
  \newtheorem*{rem*}{\protect\remarkname}
  \theoremstyle{plain}
  \newtheorem*{conjecture*}{\protect\conjecturename}

\usepackage{babel}
\providecommand{\conjecturename}{Conjecture}
  \providecommand{\definitionname}{Definition}
  \providecommand{\remarkname}{Remark}
\providecommand{\corollaryname}{Corollary}
\providecommand{\theoremname}{Theorem}

\makeatother

\usepackage{babel}
  \providecommand{\conjecturename}{Conjecture}
  \providecommand{\definitionname}{Definition}
  \providecommand{\propositionname}{Proposition}
  \providecommand{\remarkname}{Remark}
\providecommand{\corollaryname}{Corollary}
\providecommand{\theoremname}{Theorem}

\begin{document}

\title{Eigenvalue Attraction}

\author{Ramis Movassagh}

\email{q.eigenman@gmail.com}

\selectlanguage{english}%

\affiliation{Department of Mathematics, IBM T.J. Watson Research Center, Yorktown
Heights, NY, 10598}

\date{\today}
\begin{abstract}
We prove that the complex conjugate (c.c.) eigenvalues of a smoothly
varying real matrix attract (Eq. \ref{eq:EigAttrac}). We offer a
dynamical perspective on the motion and interaction of the eigenvalues
in the complex plane, derive their governing equations and discuss
applications. C.c.\ pairs closest to the real axis, or those that
are ill-conditioned, attract most strongly and can collide to become
exactly real. As an application we consider random perturbations of
a fixed matrix $M$. If $M$ is Normal, the total expected force on
any eigenvalue is shown to be only the attraction of its c.c.\ (Eq.
\ref{eq:TotalNormalForce}) and when $M$ is circulant the strength
of interaction can be related to the power spectrum of white noise.
We extend this by calculating the expected force (Eq. \ref{eq:Expected_cc_attraction})
for real stochastic processes with zero-mean and independent intervals.
To quantify the dominance of the c.c.\ attraction, we calculate the
variance of other forces. We apply the results to the Hatano-Nelson
model and provide other numerical illustrations. It is our hope that
the simple dynamical perspective herein might help better understanding
of the aggregation and low density of the eigenvalues of real random
matrices on and near the real line respectively. In the appendix we
provide a Matlab code for plotting the trajectories of the eigenvalues.
\end{abstract}
\maketitle
\tableofcontents{}\newpage{}

\section{\label{sec:Motivation}Background, illustration and summary of main
results }

\subsection{Background}

Much work has been devoted to the understanding of the behavior of
eigenvalues in the presence of randomness. The folklore of random
matrix analysis, especially in the case of Hermitian matrices, suggests
that the eigenvalues of a perturbed matrix repel. This has been pointed
out previously by various authors \cite{TrefethenEmbree2005,TerryTao2009}
and is well known in quantum physics \cite[p. 304-305]{LandauLifshitz1981}.
More recently, in agreement with the universality conjectures, the
level repulsion was proved for the eigenvalues of a Wigner matrix
\cite{Yau2010}. 

The stochastic dynamics of the eigenvalues of Hermitian matrices have
been vigorously studied in the past \cite[recommended]{TTao2012}.
Most celebrated is Dyson's Brownian motion, which proves that the
eigenvalues of a Hermitian matrix undergoing a Wiener process perform
a Brownian motion \cite{Dyson1962}.

In physics, one mainly studies Hermitian matrices and operators as
their eigenvalues correspond to observable quantities, which need
to be real. However, in recent years, non-Hermitian models have gained
much attention in the context of pinning of vortices in type II superconductors
initiated by Hatano and Nelson \cite{HatanoNelson1997} and followed
up in works on the nature of localized states and eigenvalue distributions
\cite{BrezinZee1998,Brouwer1997,FeinbergZee1997,Widom7}. Non-Hermitian
models also come up in fluid mechanics \cite[Ref. therein]{TrefethenEmbree2005},
transport phenomena in photonics \cite{Ramezani2011} and biophysical
phenomena \cite{Nelson2012}. 

In the Hatano-Nelson model, the eigenvalue distribution gives rise
to ``wings'' of real eigenvalues when the perturbation is sufficiently
strong (see for example \cite[Section 36]{TrefethenEmbree2005} and
citations therein as well as Figs. \ref{fig:Hatano-Nelson} and \ref{fig:Dynamics_HN_different_g}).
The wings result from the motion of complex eigenvalues that move
in response to the perturbation and ultimately sit on the real axis.
Goldsheid et al derived an equation for the shape of ``the winged''
spectrum \cite{Goldsheid1998}. According to \cite{HatanoNelson1997},
these eigenvalues correspond to localized eigenstates.

In investigating the (de)localization of the eigenstates, Feinberg
and Zee \cite{FeinbergZee1997}, argued that imaginary eigenvalues
near the real axis can attract when perturbed by a Hermitian matrix
by providing a $2\times2$ example of an imaginary diagonal matrix
perturbed by a $2\times2$ Hermitian matrix with zero diagonal entries.
Later, Bloch et al \cite{Bloch2012} considered antisymmetric perturbations
of real symmetric matrices in the context of two-color quantum chromodynamics
and provided examples that a Hermitian matrix perturbed by a real
antisymmetric perturbation can give rise to attraction of eigenvalues.
To our knowledge, attraction of the eigenvalues and their eventual
aggregation on the real line, in a general setting, has not been proved. 

In this paper, we prove a simple theorem, which under very mild set
of assumptions shows that the complex conjugate eigenvalues of a smoothly
varying real matrix $M\left(t\right)$ attract (i.e., pull on each
other). We then consider probabilistic settings where the underlying
evolution of the matrix is random and prove statistical dynamical
properties of any given eigenvalue. We extend the results to real
stochastic processes, which naturally leads to a conjecture on the
cause of aggregation and low density of the eigenvalues of real random
matrices on and near the real line. A main emphasis of this work is
a many-body dynamical perspective on the interaction and motion of
the eigenvalues in the complex plane. Throughout, we illustrate the
theory with examples and numerical results.

\subsection{\label{sub:Illustration}An illustration}
\begin{rem}
\textbf{Explanation of the figures:} All the plots were done in Matlab.
We take the vertical (horizontal) axis to be the imaginary (real)
axis and plot the eigenvalues of $M\left(t\right)$. The red dots
are the eigenvalues of $M\left(0\right)$. To show the dynamics of
the eigenvalues as a function of $t$, we plot the eigenvalues of
$M\left(t\right)$ in the complex plane in gray scale, where at $t=0$
they are shown in white (coincide with the red dots) and darken as
$t$ increases till their final position at $t=t_{max}$ shown in
black. The eigenvalues of $M\left(t\right)$ at any $0\le t\le t_{max}$
have the same gray scale color. In Matlab we use ``$\mathtt{hold}\mbox{ }\mathtt{on;}$''
to show the eigenvalues for all $t$. In the appendix we provide a
Matlab code that can be used to plot the trajectory of the eigenvalues
similar to what is done here.
\end{rem}
Demo 1: In (Fig. \ref{fig:Stochastic-dynamics-of}, left) we show
the spectral dynamics of the Hatano-Nelson model $M\left(t\right)=H+\delta t\mbox{ }P$,
where $M\left(0\right)\equiv H$ is 
\begin{figure}
\centering{}\includegraphics[scale=0.4]{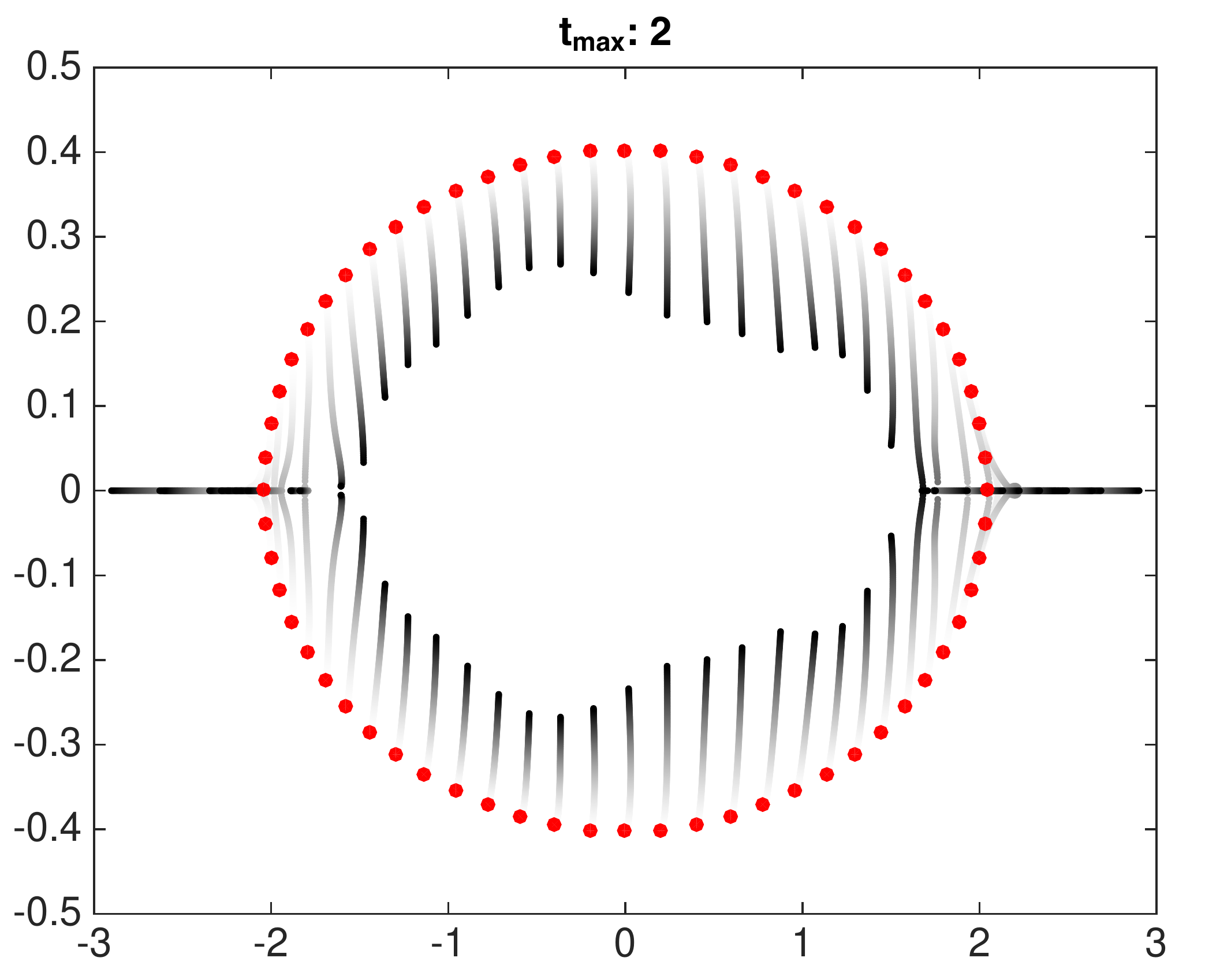}$\quad$\includegraphics[scale=0.34]{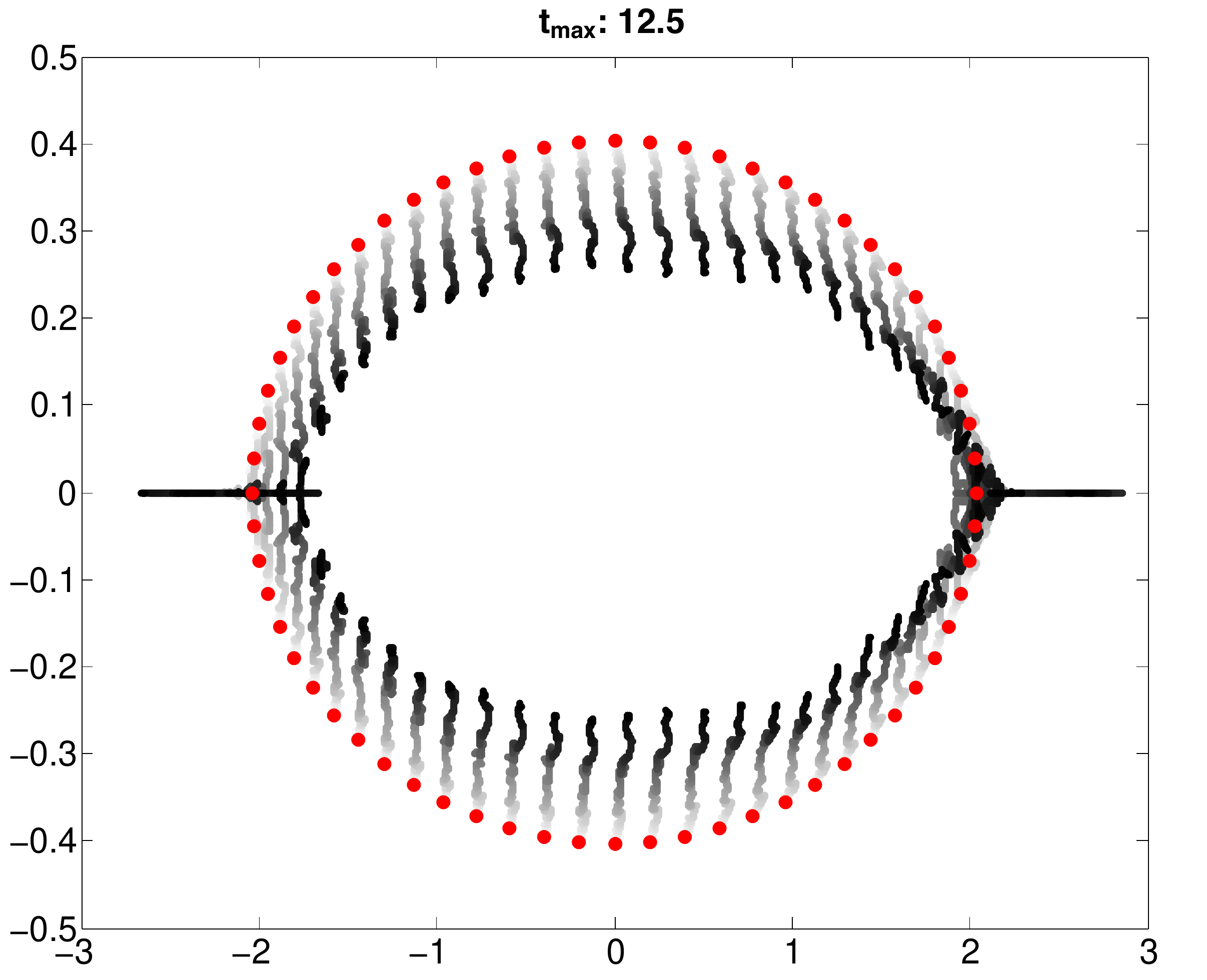}\caption{\label{fig:Hatano-Nelson}Hatano-Nelson model with $g=0.2$. Left
(Demo 1): Small perturbation. Right (Demo 2): \label{fig:Stochastic-dynamics-of}
Stochastic dynamics. In Demo 2 the final time is much larger because
the random impulses imparted at times $0.25i$ make for much slower
net motion over time.}
\end{figure}
 
\begin{equation}
H=\left[\begin{array}{ccccc}
0 & e^{g} &  &  & e^{-g}\\
e^{-g} & 0 & e^{g}\\
 & e^{-g} & 0 & \ddots\\
 &  & \ddots & \ddots & e^{g}\\
e^{g} &  &  & e^{-g} & 0
\end{array}\right]\quad,\label{eq:Hatano-Nelson}
\end{equation}
where $g$ governs the non-Hermiticity and $P$ is a real diagonal
matrix of random Gaussians. In this figure we took $t_{max}=2$, $g=0.2$,
$n=64$, which makes $\left\Vert H\right\Vert _{2}=2.04$. The boundary
terms, i.e., $\left(1,n\right)$ and $\left(n,1\right)$ entries,
are crucial for the spectral properties \cite[Sec. 36]{TrefethenEmbree2005}. 

Comment: When $g\rightarrow0$, this model coincides with the Anderson
model of localization with periodic boundary conditions \cite{Anderson58}.
However, the latter is a Hermitian model whose properties are quite
different from the  Hatano-Nelson model with $g\ne0$.\\

Demo 2: In (Fig. \ref{fig:Stochastic-dynamics-of}, right) We took
$g=0.2$ as before and take the time discretization $t_{i}=0.25i$
for $i=0,1,\cdots,50$, when at each $t_{i}$ an independent matrix
$P\left(t_{i}\right)$ is introduced according to $P\left(t_{i}\right)=\mbox{diag}\left(\epsilon_{1},\epsilon_{2},\cdots,\epsilon_{n}\right)$
where $\epsilon_{j}$'s are drawn independently from a standard normal
distribution and each $P\left(t_{i}\right)$ is normalized to have
a unit $2-$norm; i.e, $\mbox{diag}\left[P\left(t_{i}\right)\right]$
is uniform on the unit sphere. A piece-wise linear discrete stochastic
process is therefore constructed. We plot the eigenvalues of $M\left(0.25i+\delta t\right)=M\left(0.25i\right)+\delta t\mbox{ }P\left(t_{i}\right)$
with $0\le\delta t\le0.25$ and time steps $\Delta t=0.01$. In (Fig.
\ref{fig:Stochastic-dynamics-of}, right), we show the evolution of
$M\left(t\right)$ for the total time $t\in\left[0,12.5\right]$.
Note that the eigenvalues move towards the real line as before but
make less progress because of the stochastic kicks at times $0.25i$.
Below we show that there is an attraction between complex conjugate
pairs that largely governs this behavior of the spectral dynamics
in this case. See Section \ref{sec:Stochatic-dynamics} for further
theoretical discussion.

Comment: The procedure just described does not provide a smooth stochastic
process as it is not differentiable at times $0.25i$ (a set of measure
zero). This issue persists for piece-wise linear discretization where
in each interval a new random $P\left(t_{i}\right)$ is used. In fact
a continuous Brownian motion provides an example of a continuous map
that, with probability one, is nowhere differentiable \cite{TTao2012}.

\subsection{Summary of the main results}

In this paper we consider the general problem of the interaction of
the eigenvalues of a matrix, $M\left(t\right)$, whose entries vary
with respect to a real parameter $t$, which we think of as time.
Although, eigenvalue repulsion is quite a general feature of Hermitian
matrices, attraction of eigenvalues is rarely considered. Below we
will first define Eigenvalue Attraction (see Def. \ref{(Attraction-and-Repulsion)})
and give the general equations of the motion of any one of the eigenvalues,
which depends on the other eigenvalues and the inertia of the matrix
itself. We introduce a new notation more akin to formulations of interacting
many-body systems. Theorem \ref{Thm:(Eigenvalue-Attraction)} gives
a simple proof of complex conjugate attraction. Namely, any non-real
eigenvalue of $M\left(t\right)\in\mathbb{R}^{n\times n}$ attracts
its complex conjugate according to
\[
\mbox{Force of }\overline{\lambda_{i}}\mbox{ on }\lambda_{i}=-i\mbox{ }\frac{\left|\mathbf{u}_{i}^{T}\overset{\centerdot}{M}\left(t\right)\mathbf{v}_{i}\right|^{2}}{\mbox{Im}\left(\lambda_{i}\left(t\right)\right)},
\]
where $\mathbf{v}_{i}$ and $\mathbf{u}_{i}$ are the right (i.e.,
standard) and left eigenvectors corresponding to the eigenvalue $\lambda_{i}$
and $\overline{\lambda_{i}}$ denotes the complex conjugate of $\lambda_{i}$. 

We will apply this to various $M\left(t\right)$ such as the Hatano-Nelson
model or a convex combination of two deterministic matrices. We then
introduce randomness to obtain probabilistic statements. In particular,
we consider the pencil of matrices $M\left(t\right)\equiv M+\delta t\mbox{ }P$,
where $M$ is a fixed matrix and $P$ is a random real matrix whose
entries have zero mean, finite fourth moments and are independently
and identically distributed (iid). This special case of perturbing
a fixed matrix comes up often in applications. In this limit and in
order to quantify the dominance of the complex conjugate attraction,
we calculate the expectation and variance of all other forces excluding
the complex conjugate. We prove that when $M$ is a normal matrix
(i.e., unitary diagonalizable), the total expected force on any eigenvalue
is only due to its complex conjugate, and when $M$ is circulant and
$P$ is diagonal, the force only depends on the eigenvalues and that
the strength of interaction is the power spectrum of the diagonal
entries of $P$, which can be a constant independent of the eigenpairs.
For example, if $p_{ii}\sim{\cal N}\left(0,1\right)$ then the strength
of interaction is the power spectrum of white noise. We will prove
other results applicable to general circulant matrices and apply them
to the Hatano-Nelson model to analyze its spectral dynamics. 

We then make a time discretization $0<t_{1}<t_{2}<\cdots$ and define
$M\left(t\right)$ to be a stochastic process defined by $M\left(t_{i}+\delta t\right)=M\left(t_{i}\right)+\delta t\mbox{ }P\left(t_{i}\right)$,
where $\delta t\in\left[0,t_{i+1}-t_{i}\right)$, $M\left(0\right)$
is a fixed real matrix and each $P\left(t_{i}\right)$ is a real and
random, whose entries are independent with zero mean and finite fourth
moments. We construct a smooth family of stochastic processes $\overset{\centerdot}{M}_{\epsilon}\left(t\right)=P_{\epsilon}\left(t\right)$,
where $M\left(t\right)=\lim_{\epsilon\rightarrow0}M_{\epsilon}\left(t\right)$
and explicitly write down the differential equations governing the
motion of any eigenvalue. The expected force of attraction is always
\begin{eqnarray}
\mathbb{E}[\mbox{Force of }\overline{\lambda_{i}}\mbox{ on }\lambda_{i}] & = & -i\frac{\sum_{m,\ell}\mathbb{E}\left[p_{m\ell}^{2}\right]\left|u_{i}^{*,m}\right|^{2}\left|v_{i}^{,\ell}\right|^{2}}{2\;\mbox{Im}\left(\lambda_{i}\right)}\label{eq:FinalResult}\\
 & \overset{\mbox{iid}}{=} & -i\frac{\mathbb{E}\left[p^{2}\right]\left\Vert \mathbf{u_{i}}\right\Vert _{2}^{2}}{2\;\mbox{Im}\left(\lambda_{i}\right)}\label{eq:FinalREsultiid}
\end{eqnarray}
where $\overline{\lambda_{i}}$ is an eigenvalue that is complex conjugate
to $\lambda_{i}$, $\mathbf{v_{i}}$ and $\mathbf{u_{i}^{*}}$ are
the corresponding right and left eigenvectors respectively and the
second equality assumes $\mathbb{E}\left[p_{m,\ell}^{2}\right]$ is
the same for all $m$ and $\ell$. Clearly, the attraction is strongest
near the real line. Since the proportionality constant depends on
the $2-$norm of the left eigenvector, the force of attraction can
be quite strong for ill-conditioned eigenvalues. When this attractive
force is dominant over the force exerted by the rest of the eigenvalues,
the complex conjugate pair approach one another and eventually collide
and ``scatter'' near and ultimately reside at different points on
the real line. At this point, the well-known repulsion mechanism takes
over and the reality of the matrix ensures that each eigenvalue remains
real.

The motion on the real line is not permanent. In most cases, an eigenvalue
that moves about on the real line gets close enough to (i.e., collides
with) another eigenvalue on the real line, after which they form a
new complex conjugate pair and shoot off into the complex plane. The
alternative would be that they would repel and remain on the real
line. However, in the majority of cases we investigated, it seems
'energetically' more favorable for them to form a new complex conjugate
pair perhaps because there are more degrees of freedom available away
from the real line. In this paper, we will not rigorously investigate
this to any depth. 

As mentioned above, any stochastic process (e.g., Wiener process,
Brownian motion) is non-smooth, despite often being continuous. The
appearance of a new $P\left(t_{i}\right)$ makes the limits of the
derivative from left and right unequal $\overset{\centerdot}{M}\left(t_{i}^{>}\right)\ne\overset{\centerdot}{M}\left(t_{i}^{<}\right)$,
yet there are powerful tools of matrix calculus that can be utilized
if $M\left(t\right)$ were differentiable. Moreover, from the applied
perspective, nothing is instantaneous. 

In Section \ref{sec:Stochatic-dynamics}, we give the basic definitions
of discrete stochastic processes and introduce a smoothing construction
that can be used to smoothen any discrete stochastic process with
a control over the rapidity of (dis)appearance of every $P\left(t_{i}\right)$
within $\left[t_{i},t_{i+1}\right]$ (see Eq. \ref{eq:M_t}). The
original (non-differentiable) stochastic process is
\[
M\left(t\right)=\lim_{\epsilon\rightarrow0}M_{\epsilon}\left(t\right)\quad.
\]

We will conclude by applying the dynamical perspective developed here
to an open problem pertaining to the sparsity of the eigenvalues of
random real matrices near the real line, and then list further open
problems. It is our hope that this work will prove useful in proving
the conjecture stated in Sec \ref{sec:Further-discussions}. 
\begin{rem}
We make a remark that should otherwise be obvious. In what follows
the eigenvalue attraction holds for all $t$ and $M(t)$ does {\it not}
need to be a perturbation of a fixed matrix. The latter is, however,
an application of this work (see Subsection \ref{sec:Random-perturbations}).
Moreover, attraction holds for deterministic evolutions under minimal
assumptions and  randomness is {\it not} a requirement.
\end{rem}

\section{\label{sec:Dynamics-of-eigenvalues}Eigenvalues as a many-body system}

In this paper we take the point of view that the eigenvalues are interacting
identical particles whose motions take place in the complex plane
and our goal is to better understand their dynamics. The eigenvalues
of $M\left(t\right)$ are also functions of time and the $i^{th}$
eigenvalue is denoted by $\lambda_{i}\left(t\right)$. The eigenvalues
of a continuously varying $M\left(t\right)$ are also continuous in
$t$. That is, their motion follows a connected path in the complex
plane. This follows from the fact that eigenvalues are roots of a
characteristic polynomial, which itself is continuous, and a theorem
due to Rouché \cite[Chapter 4 ]{StewartSun}. 

Comment: We take the eigenvalues to have unit mass, whereby the acceleration,
$\overset{\centerdot\centerdot}{\lambda_{i}}\left(t\right)$, can
be identified with the ``force'' required to produce that acceleration
on an eigenvalue. Below we shall use the word force as it provides
better intuition.

\subsection{\label{sub:General-dynamics-of}General dynamics of eigenvalues}

Here we follow a derivation similar to that given by T. Tao \cite{TerryTao2009}
to obtain the governing dynamical equations for the eigenpairs of
a general smoothly varying $M\left(t\right)$, although the equations
were derived in earlier references. We assume that the eigenvalues
are simple. The eigenvalue equations are 
\begin{eqnarray}
M\left(t\right)\mbox{ }\mathbf{v_{i}}\left(t\right) & = & \lambda_{i}\left(t\right)\mathbf{v_{i}}\left(t\right)\label{eq:LeftEvec-1}\\
\mathbf{u_{i}}^{*}\left(t\right)\mbox{ }M\left(t\right) & = & \lambda_{i}\left(t\right)\mathbf{u_{i}}^{*}\left(t\right),\label{eq:RightEvec-1}
\end{eqnarray}
where $\lambda_{i}\left(t\right)$ are the eigenvalues, $\mathbf{v_{i}}\left(t\right)$
the (right) eigenvectors, which we take to be normalized, and $\mathbf{u_{i}^{*}}\left(t\right)$
are the left eigenvectors dual to $\mathbf{v_{i}}\left(t\right)$.
If we consider the matrix of eigenvectors $\mathbf{V}\left(t\right)=\left[\mathbf{v_{1}}\left(t\right)\mbox{ }\mathbf{v_{2}}\left(t\right),\dots,\mathbf{v_{n}}\left(t\right)\right]$,
then $\mathbf{u_{j}}^{*}\left(t\right)$ is the $j^{\mbox{th}}$ row
of $\mathbf{V}^{-1}\left(t\right)$ and 
\begin{equation}
\mathbf{u_{j}}^{*}\left(t\right)\mbox{ }\mathbf{v_{i}}\left(t\right)=\delta_{ij}\quad.\label{eq:orthogonality}
\end{equation}

Since $\mathbf{v}_{1}\left(t\right),\dots,\mathbf{v}_{n}\left(t\right)$
form a basis for $\mathbb{C}^{n}$, any vector $\mathbf{x}$ has the
expansion $\mathbf{x}=\sum_{j=1}^{n}\left[\mathbf{u_{j}}^{*}\left(t\right)\mbox{ }\mathbf{x}\right]\mathbf{v_{j}}\left(t\right)$.
Differentiating Eqs. \ref{eq:LeftEvec-1} and \ref{eq:RightEvec-1}
with respect to $t$, gives 
\begin{eqnarray*}
\overset{\centerdot}{M}\left(t\right)\mbox{ }\mathbf{v_{i}}\left(t\right)+M\left(t\right)\overset{\centerdot}{\mathbf{v_{i}}}\left(t\right) & = & \overset{\centerdot}{\lambda_{i}}\left(t\right)\mathbf{v_{i}}\left(t\right)+\lambda_{i}\left(t\right)\overset{\centerdot}{\mathbf{v_{i}}}\left(t\right)\\
\mathbf{u_{i}}^{*}\left(t\right)\mbox{ }\overset{\centerdot}{M}\left(t\right)+\overset{\centerdot}{\mathbf{u_{i}}^{*}}\left(t\right)\mbox{ }M\left(t\right) & = & \overset{\centerdot}{\lambda_{i}}\left(t\right)\mathbf{u_{i}^{*}}\left(t\right)+\lambda_{i}\left(t\right)\overset{\centerdot}{\mathbf{u_{i}}^{*}}\left(t\right)
\end{eqnarray*}
 Multiplying the first equation on the left by $\mathbf{u_{i}}^{*}\left(t\right)$,
and using Eq. \ref{eq:RightEvec-1} , we obtain the ``velocity''
of $\lambda\left(t\right)$ in the complex plane \footnote{We remark that the theory of pseudo-spectra \cite{TrefethenEmbree2005}
quantifies how far an eigenvalue can wander without quantifying the
direction of the motion} 
\begin{equation}
\overset{\centerdot}{\lambda_{i}}\left(t\right)=\mathbf{u_{i}}^{*}\left(t\right)\mbox{ }\overset{\centerdot}{M}\left(t\right)\mbox{ }\mathbf{v_{i}}\left(t\right).\label{eq:lambda_prime}
\end{equation}

In order to compute the acceleration on any eigenvalue we shall need
the derivatives of the left and right eigen\textit{vectors}. They
are simple to compute \cite{TerryTao2009},

\begin{eqnarray}
\overset{\centerdot}{\mathbf{v_{i}}}\left(t\right) & = & \sum_{j\ne i}\frac{\mathbf{u_{j}^{*}}\left(t\right)\mbox{ }\overset{\centerdot}{M}\left(t\right)\mbox{ }\mathbf{v_{i}}\left(t\right)}{\lambda_{i}\left(t\right)-\lambda_{j}\left(t\right)}\mathbf{v_{j}}\left(t\right)+\eta_{i}\left(t\right)\mathbf{v_{i}}\left(t\right)\label{eq:v_p}\\
\overset{\centerdot}{\mathbf{u_{i}}^{*}}\left(t\right) & = & \sum_{j\ne i}\frac{\mathbf{u_{i}^{*}}\left(t\right)\mbox{ }\overset{\centerdot}{M}\left(t\right)\mbox{ }\mathbf{v_{j}}\left(t\right)}{\lambda_{i}\left(t\right)-\lambda_{j}\left(t\right)}\mathbf{u_{j}^{*}}\left(t\right)-\eta_{i}\left(t\right)\mathbf{\mathbf{u_{i}^{*}}}\left(t\right)\label{eq:u_p}
\end{eqnarray}
 where $\eta_{i}\left(t\right)$ is a scalar function because a constant
multiple of an eigenvector is also an eigenvector. The second derivative,
or acceleration, of the eigenvalue $\lambda_{i}\left(t\right)$ is
obtained by differentiating Eq. \ref{eq:lambda_prime} one more time,
\[
\overset{\centerdot\centerdot}{\lambda_{i}}\left(t\right)=\overset{\centerdot}{\mathbf{u_{i}}^{*}}\left(t\right)\mbox{ }\overset{\centerdot}{M}\left(t\right)\mbox{ }\mathbf{v_{i}}\left(t\right)+\mathbf{u_{i}}^{*}\left(t\right)\mbox{ }\overset{\centerdot\centerdot}{M}\left(t\right)\mbox{ }\mathbf{v_{i}}\left(t\right)+\mathbf{u_{i}}^{*}\left(t\right)\mbox{ }\overset{\centerdot}{M}\left(t\right)\mbox{ }\overset{\centerdot}{\mathbf{v_{i}}}\left(t\right).
\]

Using Eqs. \ref{eq:v_p} and \ref{eq:u_p}, the second derivative
becomes \cite{TerryTao2009} 

\begin{eqnarray}
\overset{\centerdot\centerdot}{\lambda_{i}}\left(t\right) & = & \mathbf{u_{i}}^{*}\left(t\right)\mbox{ }\overset{\centerdot\centerdot}{M}\left(t\right)\mbox{ }\mathbf{v_{i}}\left(t\right)+2\sum_{j\ne i}\frac{\left[\mathbf{u_{i}^{*}}\left(t\right)\mbox{ }\overset{\centerdot}{M}\left(t\right)\mbox{ }\mathbf{v_{j}}\left(t\right)\right]\left[\mathbf{u_{j}^{*}}\left(t\right)\mbox{ }\overset{\centerdot}{M}\left(t\right)\mbox{ }\mathbf{v_{i}}\left(t\right)\right]}{\lambda_{i}\left(t\right)-\lambda_{j}\left(t\right)}\label{eq:lambda_pp_final}\\
 & = & \mathbf{v_{i}}^{*}\left(t\right)\mbox{ }\overset{\centerdot\centerdot}{M}\left(t\right)\mbox{ }\mathbf{v_{i}}\left(t\right)+2\sum_{j\ne i}\frac{\left|\mathbf{v_{i}^{*}}\left(t\right)\mbox{ }\overset{\centerdot}{M}\left(t\right)\mbox{ }\mathbf{v_{j}}\left(t\right)\right|^{2}}{\lambda_{i}\left(t\right)-\lambda_{j}\left(t\right)}\quad\mbox{if the matrix is Hermitian},\label{eq:lambda_pp_final_Normal}
\end{eqnarray}
where for normal matrices (e.g., Hermitian) one has $\mathbf{u}_{i}=\mathbf{v}_{i}$.
As pointed out by Tao, the first term can be seen as the inertial
force of the matrix and the second the force of interaction of the
eigenvalues. The origin of instantaneous repulsive force between eigenvalues
of a Hermitian matrix is easily seen by the second term in Eq. \ref{eq:lambda_pp_final_Normal}
\cite{TerryTao2009}. For example, if $\lambda_{i}\left(t\right)>\lambda_{j}\left(t\right)$,
then the force is positive, the effect of $\lambda_{j}\left(t\right)$
in the sum is to push $\lambda_{i}\left(t\right)$ to the right. Similarly,
if $\lambda_{i}\left(t\right)<\lambda_{j}\left(t\right)$, the effect
of $\lambda_{j}\left(t\right)$ is to exert a negative force on $\lambda_{i}\left(t\right)$.
Moreover the strength of the force is inversely proportional to their
distance ($1/\left|\lambda_{i}\left(t\right)-\lambda_{j}\left(t\right)\right|$)
which is clearly strongest when the eigenvalues are closest \cite{TerryTao2009}.
The repulsion is at work for Hermitian matrices for all $t$. 

Comment: Eqs. (\ref{eq:lambda_prime}-\ref{eq:lambda_pp_final_Normal})
are essentially the standard first and second order perturbation theory
results. See for example, Dirac \cite[Section 43]{dirac1958principles}
\footnote{Strictly speaking Dirac's derivation of Equation 10 in Section 43
of this reference, does not hold in general (e.g., non-Hermitian)
as the left eigenvectors are not 'bras' in his notation. The latter
is a Hermitian conjugate of a standard (right) eigenvector. In his
book, Dirac cites (Born, Heisenberg and Jordan, z.f. Physik 35, 565
(1925)) for these formulas}, Kato \cite{kato1976perturbation}, Wilkinson's wonderful exposition
\cite{wilkinson1965algebraic} and more recently \cite{tao2011random}.

\subsection{\label{sub:Eigenvalue-Attraction-Proof}Definition and Proof of Eigenvalue
Attraction}

Below we denote the complex conjugate of an eigenvalue or entry-wise
complex conjugation of an eigenvector with an over-line. To better
visualize the kinematics of the eigenvalues in the complex plane,
we write 
\begin{eqnarray}
\frac{1}{\lambda_{i}\left(t\right)-\lambda_{j}\left(t\right)} & = & \frac{\overline{\lambda_{i}\left(t\right)}-\overline{\lambda_{j}\left(t\right)}}{\left|\lambda_{i}\left(t\right)-\lambda_{j}\left(t\right)\right|^{2}}\equiv\frac{\hat{\mathbf{r}}_{ij}}{\left|\mathbf{r}_{ij}\right|};\label{eq:centralForce}\\
\mathbf{r}_{ij} & \equiv & \overline{\lambda_{i}\left(t\right)}-\overline{\lambda_{j}\left(t\right)}\nonumber \\
\hat{\mathbf{r}}_{ij} & \equiv & \frac{\overline{\lambda_{i}\left(t\right)}-\overline{\lambda_{j}\left(t\right)}}{\left|\lambda_{i}\left(t\right)-\lambda_{j}\left(t\right)\right|}\nonumber 
\end{eqnarray}
where $\mathbf{r}_{ij}$ is a vector in the complex plane stretching
from $\overline{\lambda_{j}\left(t\right)}$ to $\overline{\lambda_{i}\left(t\right)}$
and $\hat{\mathbf{r}}_{ij}$ is the corresponding unit vector. One
could further simplify the notation by denoting the complex number
\[
c_{ij}\equiv\mathbf{u_{i}^{*}}\left(t\right)\mbox{ }\overset{\centerdot}{M}\left(t\right)\mbox{ }\mathbf{v_{j}}\left(t\right)\quad,
\]
whereby $\overset{\centerdot}{\lambda_{i}}\left(t\right)=c_{ii}$
. In this paper, we denote complex conjugation of the entries of an
eigenvector by an over-bar on the corresponding index. For example,
$c_{\bar{i}\mbox{ }j}=\mathbf{u_{i}^{T}}\left(t\right)\mbox{ }\overset{\centerdot}{M}\left(t\right)\mbox{ }\mathbf{v_{j}}\left(t\right)$
and since $\overset{\centerdot}{M}\left(t\right)$ is real, $c_{i\mbox{ }\bar{j}}\mbox{ }c_{\bar{i}\mbox{ }j}=\left|\mathbf{u_{i}^{T}}\left(t\right)\mbox{ }\overset{\centerdot}{M}\left(t\right)\mbox{ }\mathbf{v_{j}}\left(t\right)\right|^{2}$
is a real non-negative number.

With this notation Eqs. \ref{eq:lambda_pp_final} and \ref{eq:lambda_pp_final_Normal}
read

\begin{eqnarray}
\overset{\centerdot\centerdot}{\lambda_{i}}\left(t\right) & = & \mathbf{u_{i}}^{*}\left(t\right)\mbox{ }\overset{\centerdot\centerdot}{M}\left(t\right)\mbox{ }\mathbf{v_{i}}\left(t\right)+2\sum_{j\ne i}c_{ij}c_{ji}\frac{\hat{\mathbf{r}}_{ij}}{\left|\mathbf{r}_{ij}\right|}\label{eq:lambda_pp_finalRAMIS-1}\\
\overset{\centerdot\centerdot}{\lambda_{i}}\left(t\right) & = & \mathbf{v_{i}}^{*}\left(t\right)\mbox{ }\overset{\centerdot\centerdot}{M}\left(t\right)\mbox{ }\mathbf{v_{i}}\left(t\right)+2\sum_{j\ne i}\left|c_{ij}\right|^{2}\frac{\hat{\mathbf{r}}_{ij}}{\left|\mathbf{r}_{ij}\right|}\quad\mbox{if the matrix is Hermitian}.\label{eq:lambda_pp_final_HermitianRAMIS}
\end{eqnarray}

We think of Eq. \ref{eq:lambda_pp_finalRAMIS-1} as 
\begin{eqnarray*}
\overset{\centerdot\centerdot}{\lambda_{i}}\left(t\right) & = & \mathbf{u_{i}}^{*}\left(t\right)\mbox{ }\overset{\centerdot\centerdot}{M}\left(t\right)\mbox{ }\mathbf{v_{i}}\left(t\right)+2\sum_{j\ne i}c_{ij}c_{ji}\frac{\hat{\mathbf{r}}_{ij}}{\left|\mathbf{r}_{ij}\right|}\\
 & = & \left\{ \mbox{Inertial force of }M\right\} +\sum_{j\ne i}\left\{ \mbox{force of }\lambda_{j}\mbox{ on }\lambda_{i}\right\} .
\end{eqnarray*}

\begin{defn}
\label{def:(Central-Force)-The} The force between $\lambda_{i}\left(t\right)$
and $\lambda_{j}\left(t\right)$ is called central if in Eq. \ref{eq:lambda_pp_finalRAMIS-1},
$c_{ij}c_{ji}=f(\lambda_{i}\left(t\right),\lambda_{j}\left(t\right),\dots)\left(\lambda_{i}\left(t\right)-\lambda_{j}\left(t\right)\right)$,
where $f(\lambda_{i}\left(t\right),\lambda_{j}\left(t\right),\dots)$
is a real function.
\end{defn}
Comment: As shown above $f(\lambda_{i}\left(t\right),\lambda_{j}\left(t\right),\dots)$
is generally a complex-valued function of $i^{th}$ and $j^{th}$
eigenvalues and eigenvectors as well as $\overset{\centerdot}{M}\left(t\right)$.
\begin{defn}
\label{(Attraction-and-Repulsion)}(Attraction and Repulsion) We say
$\lambda_{j}\left(t\right)$ attracts (repels) $\lambda_{i}\left(t\right)$
if the force between them is central and $f(\lambda_{i}\left(t\right),\lambda_{j}\left(t\right),\dots)$
is a negative (positive) function.\end{defn}
\begin{rem}
For general (not self-adjoint) matrices the force between any two
eigenvalues is not necessarily repulsive nor attractive. As seen in
Eq.\ref{eq:lambda_pp_finalRAMIS-1} , the orientation of the force
is along the ray $c_{ij}c_{ji}\hat{\mathbf{r}}_{ij}$, rendering generally
a non-central force law between the eigenvalues. This contrasts the
purely central (and repulsive) nature of the interaction of the eigenvalues
for Hermitian matrices (Eq. \ref{eq:lambda_pp_final_HermitianRAMIS}).
However, it is generally true that the force law between any two eigenvalues
is inversely proportional to their distance. \end{rem}
\begin{thm}
(Eigenvalue Attraction) \label{Thm:(Eigenvalue-Attraction)}Complex
conjugate eigenvalues of $M\left(t\right)$ attract (see Def. \ref{(Attraction-and-Repulsion)})
as long as $M\left(t\right)$ is real, the pair is not degenerate,
and $\mathbf{u_{i}^{T}}\left(t\right)\mbox{ }\overset{\centerdot}{M}\left(t\right)\mbox{ }\mathbf{v_{i}}\left(t\right)\ne0$
.\end{thm}
\begin{proof}
The non-real eigenvalues and eigenvectors of a real matrix come in
complex conjugate pairs \cite[Chapter 24]{Trefethen}. In Eq. \ref{eq:lambda_pp_finalRAMIS-1},
the interaction of $\overline{\lambda_{i}\left(t\right)}$ and $\lambda_{i}\left(t\right)$
is given by the term where $j=\bar{i}$ where $\lambda_{i}\left(t\right)-\overline{\lambda_{i}\left(t\right)}=2i\mbox{ }\mbox{Im}\left(\lambda_{i}\right)$
and 
\begin{eqnarray}
\frac{2\mbox{ }c_{i\mbox{ }\bar{i}}\mbox{ }c_{\bar{i}\mbox{ }i}}{\lambda_{i}\left(t\right)-\overline{\lambda_{i}\left(t\right)}} & = & -i\mbox{ }\frac{\left|c_{i\bar{\mbox{ }i}}\right|^{2}}{\mbox{Im}\left(\lambda_{i}\left(t\right)\right)}\quad.\label{eq:EigAttrac}
\end{eqnarray}

i. $\mbox{Im}\left(\lambda_{i}\right)>0$, then the right hand side
of Eq. \ref{eq:EigAttrac} is a negative imaginary number: The effect
of $\overline{\lambda_{i}}$ on $\lambda_{i}$ at time $t$ is to
push $\lambda_{i}$ downwards along the imaginary axis with a magnitude
that is inversely proportional to their distance. The constant of
proportionality is $\left|c_{i\mbox{ }\bar{i}}\right|^{2}$-- the
numerator of Eq. \ref{eq:EigAttrac}.

ii. $\mbox{Im}\left(\lambda_{i}\right)<0$, then the right hand side
of Eq. \ref{eq:EigAttrac} is a positive imaginary number: The effect
of $\overline{\lambda_{i}}$ on $\lambda_{i}$ at time $t$ is to
push $\lambda_{i}$ upward along the imaginary axis with a magnitude
that is inversely proportional to their distance with the same constant
of proportionality.
\end{proof}
Comment: Let $M$ be a fixed non-symmetric real matrix and define
$M\left(t\right)=M+t\mbox{ }\mathbb{I}$, where  $\mathbb{I}$ is
the identity matrix. The eigenvalues of $M\left(t\right)$ are those
of $M$ shifted by $t$. No attraction is expected. Indeed by Eq.
\ref{eq:orthogonality}, $c_{ii}=\mathbf{u_{i}^{T}}\left(t\right)\mathbf{v_{i}}\left(t\right)=0$.
Moreover, from Eq. \ref{eq:lambda_prime} we have $\overset{\centerdot}{\lambda_{i}}\left(t\right)=t\mbox{ }$,
i.e., a net drift in the complex plane of any eigenvalue with velocity
$t$. 
\begin{rem}
Proving that complex conjugate eigenvalues attract does \textit{not}
imply that in the long-run all the eigenvalues will necessarily be
real. There are three forces that act on any eigenvalue $\lambda_{i}$
: \end{rem}
\begin{enumerate}
\item The inertial force due to $\overset{\centerdot\centerdot}{M}\left(t\right)$. 
\item The attraction of its complex conjugate 
\item The force of the remaining $n-2$ eigenvalues. 
\end{enumerate}
The governing equation is accordingly written as
\begin{equation}
\overset{\centerdot\centerdot}{\lambda_{i}}\left(t\right)=\mathbf{u_{i}}^{*}\left(t\right)\mbox{ }\overset{\centerdot\centerdot}{M}\left(t\right)\mbox{ }\mathbf{v_{i}}\left(t\right)-i\mbox{ }\frac{\left|c_{i\bar{\mbox{ }i}}\right|^{2}}{\mbox{Im}\left(\lambda_{i}\right)}+2\sum_{j\ne\left\{ i,\bar{i}\right\} }c_{ij}c_{ji}\frac{\hat{\mathbf{r}}_{ij}}{\left|\mathbf{r}_{ij}\right|}\quad.\label{eq:KEY_Equation}
\end{equation}
Since the force law is inversely proportional to the distance of the
eigenvalues, the law of interaction needs to be logarithmic i.e.,
the potential is $V\left(\left|\mathbf{r}_{ij}\right|\right)\propto\log\left|\mathbf{r}_{ij}\right|$.
The logarithmic interaction on the plane implies a short-range force,
where the main contribution to the total force comes from the eigenvalues
that are in the vicinity. 
\begin{prop}
The real eigenvalues of $M\left(t\right)\in\mathbb{R}^{n\times n}$
interact via a central force.\end{prop}
\begin{proof}
Suppose $\lambda_{i}$ and $\lambda_{j}$ are two simple real eigenvalues
then the corresponding eigenvectors can be taken to be real. Then
$c_{ij}$ and $c_{ji}$ are both real and by Definition \ref{def:(Central-Force)-The}
interact via a central force.
\end{proof}
A natural question then is: how dominant is the complex conjugate
attraction in determining the motion of any eigenvalue? One cannot
give a general answer to this question as the relevance of complex
conjugate attraction depends on the particular $M\left(t\right)$
and the particular eigenvalue considered. However, in examining Eq.
\ref{eq:KEY_Equation}, it is clear that complex conjugate pairs closest
to the real line attract strongest. We will show that ill-conditioning
enhances the complex conjugate attraction as well. This is expected
as ill-conditioning generally implies higher sensitivity to perturbations
\cite{Trefethen,TrefethenEmbree2005}. Below we provide some exact
results and demonstrations for certain general sub-manifolds of real
matrices and time evolutions that are of interest. Mathematically,
whether the second term in Eq. \ref{eq:KEY_Equation} dominates the
time evolution depends on its magnitude relative to the first derivative
and the other terms on the right hand side (i.e, third term). After
some demonstrations of deterministic evolution in the next section,
we proceed to calculate the expected values and variances about that
expectation for the first and second derivatives when the evolution
is driven by randomness. Further, we derive conditions under which
the total expected force is only due to the complex conjugate. 
\begin{rem}
When the dominant force is that of the complex conjugate attraction,
it is interesting that a matrix with a simple spectrum is forced to
form eigenvalue degeneracies even if the perturbation (or stochastic
process below) is generic. This happens momentarily when any $\lambda_{i}$
and $\overline{\lambda_{i}}$ collide on the real line, whereby standard
perturbation theory and considerations above break down and the colliding
eigenvalues generically form $2\times2$ Jordan blocks \cite[Chapter 1]{wilkinson1965algebraic}.
Following a result due to Lidskii-Vishik-Lyusternik, each of the degenerate
eigenvalues for a small $\delta t$ after gets a correction of order
$\left(\delta t\right)^{1/2}$ with a coefficient determined entirely
by specific entries of $\overset{\centerdot}{M}\left(t\right)$ \cite[See for a review]{moro1997lidskii}.
Consequently and generically these eigenvalues move in different directions.
\begin{figure}
\includegraphics[scale=0.4]{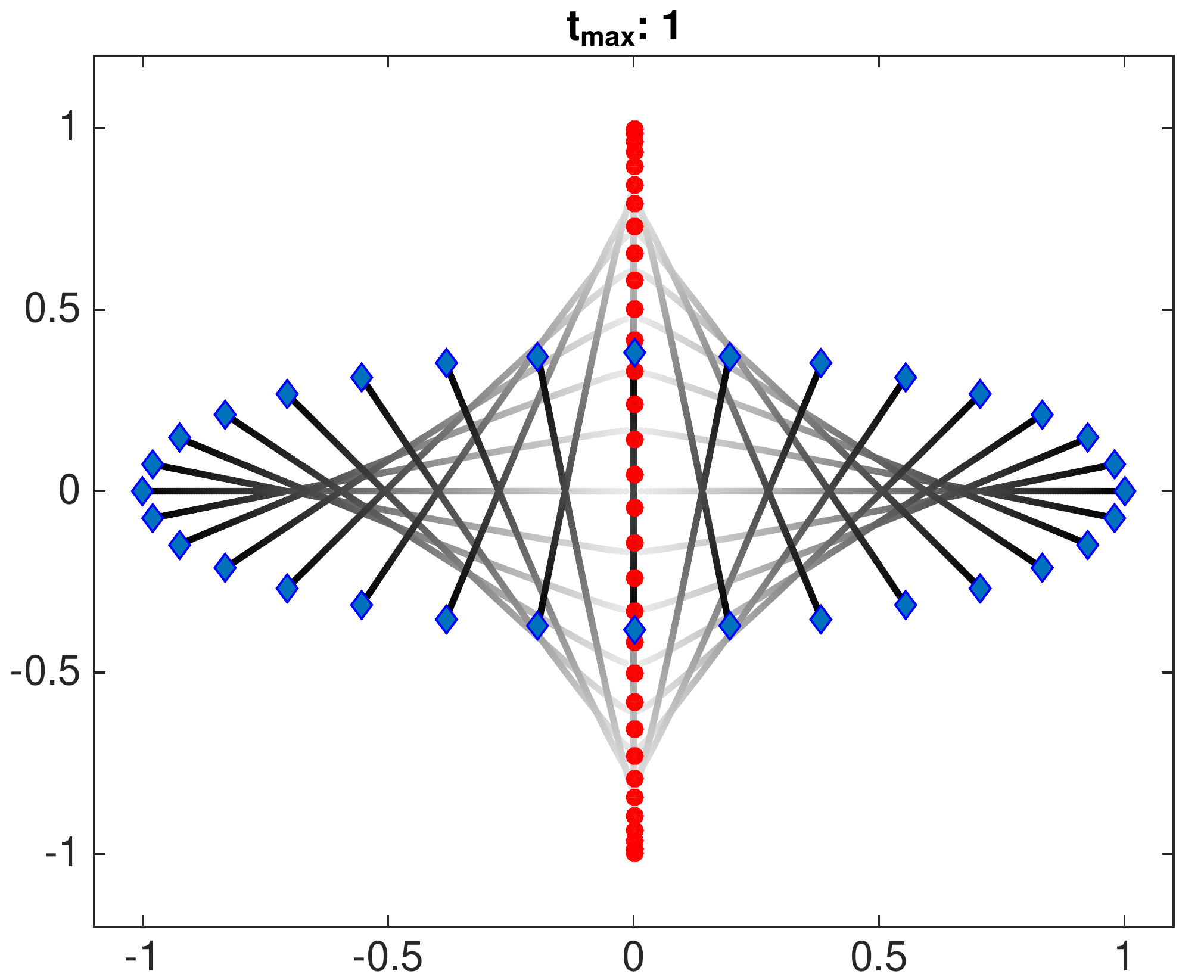}\includegraphics[scale=0.4]{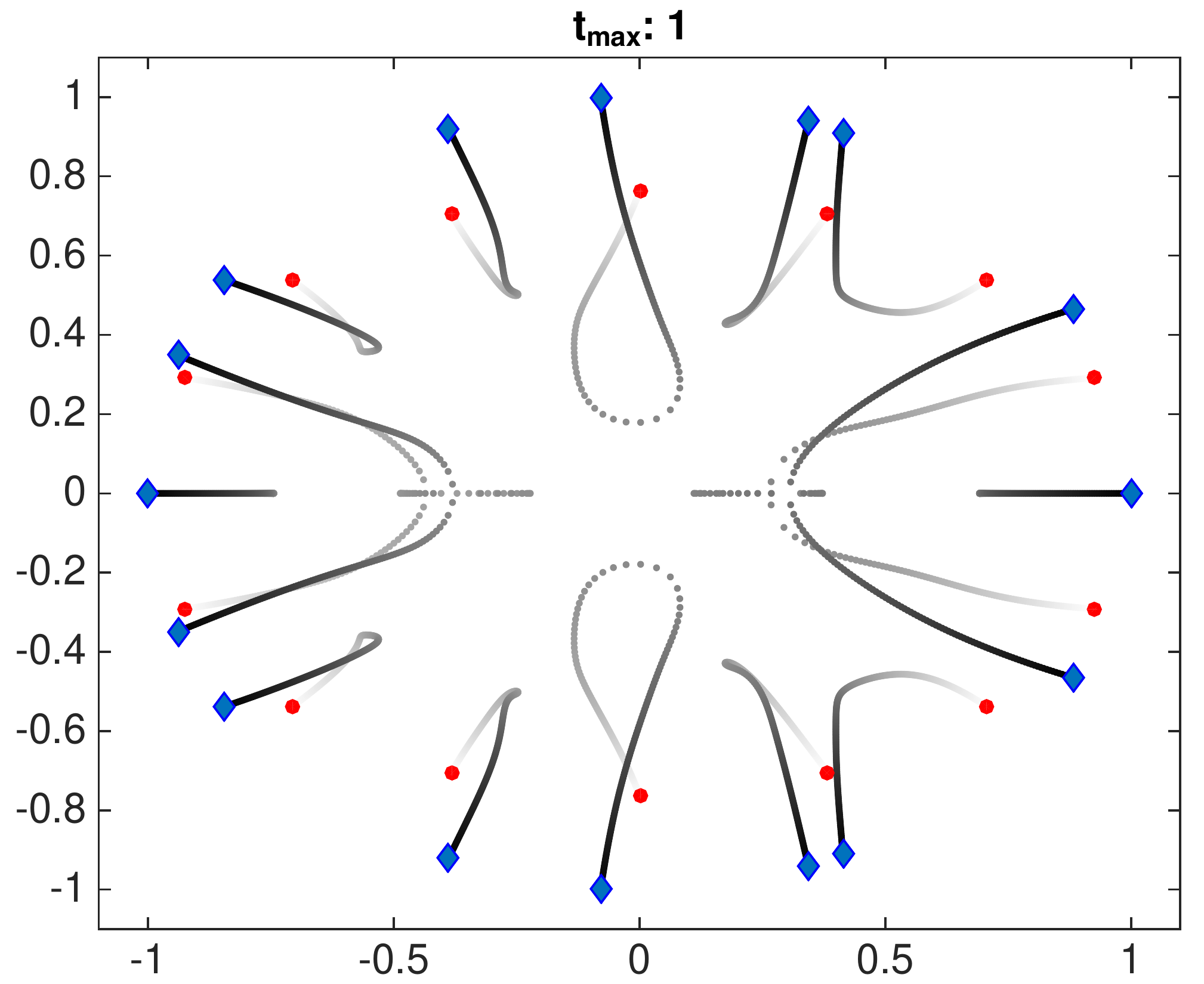}\caption{\label{fig:Example-1}Example 1 (Left): Interpolation between a purely
imaginary matrix and the Hatano-Nelson model. Example 2 (Right): Interpolation
between the Hatano-Nelson model and a random orthogonal matrix.}
\end{figure}

\end{rem}

\section{\label{sec:Random-perturbations}Applications of eigenvalue attraction}

\subsection{Smooth interpolation between fixed matrices}

As an illustration of Theorem \ref{Thm:(Eigenvalue-Attraction)},
let us take $M\left(t\right)=\left(1-t\right)M_{1}+t\mbox{ }M_{2}$,
where $M_{1}$ and $M_{2}$ are two fixed matrices and $t\in\left[0,1\right]$;
the eigenvalues of $M\left(t\right)$ interpolate between the two.
In Figs \ref{fig:Example-1} and \ref{fig:Example-3}, as before,
the eigenvalues of $M_{1}$ are shown in red filled circles and they
darken till their final position, which is the eigenvalues of $M_{2}$
shown in blue diamonds. We take the size of the matrices to be $16\times16$
so the trajectories of eigenvalues can be visually traced more easily.
We normalize all the matrices to have unit $2-$norm. The examples
shown in Figs. \ref{fig:Example-1} and \ref{fig:Example-3} are:

Example 1: $M_{1}=\left[\begin{array}{ccccc}
0 & +1\\
-1 & 0 & +1\\
 & -1 & 0 & \ddots\\
 &  & \ddots & \ddots & +1\\
 &  &  & -1 & 0
\end{array}\right]$ and $M_{2}$ is Hatano-Nelson with $g=-0.4$. Since $M_{1}=-M_{1}^{*}$,
its eigenvalues are purely imaginary. 

Example 2: $M_{1}$ is Hatano Nelson with $g=1$ and $M_{2}$ is a
random orthogonal matrix, hence all the diamonds sit on the unit circle.

Example 3: $M_{1}$ and $M_{2}$ are two real matrices whose entries
are iid drawn from a standard normal distribution. That is, $M_{1}$
and $M_{2}$ are drawn from the Ginibre ensemble. 

Example 4: $M_{1}$ is Hatano Nelson with $g=-0.3$ and $M_{2}$ is
Hatano-Nelson with $g=+0.3$, note that the spectra of $M_{1}$ and
$M_{2}$ coincide and that every one of the complex conjugate pairs
meet on the real line and walk all the way to the other side.

Example 5: $M_{1}$ and $M_{2}$ are both random orthogonal matrices.
\begin{figure}
\begin{centering}
\includegraphics[scale=0.4]{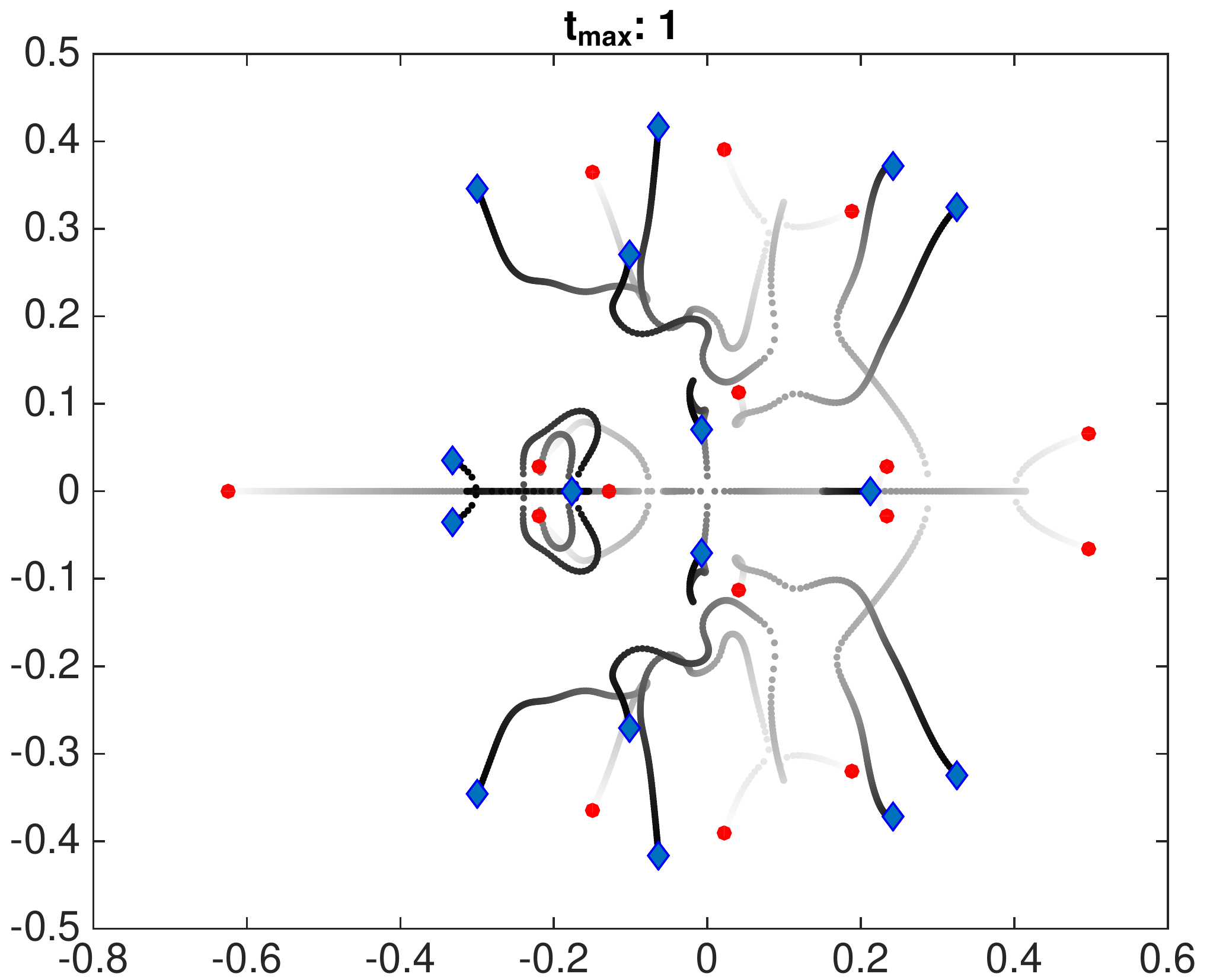}\includegraphics[scale=0.4]{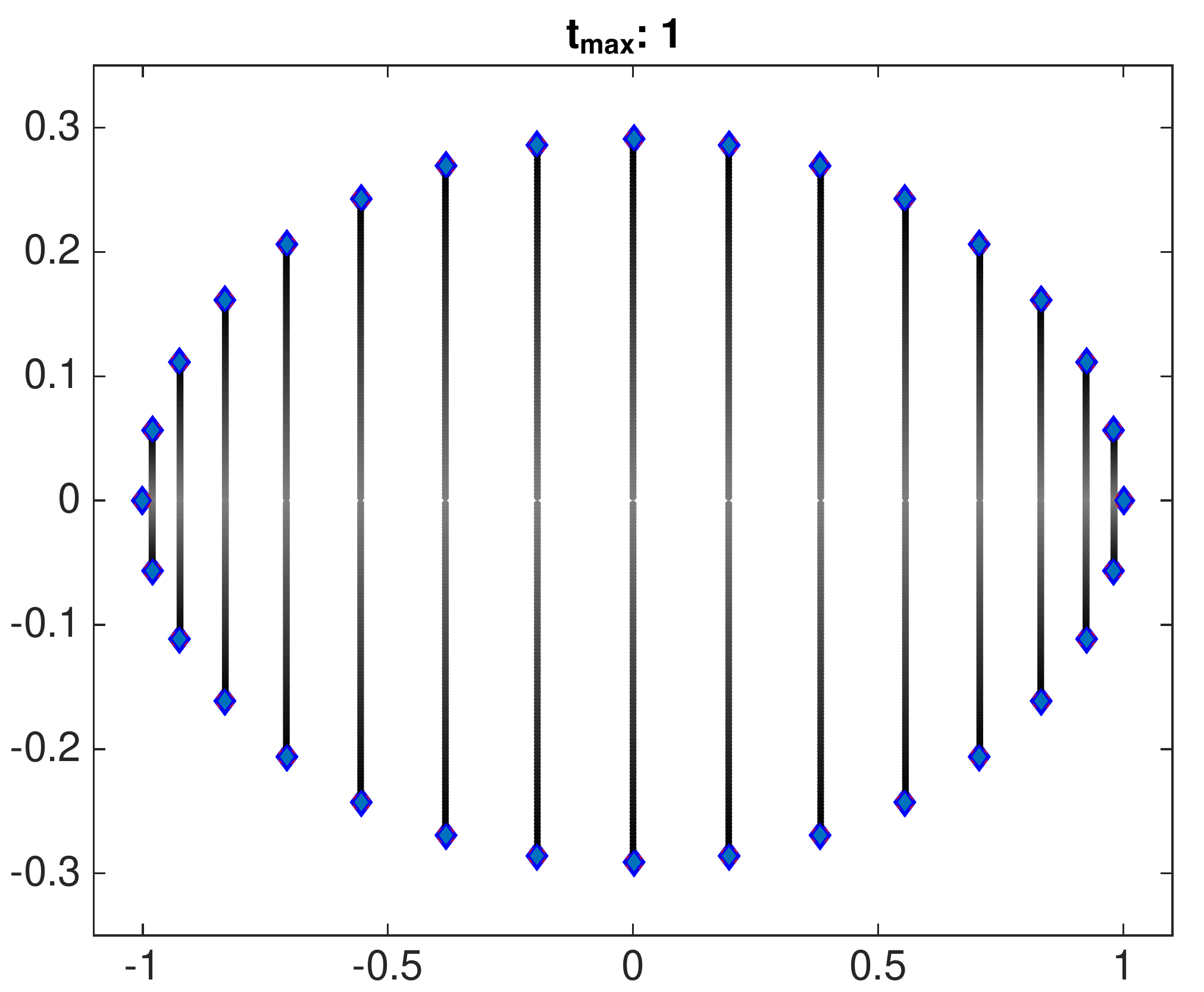}\\

\par\end{centering}

\centering{}\includegraphics[scale=0.4]{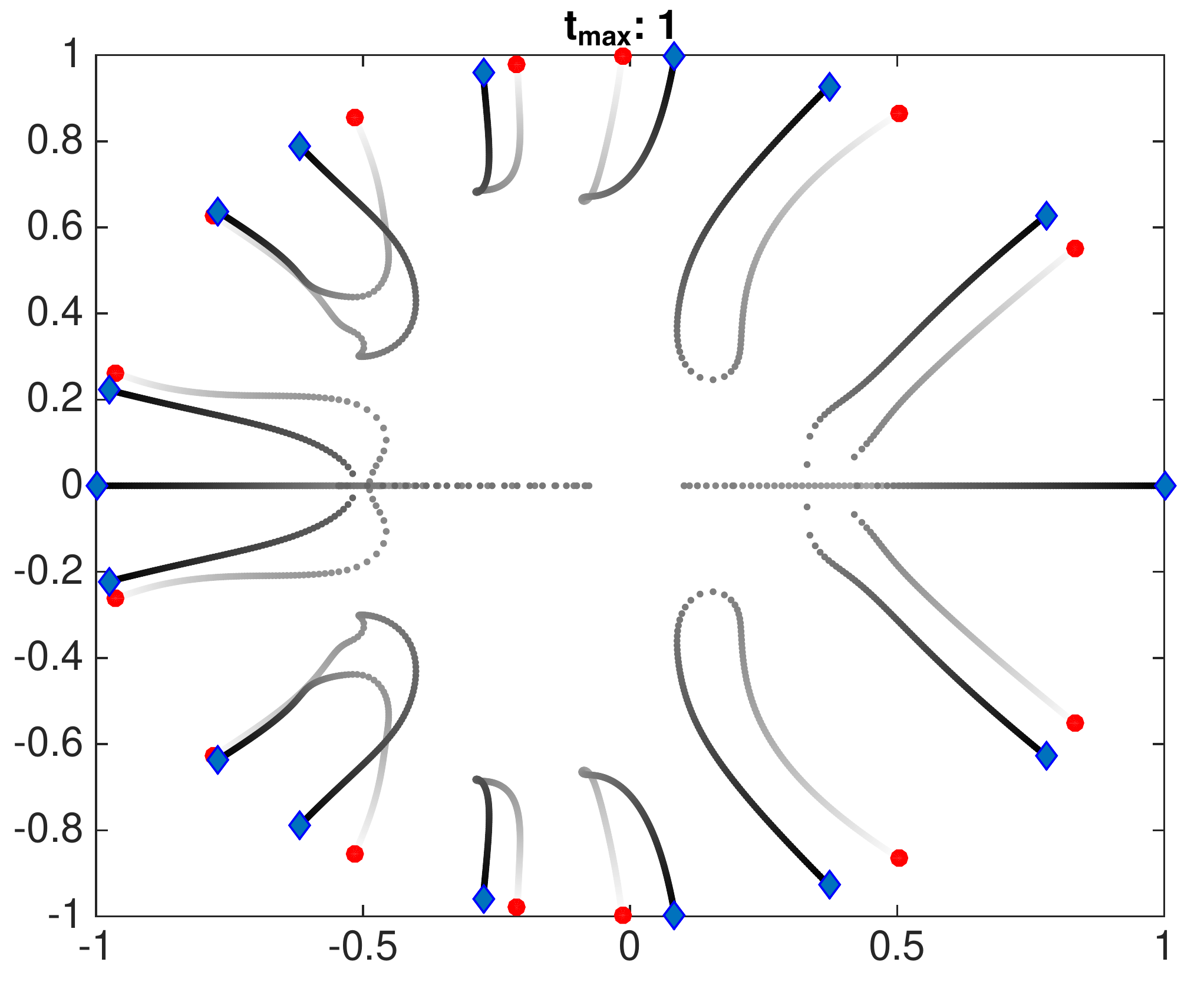}\caption{\label{fig:Example-3}Example 3 (Left): $M_{1}$ and $M_{2}$ are
drawn from the Ginibre ensemble. Example 4 (Right): $M_{1}$ and $M_{2}$
are the Hatano-Nelson model with $g=-0.3$ and $g=+0.3$ respectively.
Example 5 (Bottom): $M_{1}$ and $M_{2}$ are two independent random
orthogonal matrices.}
 
\end{figure}

In all the examples shown here, we expect the complex conjugate attraction
to be more dominant among other forces for pairs closest to the real
line. Note that these complex conjugate pairs first collide and then
they move towards their final positions (shown in blue diamonds).
A nice corollary is that a real and isolated eigenvalue, will remain
real during the time evolution. Since the non-real eigenvalues of
$M\left(t\right)$ come in complex conjugate pairs, a real eigenvalue
cannot move off the real axis. Leaving the real axis can only happen
if two eigenvalues collide on the real line and scatter off into the
complex plane and remain mirror images (i.e., complex conjugate).
For example such collisions are seen in Examples 2 and 3, where in
the latter, one can see a cascade of this. That is a complex conjugate
pair become real and while moving on the real line a new collision
takes the eigenvalues off the real line for a short while, but attraction
pulls them back in again for a second collision.

\subsection{\label{sub:Random-perturbations-of}Random perturbations of a fixed
matrix}

In this section we focus on perturbations of a fixed matrix which
come up often in sciences and engineering, where $M\left(t\right)$
has the form
\begin{equation}
M\left(t\right)=M+\delta t\mbox{ }P\label{eq:LineMatrices}
\end{equation}
where $M$ and $P$ are $n\times n$ real matrices, $M$ is fixed,
$P$ is a random matrix whose entries are independently and identically
distributed (iid) with zero mean and finite fourth moments, and $\delta t$
is a small real parameter such that $\left|\delta t\right|\left\Vert P\right\Vert \ll\left\Vert M\right\Vert $;
i.e., $\delta t\mbox{ }P$ is a perturbation to $M$. The discussion
of this section is important for the formulation and analysis of stochastic
dynamics of the eigenvalues in Section \ref{sec:Stochatic-dynamics}.
Recall that the velocity and acceleration of an eigenvalue are
\begin{eqnarray*}
\overset{\centerdot}{\lambda_{i}}\left(t\right) & = & c_{ii}.\\
\overset{\centerdot\centerdot}{\lambda_{i}}\left(t\right) & = & -i\mbox{ }\frac{\left|c_{i\bar{\mbox{ }i}}\right|^{2}}{\mbox{Im}\left(\lambda_{i}\right)}+2\sum_{j\ne\left\{ i,\bar{i}\right\} }c_{ij}c_{ji}\frac{\hat{\mathbf{r}}_{ij}}{\left|\mathbf{r}_{ij}\right|}
\end{eqnarray*}
where $c_{ij}\equiv\mathbf{u_{i}^{*}}\left(t\right)\mbox{ }P\mbox{ }\mathbf{v_{j}}\left(t\right)$
and $\frac{\hat{\mathbf{r}}_{ij}}{\left|\mathbf{r}_{ij}\right|}=\frac{1}{\lambda_{i}-\lambda_{j}}$
as before.

\subsubsection{First variation}

The expected value, with respect to entries of $P$, of the first
variation is

\[
\mathbb{E}[\overset{\centerdot}{\lambda_{i}}\left(t\right)]=\mathbb{E}[P_{ab}]\mbox{ }\overline{u_{i}}^{a}v_{i}^{b}=0
\]
where for notational convenience from now on we drop the dependence
of the eigenpairs on $t$ and sum over the repeated indices that label
the components of eigenvectors. We comment that $\mathbb{E}[\overset{\centerdot}{\lambda_{i}}\left(t\right)]=0$
even when entries of $P$ are not identically distributed.

The variance is $\mathbb{E}[|\overset{\centerdot}{\lambda_{i}}|^{2}]-|\mathbb{E}[\overset{\centerdot}{\lambda_{i}}]|^{2}$.
From above $|\overset{\centerdot}{\lambda_{i}}|^{2}=P_{ab}P_{cd}\bar{u}_{i}^{a}v_{i}^{b}u_{i}^{c}\bar{v}_{i}^{d}$
and $\mathbb{E}[\overset{\centerdot}{\lambda_{i}}]=0$. In general,
the variance of the first variation denoted by $\sigma_{i,1}^{2}$
is $\mathbb{E}[|\overset{\centerdot}{\lambda_{i}}|^{2}]=\mathbb{E}[p^{2}]\bar{u}_{i}^{a}v_{i}^{b}u_{i}^{a}\bar{v}_{i}^{b}=\mathbb{E}[p^{2}]\left\Vert \mathbf{u}_{i}\right\Vert _{2}^{2}$,
where $p$ denotes any entry of the matrix $P$ whose entries are
iid. Moreover, if $M$ is a normal matrix then $\left\Vert \mathbf{u}_{i}\right\Vert _{2}^{2}=\left\Vert \mathbf{v}_{i}\right\Vert _{2}^{2}=1$.

Next suppose that $P=\mbox{diag}\left(p_{1},p_{2},\dots,p_{n}\right)$
with $\mathbb{E}\left[p_{i}\right]=0$. Then $|\overset{\centerdot}{\lambda_{i}}|^{2}=p_{a}p_{b}\bar{u}_{i}^{a}v_{i}^{a}u_{i}^{b}\bar{v}_{i}^{b}$.
Since $\mathbf{u}_{i}^{*}\mathbf{v}_{i}=1$, for diagonal perturbation
we have $\mathbb{E}[|\overset{\centerdot}{\lambda_{i}}|^{2}]=\mathbb{E}[p^{2}]$. 

The special cases are worth summarizing

\begin{center}
\begin{tabular}{|c|c|c|}
\hline 
$M$ & $P$ & $\sigma_{i,1}^{2}$\tabularnewline
\hline 
General & General & $\mathbb{E}[p^{2}]\left\Vert \mathbf{u}_{i}\right\Vert _{2}^{2}$\tabularnewline
\hline 
General & Diagonal & $\mathbb{E}[p^{2}]$\tabularnewline
\hline 
Normal & General & $\mathbb{E}[p^{2}]$\tabularnewline
\hline 
\end{tabular}
\par\end{center}

Now suppose that $M$ is a \textit{circulant} matrix \cite{GrayReview}
and $P$ is diagonal. In the case of circulant matrices $\mathbf{v}_{j}^{T}=\frac{1}{\sqrt{n}}\left[1,\omega_{j},\omega_{j}^{2},\dots,\omega_{j}^{n-1}\right]$
with $\omega_{j}=\exp\left(2\pi ij/n\right)$. In this case (note
that below no expectation is taken)
\begin{equation}
\overset{\centerdot}{\lambda_{i}}=\overline{v_{i}}^{a}p_{a}v_{i}^{a}=p_{a}\left|v_{i}^{a}\right|^{2}=\frac{1}{n}\sum_{a}p_{a}.\qquad\mbox{if the matrix is Circulant}.\label{eq:lambda_dot_Circulant}
\end{equation}

This expression is the empirical mean of the diagonal entries of $P$,
which for large $n$ tends to zero. Therefore, the dynamics of the
eigenvalues are primarily governed by the second variation.

\subsubsection{Second variation}

When $M\left(t\right)$ is the pencil of matrices given by Eq. \ref{eq:LineMatrices},
$\overset{\centerdot\centerdot}{M}\left(t\right)=0$ , $\overset{\centerdot}{M}\left(t\right)=P$
and

\begin{eqnarray}
\overset{\centerdot\centerdot}{\lambda_{i}} & = & -i\frac{\left|\mathbf{u_{i}^{T}}\mbox{ }P\mbox{ }\mathbf{v_{i}}\right|^{2}}{2\mbox{ Im}\left(\lambda_{i}\right)}+\sum_{j\ne\left\{ i,\bar{i}\right\} }\frac{\left[\mathbf{u_{i}^{*}}\mbox{ }P\mbox{ }\mathbf{v_{j}}\right]\left[\mathbf{u_{j}^{*}}\mbox{ }P\mbox{ }\mathbf{v_{i}}\right]}{\lambda_{i}-\lambda_{j}}.\label{eq:lambda_pp_refined}
\end{eqnarray}

To quantify the dominance of the attractive force, we calculate the
expected value and variance of forces on $\lambda_{i}$ excluding
$\bar{\lambda_{i}}$. The general form of the expected force of the
remaining $n-2$ eigenvalues on $\lambda_{i}$ (excluding its complex
conjugate); i.e. the sum in the foregoing equation or the sum in Eq.
\ref{eq:KEY_Equation} is

\begin{eqnarray}
\mathbb{E}\sum_{j\ne\left\{ i,\bar{i}\right\} }\frac{\left[\mathbf{u_{i}^{*}}\mbox{ }P\mbox{ }\mathbf{v_{j}}\right]\left[\mathbf{u_{j}^{*}}\mbox{ }P\mbox{ }\mathbf{v_{i}}\right]}{\lambda_{i}-\lambda_{j}} & = & \mathbb{E}\sum_{j\ne\left\{ i,\bar{i}\right\} }\frac{\left[\overline{u_{i}}^{a}p_{ab}v_{j}^{b}\right]\left[\overline{u_{j}}^{c}p_{cd}v_{i}^{d}\right]}{\lambda_{i}-\lambda_{j}}=\sum_{j\ne\left\{ i,\bar{i}\right\} }\frac{\mathbb{E}\left[p_{ab}p_{cd}\right]\left[\overline{u_{i}}^{a}v_{j}^{b}\overline{u_{j}}^{c}v_{i}^{d}\right]}{\lambda_{i}-\lambda_{j}}\nonumber \\
 & = & \mathbb{E}[p^{2}]\sum_{j\ne\left\{ i,\bar{i}\right\} }\frac{\mathbf{\left(v_{i}^{T}v_{j}\right)}\mathbf{\left(u_{i}^{*}\overline{u_{j}}\right)}}{\lambda_{i}-\lambda_{j}},\label{eq:Expected_secondVar_General}
\end{eqnarray}
where $p$ denotes any entry of $P$ and $\mathbb{E}\left[p_{ab}p_{cd}\right]=\mathbb{E}[p^{2}]\delta_{ac}\delta_{bd}$.
Although, the attraction is always present, we remark that similar
treatment of the first term in Eq. \ref{eq:lambda_pp_refined} leads
to 
\begin{equation}
\mathbb{E}[\overset{\centerdot\centerdot}{\lambda_{i}}]=-i\frac{\mathbb{E}[p^{2}]\left\Vert \mathbf{u_{i}}\right\Vert _{2}^{2}}{\mbox{Im}\left(\lambda_{i}\right)}+\mathbb{E}[p^{2}]\sum_{j\ne\left\{ i,\bar{i}\right\} }\frac{\mathbf{\left(v_{i}^{T}v_{j}\right)}\mathbf{\left(u_{i}^{*}\overline{u_{j}}\right)}}{\lambda_{i}-\lambda_{j}}\quad.\label{eq:Expected_force}
\end{equation}

We now calculate the variance of the second term on the right hand
side of Eq. \ref{eq:Expected_force}, denoted by $\sigma_{i,2}^{2}$,
in Eq. \ref{eq:lambda_pp_refined}
\begin{figure}
\begin{centering}
\includegraphics[scale=0.35]{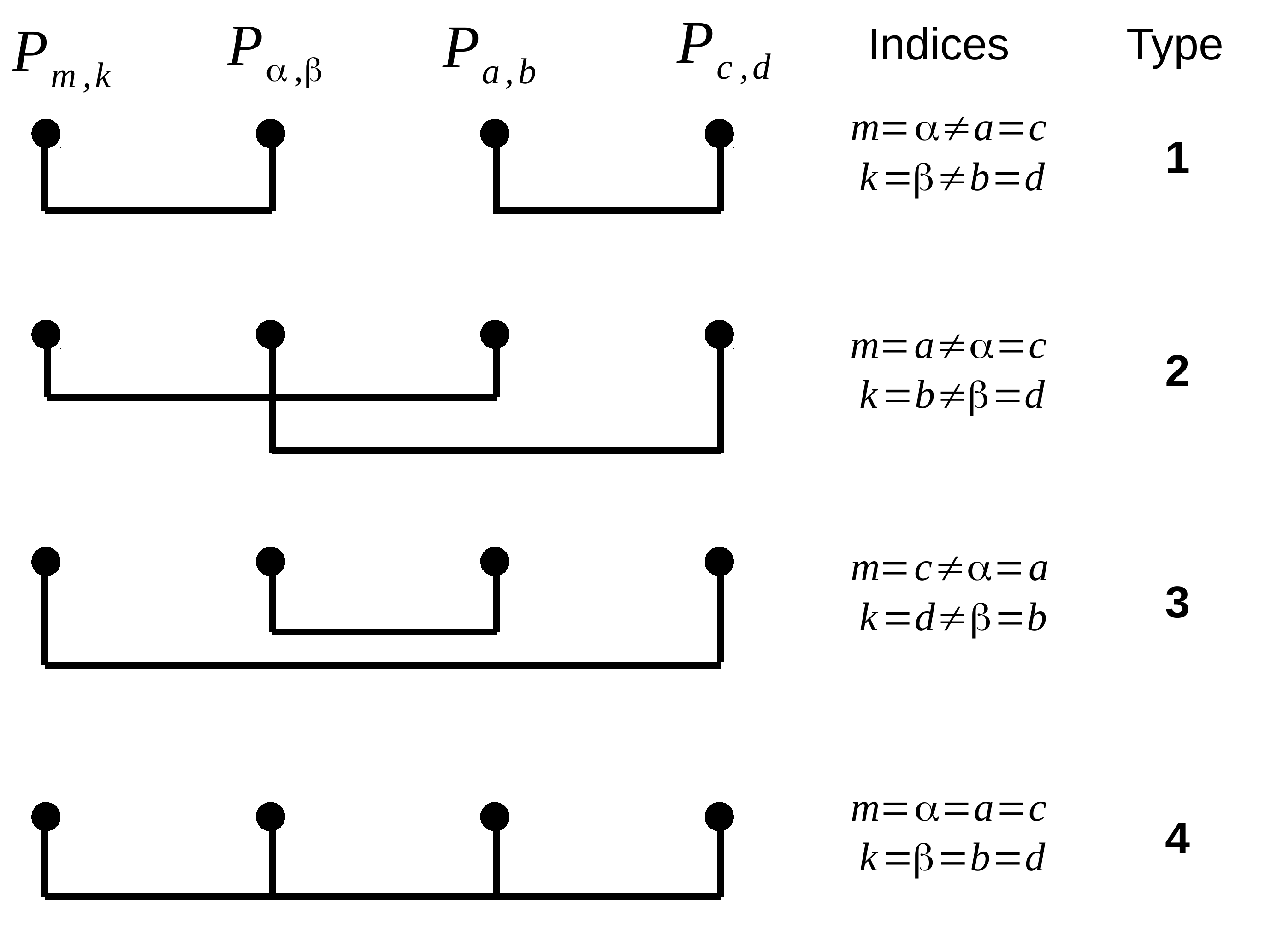}
\par\end{centering}

\caption{\label{fig:Non-zero-in-Expectation}Non-zero contribution in $\mathbb{E}\left\{ P_{mk}P_{\alpha\beta}P_{ab}P_{cd}\mbox{ }\right\} $ }
\end{figure}
 
\begin{eqnarray}
\mathbb{E}[\sigma_{i,2}^{2}] & = & \mathbb{E}\left\{ \left|\sum_{j\ne\left\{ i,\bar{i}\right\} }\frac{\left[\mathbf{u_{i}^{*}}\mbox{ }P\mbox{ }\mathbf{v_{j}}\right]\left[\mathbf{u_{j}^{*}}\mbox{ }P\mbox{ }\mathbf{v_{i}}\right]}{\lambda_{i}-\lambda_{j}}\right|^{2}\right\} -\left|\mathbb{E}\sum_{j\ne\left\{ i,\bar{i}\right\} }\frac{\left[\mathbf{u_{i}^{*}}\mbox{ }P\mbox{ }\mathbf{v_{j}}\right]\left[\mathbf{u_{j}^{*}}\mbox{ }P\mbox{ }\mathbf{v_{i}}\right]}{\lambda_{i}-\lambda_{j}}\right|^{2}\nonumber \\
 & = & \sum_{j,\ell\ne\left\{ i,\bar{i}\right\} }\frac{\mathbb{E}\left[p_{mk}p_{\alpha\beta}p_{ab}p_{cd}\right]\bar{u_{i}}^{m}v_{j}^{k}\bar{u}_{j}^{\alpha}v_{i}^{\beta}u_{i}^{a}\bar{v}_{\ell}^{b}u_{\ell}^{c}\bar{v}_{i}^{d}}{\left(\lambda_{i}-\lambda_{j}\right)\left(\overline{\lambda_{i}}-\overline{\lambda_{\ell}}\right)}-\mathbb{E}^{2}[p^{2}]\left|\sum_{j\ne\left\{ i,\bar{i}\right\} }\frac{\mathbf{\left(v_{i}^{T}v_{j}\right)}\mathbf{\left(u_{i}^{*}\overline{u_{j}}\right)}}{\lambda_{i}-\lambda_{j}}\right|^{2}.\label{eq:kappa_sqr}
\end{eqnarray}
There are potentially four nonzero contribution to the first sum shown
in Fig. \ref{fig:Non-zero-in-Expectation}. We proceed to calculate
them one by one.\\

\begin{flushleft}
\textbf{Type 1 :} The contribution of this type to Eq. \ref{eq:kappa_sqr}
is 
\begin{eqnarray*}
\mathbb{E}^{2}[p^{2}]\sum_{j,\ell\ne\left\{ i,\bar{i}\right\} }\frac{\bar{u_{i}}^{m}v_{j}^{k}\bar{u}_{j}^{m}v_{i}^{k}u_{i}^{a}\bar{v}_{\ell}^{b}u_{\ell}^{a}\bar{v}_{i}^{b}}{\left(\lambda_{i}-\lambda_{j}\right)\left(\overline{\lambda_{i}}-\overline{\lambda_{\ell}}\right)} & = & \mathbb{E}^{2}[p^{2}]\sum_{j,\ell\ne\left\{ i,\bar{i}\right\} }\frac{\left(\mathbf{u}_{i}^{*}\bar{\mathbf{u}}_{j}\right)\left(\mathbf{v}_{j}^{T}\mathbf{v}_{i}\right)\left(\mathbf{u}_{i}^{T}\mathbf{u}_{\ell}\right)\left(\mathbf{v}_{\ell}^{*}\bar{\mathbf{v}}_{i}\right)}{\left(\lambda_{i}-\lambda_{j}\right)\left(\overline{\lambda_{i}}-\overline{\lambda_{\ell}}\right)}\\
 & = & \mathbb{E}^{2}[p^{2}]\left|\sum_{\ell\ne\left\{ i,\bar{i}\right\} }\frac{\left(\mathbf{v}_{\ell}^{T}\mathbf{v}_{i}\right)\left(\mathbf{u}_{i}^{*}\bar{\mathbf{u}}_{\ell}\right)}{\lambda_{i}-\lambda_{\ell}}\right|^{2}\quad.
\end{eqnarray*}
where the sum over $j$ is the complex conjugate of the sum over $\ell$.
Note that Type 1 cancels the second term in Eq. \ref{eq:kappa_sqr}.\\

\par\end{flushleft}

\begin{flushleft}
\textbf{Type 2 :} Since $\left\Vert \mathbf{v}_{i}\right\Vert _{2}^{2}=1$,
the contribution of this type to Eq. \ref{eq:kappa_sqr} is 
\par\end{flushleft}

\[
\mathbb{E}^{2}[p^{2}]\sum_{j,\ell\ne\left\{ i,\bar{i}\right\} }\frac{\bar{u_{i}}^{a}v_{j}^{b}\bar{u}_{j}^{c}v_{i}^{d}u_{i}^{a}\bar{v}_{\ell}^{b}u_{\ell}^{c}\bar{v}_{i}^{d}}{\left(\lambda_{i}-\lambda_{j}\right)\left(\overline{\lambda_{i}}-\overline{\lambda_{\ell}}\right)}=\mathbb{E}^{2}[p^{2}]\left\Vert \mathbf{u}_{i}\right\Vert _{2}^{2}\sum_{j,\ell\ne\left\{ i,\bar{i}\right\} }\frac{\left(\mathbf{v}_{\ell}^{*}\mathbf{v}_{j}\right)\left(\mathbf{u}_{j}^{*}\mathbf{u}_{\ell}\right)}{\left(\lambda_{i}-\lambda_{j}\right)\left(\overline{\lambda_{i}}-\overline{\lambda_{\ell}}\right)}\quad.
\]

\begin{flushleft}
\textbf{Type 3 :} Similarly to Type 1, the contribution of this type
to Eq. \ref{eq:kappa_sqr} is calculated to be
\par\end{flushleft}

\begin{eqnarray*}
\mathbb{E}^{2}[p^{2}]\sum_{j,\ell\ne\left\{ i,\bar{i}\right\} }\frac{\left(\mathbf{u}_{j}^{*}\mathbf{u}_{i}\right)\left(\mathbf{u}_{i}^{*}\mathbf{u}_{\ell}\right)\left(\mathbf{v}_{\ell}^{*}\mathbf{v}_{i}\right)\left(\mathbf{v}_{i}^{*}\mathbf{v}_{j}\right)}{\left(\lambda_{i}-\lambda_{j}\right)\left(\overline{\lambda_{i}}-\overline{\lambda_{\ell}}\right)} & = & \mathbb{E}^{2}[p^{2}]\left|\sum_{\ell\ne\left\{ i,\bar{i}\right\} }\frac{\left(\mathbf{u}_{\ell}^{*}\mathbf{u}_{i}\right)\left(\mathbf{v}_{i}^{*}\mathbf{v}_{\ell}\right)}{\left(\lambda_{i}-\lambda_{\ell}\right)}\right|^{2}\quad.
\end{eqnarray*}

\begin{flushleft}
\textbf{Type 4 :} The contribution of this type to Eq. \ref{eq:kappa_sqr}
is 
\[
\mathbb{E}[p^{4}]\sum_{j,\ell\ne\left\{ i,\bar{i}\right\} }\frac{\bar{u_{i}}^{a}v_{j}^{b}\bar{u}_{j}^{a}v_{i}^{b}u_{i}^{a}\bar{v}_{\ell}^{b}u_{\ell}^{a}\bar{v}_{i}^{b}}{\left(\lambda_{i}-\lambda_{j}\right)\left(\overline{\lambda_{i}}-\overline{\lambda_{\ell}}\right)}=\mathbb{E}[p^{4}]\sum_{j,\ell\ne\left\{ i,\bar{i}\right\} }\frac{\left|u_{i}^{a}\right|^{2}\bar{u}_{j}^{a}u_{\ell}^{a}\left|v_{i}^{b}\right|^{2}v_{j}^{b}\bar{v}_{\ell}^{b}}{\left(\lambda_{i}-\lambda_{j}\right)\left(\overline{\lambda_{i}}-\overline{\lambda_{\ell}}\right)}\quad.
\]
We conclude that Eq. \ref{eq:kappa_sqr} can be written as 
\begin{eqnarray}
\mathbb{E}\left[\sigma_{i,2}^{2}\right] & = & \mathbb{E}^{2}[p^{2}]\left\{ \left\Vert \mathbf{u}_{i}\right\Vert _{2}^{2}\sum_{j,\ell\ne\left\{ i,\bar{i}\right\} }\frac{\left(\mathbf{v}_{\ell}^{*}\mathbf{v}_{j}\right)\left(\mathbf{u}_{j}^{*}\mathbf{u}_{\ell}\right)}{\left(\lambda_{i}-\lambda_{j}\right)\left(\overline{\lambda_{i}}-\overline{\lambda_{\ell}}\right)}+\left|\sum_{\ell\ne\left\{ i,\bar{i}\right\} }\frac{\left(\mathbf{u}_{\ell}^{*}\mathbf{u}_{i}\right)\left(\mathbf{v}_{i}^{*}\mathbf{v}_{\ell}\right)}{\left(\lambda_{i}-\lambda_{\ell}\right)}\right|^{2}\right\} \label{eq:Expected_Kappa_Final}\\
 & + & \mathbb{E}[p^{4}]\sum_{j,\ell\ne\left\{ i,\bar{i}\right\} }\frac{\left|u_{i}^{a}\right|^{2}\bar{u}_{j}^{a}u_{\ell}^{a}\left|v_{i}^{b}\right|^{2}v_{j}^{b}\bar{v}_{\ell}^{b}}{\left(\lambda_{i}-\lambda_{j}\right)\left(\overline{\lambda_{i}}-\overline{\lambda_{\ell}}\right)}\quad.\nonumber 
\end{eqnarray}

\par\end{flushleft}

The derivations of this section are purely theoretical. In practice,
actual estimates of Eq. \ref{eq:Expected_Kappa_Final} would depend
on the particular $M\left(t\right)$; i.e., $P$ and the eigenpairs
of $M$. There are universality results on the eigenvectors of generic
Hermitian matrices such as generalized Wigner matrices \cite{bourgade2013eigenvector}
but not nearly as much is known for generic matrices. Once an estimate
is found, convergence to the expected value of the other forces can
be proved using standard techniques such as Markov's inequality. Below
we discuss special cases that are of theoretical and applied interest.

\subsection{\label{sec:CirculentANDDiagonal} Normal and circulant matrices}

Now let us confine to the case where $M$ in Eq. \ref{eq:LineMatrices}
is a normal matrix , i.e., unitary diagonalizable \cite{Trefethen}. 
\begin{cor}
\label{cor:M_Normal}Suppose $M$ is a normal matrix and $P$ is a
random matrix with entries that are iid and have zero mean. Then the
total expected force on any eigenvalue is only due to its complex
conjugate.\end{cor}
\begin{proof}
If $M$ is normal, its eigenvectors form an orthonormal basis for
$\mathbb{C}^{n}$ and $\mathbf{u}_{i}=\mathbf{v}_{i}$ with $\mathbf{v_{i}^{*}\mathbf{v_{j}}}=\delta_{ij}$.
Eq. \ref{eq:Expected_secondVar_General} gives the expected force
on $\lambda_{i}$ by all other eigenvalues excluding its complex conjugate,
which now reads

\[
\mathbb{E}[p^{2}]\sum_{j\ne\left\{ i,\bar{i}\right\} }\frac{\mathbf{\left|v_{i}^{T}v_{j}\right|^{2}}}{\lambda_{i}-\lambda_{j}}=0
\]
since the sum over $j$ excludes $\bar{i}$, by orthonormality of
the eigenvectors $\mathbf{v_{i}^{T}v_{j}}=0$ . Hence,
\begin{equation}
\mathbb{E}[\overset{\centerdot\centerdot}{\lambda_{i}}]=-i\frac{\mathbb{E}[p^{2}]}{\mbox{Im}\left(\lambda_{i}\right)}\qquad\mbox{if the matrix is Normal .}\label{eq:TotalNormalForce}
\end{equation}
The {\it total} expected force is only due to the complex conjugate 
\end{proof}
In this case the variance also takes on a simpler expression. In Eq.
\ref{eq:Expected_Kappa_Final} the first and third sums vanish because
$\ell\ne\left\{ i,\bar{i}\right\} $ and therefore $\mathbf{v}_{\ell}^{T}\mathbf{v}_{i}=0$
and $\mathbf{v}_{i}^{*}\mathbf{v}_{\ell}=0$. Moreover, in the sum
in Eq. \ref{eq:Expected_Kappa_Final}, the sum over $j$ collapses
because $\mathbf{v}_{\ell}^{*}\mathbf{v}_{j}=\delta_{\ell j}$ and
$\left\Vert \mathbf{u}_{i}\right\Vert _{2}^{2}=1$. 

Type 4 can be simplified as follows

\begin{eqnarray}
\sum_{j\ne\left\{ i,\bar{i}\right\} }\sum_{\ell\ne\left\{ i,\bar{i}\right\} }\frac{\mathbb{E}[p_{ab}^{4}]\mbox{ }\bar{v_{i}}^{a}v_{j}^{b}\bar{v}_{j}^{a}v_{i}^{b}\mbox{ }\bar{v_{i}}^{a}v_{\ell}^{b}\overline{v}_{\ell}^{a}v_{i}^{b}}{\left(\lambda_{i}-\lambda_{j}\right)\overline{\left(\lambda_{i}-\lambda_{\ell}\right)}} & = & \mathbb{E}[p^{4}]\sum_{j\ne\left\{ i,\bar{i}\right\} }\sum_{\ell\ne\left\{ i,\bar{i}\right\} }\frac{\left|v_{i}^{a}\right|^{2}\left|v_{i}^{b}\right|^{2}\bar{v}_{j}^{a}v_{\ell}^{a}\mbox{ }v_{j}^{b}\mbox{ }\overline{v}_{\ell}^{b}}{\left(\lambda_{i}-\lambda_{j}\right)\overline{\left(\lambda_{i}-\lambda_{\ell}\right)}}\label{eq:Type4_1}\\
 & \le & \mathbb{E}[p^{4}]\left|\sum_{\ell\ne\left\{ i,\bar{i}\right\} }\frac{1}{\lambda_{i}-\lambda_{\ell}}\right|^{2}\mbox{ }.\nonumber 
\end{eqnarray}
because $\sum_{ab}\left|v_{i}^{a}\right|^{2}\left|v_{i}^{b}\right|^{2}\bar{v}_{j}^{a}v_{\ell}^{a}\mbox{ }v_{j}^{b}\mbox{ }\overline{v}_{\ell}^{b}=\left|\sum_{a}\left|v_{i}^{a}\right|^{2}\bar{v}_{j}^{a}v_{\ell}^{a}\right|^{2}\mbox{ }\le\sum_{a}\left|\bar{v}_{j}^{a}v_{\ell}^{a}\right|^{2}\le1$. 

Adding the contribution of the four Types in the case where \textit{the
matrix is normal}, we get
\begin{equation}
\mathbb{E}[\sigma_{i,2}^{2}]\le\mathbb{E}[p^{2}]\sum_{\ell\ne\left\{ i,\bar{i}\right\} }\frac{1\mbox{ }}{\left|\lambda_{i}-\lambda_{\ell}\right|^{2}}+\mathbb{E}[p^{4}]\left|\sum_{\ell\ne\left\{ i,\bar{i}\right\} }\frac{1}{\lambda_{i}-\lambda_{\ell}}\right|^{2}.\label{eq:Normal_Variance_final}
\end{equation}
where the inequality comes about from the estimate of the Type 4 terms. 

How are we to visualize the force between any two eigenvalues? In
general one can consider any eigenvalue, $\lambda_{i}$ as a point
in the complex plane with coordinates $\left(\mbox{Re}\lambda_{i},i\mbox{Im}\lambda_{i}\right)$
and the unit vector $\hat{\mathbf{r}}_{ij}$ , given by Eq. \ref{eq:centralForce},
as one proportional to $\left(\mbox{Re}\left(\lambda_{i}-\lambda_{j}\right),-i\mbox{Im}\left(\lambda_{i}-\lambda_{j}\right)\right)$.
Looking at the Eq. \ref{eq:lambda_pp_finalRAMIS-1}, we note that
the force exerted on $\lambda_{i}$ from $\lambda_{j}$ is in the
direction of $c_{ij}c_{ji}\hat{\mathbf{r}_{ij}}$ (recall that $c_{ij}=\mathbf{u_{i}^{*}}\left(t\right)\mbox{ }P\mbox{ }\mathbf{v_{j}}\left(t\right)$).
Even in the case of normal matrices, unlike Hermitian matrices, this
force is not in general central (attractive or repulsive).

In many lattice models in physics the perturbations are taken to be
diagonal, which model the coupling of an external field to the lattice
sites. We now consider another example that is relevant for the application
such as the Hatano-Nelson model. Take $P=\mbox{diag}\left(p_{1},p_{2},\dots,p_{n}\right)$
and let $M$ be a circulant matrix; note that now $P$ does \textit{not}
have to be random. For circulant matrices we get\footnote{We denote the eigenvalue by $\lambda_{k}$ instead of $\lambda_{i}$
above because $i\equiv\sqrt{-1}$ appears more in the following discussion.} 
\begin{eqnarray}
\overset{\centerdot\centerdot}{\lambda_{k}} & = & 2\sum_{j\ne k}\frac{\left|\mathbf{v_{k}^{*}}\mbox{ }P\mbox{ }\mathbf{v_{j}}\right|^{2}}{\lambda_{k}-\lambda_{j}}=-i\frac{\left|v_{k}^{a}p_{a}v_{k}^{a}\right|^{2}}{\mbox{ Im}\left(\lambda_{k}\right)}+2\sum_{j\ne\left\{ k,\bar{k}\right\} }\frac{\left|\bar{v_{k}}^{a}p_{a}v_{j}^{a}\right|^{2}}{\lambda_{k}-\lambda_{j}}\nonumber \\
 & = & -i\frac{\frac{1}{n^{2}}\left|\sum_{a=1}^{n}p_{a}\omega_{2k}^{\left(a-1\right)}\right|^{2}}{\mbox{ Im}\left(\lambda_{k}\right)}+2\sum_{j\ne\left\{ k,\bar{k}\right\} }\frac{\frac{1}{n^{2}}\left|\sum_{a=1}^{n}p_{a}\omega_{j-k}^{\left(a-1\right)}\right|^{2}}{\lambda_{k}-\lambda_{j}},\label{eq:FFT_p}
\end{eqnarray}
since $P^{T}=P$. The quantity $\frac{1}{n}\sum_{a=0}^{n-1}p_{a}\omega_{j-k}^{a-1}$
is the discrete Fourier transformation of the diagonal elements of
$P$. 

If we assume that the nonzero entries of $P$ are drawn from a standard
normal distribution, Eq. \ref{eq:FFT_p} is simply the discrete Fourier
transform of white noise and $\frac{1}{n^{2}}\left|\sum_{a=0}^{n-1}p_{a}\omega_{j-k}^{a-1}\right|^{2}$
is the power intensity which is a real positive constant independent
of the frequency $j-k$ \cite{stoica1997introduction}, denoted here
by $\kappa^{2}$ . So the total force is
\begin{equation}
\overset{\centerdot\centerdot}{\lambda_{k}}=\kappa^{2}\left\{ \frac{-i}{\mbox{ Im}\left(\lambda_{k}\right)}+\sum_{j\ne\left\{ k,\bar{k}\right\} }\frac{2}{\lambda_{k}-\lambda_{j}}\right\} \quad.\label{eq:lambda_dd_Circulant}
\end{equation}

\begin{rem}
The foregoing equation and the following arguments can apply to the
more general cases. For example, in Eq. \ref{eq:lambda_pp_finalRAMIS-1}
if $P$ is taken to be diagonal to ensure reality of the numerator
and that the numerators have similar magnitudes the following analysis
applies. \end{rem}
\begin{prop}
Let $M\left(t\right)=M+\delta t\mbox{ }P$, where $M$ is circulant
and $P$ is diagonal with iid entries with zero mean. Any two distinct
eigenvalues of $M\left(t\right)$ have an expected attractive (repulsive)
central force law of interaction between them if and only if their
real (imaginary) parts are equal. \end{prop}
\begin{proof}
The force between the eigenvalues is given by Eq. \ref{eq:lambda_dd_Circulant}.
By def. \ref{(Attraction-and-Repulsion)}, two eigenvalues $\lambda_{k}$
and $\lambda_{j}$ have a central force between them if for some $\mu\in\mathbb{R}$,
\[
\frac{1}{\lambda_{k}-\lambda_{j}}=\mu\left(\lambda_{k}-\lambda_{j}\right).
\]
Therefore, we would need the unit vector $\hat{\mathbf{r}}_{kj}=\frac{\overline{\lambda_{k}}-\overline{\lambda_{j}}}{\left|\lambda_{k}-\lambda_{j}\right|}$
to be equal to the unit vector $\pm\frac{\lambda_{k}-\lambda_{j}}{\left|\lambda_{k}-\lambda_{j}\right|}$.
Let $\lambda_{k}=a+ib$ and $\lambda_{j}=c+id$ and by def. \ref{(Attraction-and-Repulsion)},
the force is central if $\left(a-c\right)-i\left(b-d\right)=\pm\left[\left(a-c\right)+i\left(b-d\right)\right]$.
If we take the negative sign then $a=c$ and $b-d$ is free. If we
take the positive sign $b=d$ and $a-c$ is free. Therefore, eigenvalues
of $M\left(t\right)$ that have the same real (imaginary) part attract
(repel). Conversely, if $\lambda_{k}$ and $\lambda_{j}$ have the
same real parts then $\overline{\lambda_{k}}-\overline{\lambda_{j}}=-i\left(b-d\right)$.
If $b>d$ then the force that $\lambda_{k}$ experiences is downwards
and if $d>b$, it experiences an upward force. The repulsion for eigenvalues
with equal imaginary parts is proved using a similar argument. In
particular, eigenvalues of a Hermitian matrix repel and non-real complex
conjugate eigenvalues attract. 
\end{proof}
As a concrete example, let $b$ be small and let $x\equiv a-c$ and
$y\equiv d-b$. The contribution of the terms from $j$ and $\bar{j}$
to the sum in Eq. \ref{eq:lambda_dd_Circulant} is proportional to
\begin{eqnarray}
\mbox{Re}\left\{ \frac{1}{\lambda_{k}-\lambda_{j}}+\frac{1}{\lambda_{k}-\overline{\lambda_{j}}}\right\}  & = & x\left\{ \frac{1}{x^{2}+y^{2}}+\frac{1}{x^{2}+y^{2}\left(1+2b/y\right)^{2}}\right\} \label{eq:Real_part}\\
\mbox{Im}\left\{ \frac{1}{\lambda_{k}-\lambda_{j}}+\frac{1}{\lambda_{k}-\overline{\lambda_{j}}}\right\}  & = & y\left\{ \frac{1}{x^{2}+y^{2}}-\frac{1+2b/y}{x^{2}+y^{2}\left(1+2b/y\right)^{2}}\right\} \label{eq:Imag_part}
\end{eqnarray}

Evidently any eigenvalue $\lambda_{j}$ \textit{repels} $\lambda_{k}$
\textit{along the real axis} since the sign of the force follows the
sign of $x$ (See the right figure in Fig. \ref{fig:Dynamics_HN_different_g}).
Let us suppose that $2b/y\equiv\epsilon\ll1$ (e.g., force of bulk
eigenvalues on a $\lambda_{k}$ near the real line), then Eq. \ref{eq:Imag_part}
becomes
\[
\mbox{Im}\left\{ \frac{1}{\lambda_{k}-\lambda_{j}}+\frac{1}{\lambda_{k}-\overline{\lambda_{j}}}\right\} =-\frac{y\epsilon}{x^{2}+y^{2}}\left\{ 1-\frac{2y^{2}}{x^{2}+y^{2}}\right\} 
\]
to interpret this equation for now suppose that the imaginary parts
of $\lambda_{k}$ and $\lambda_{j}$ are positive, then $\epsilon,y>0$
and the effect of the imaginary part of the pair $\lambda_{j},\overline{\lambda_{j}}$
on $\lambda_{k}$ is to compress it towards the real axis as long
as $x^{2}>y^{2}$. Moreover the imaginary part of the force, unlike
the real part, is $\mathcal{O}\left(\epsilon\right)$. An entirely
similar argument applies to the case where the imaginary parts of
$\lambda_{k}$ and $\lambda_{j}$ are negative. Hence as long as the
difference of the real parts is larger than that of the imaginary
parts, \textit{there is a compressive push towards the real line from
any complex conjugate pair of eigenvalues on $\lambda_{k}$ with a
small net magnitude}. In many examples of circulant matrices the pair
of eigenvalues closest to the real line appear on the ``edges''
of the spectrum and the assumptions made above are applicable (for
example see the Hatano-Nelson model below and Fig. \ref{fig:Dynamics_HN_different_g}).
We conclude that in such cases, the eigenvalues closest to the real
line become real as a result of the compression just discussed.

\subsection{Hatano-Nelson model}

Let $H\left(t\right)=H+\delta t\mbox{ }P$, where $H$ is a circulant
matrix given by Eq. \ref{eq:Hatano-Nelson} and $P$ is a real diagonal
matrix with iid entries with zero mean. The eigenpairs of $H$ are
\begin{figure}
\centering{}\includegraphics[scale=0.4]{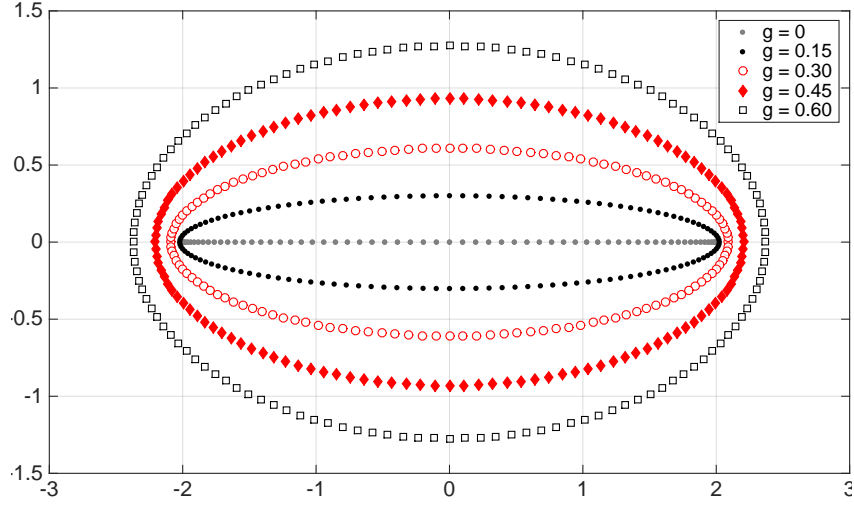}\caption{\label{fig:HatanoNelsonVSg}The spectrum of Hatano-Nelson model (Eq.
\ref{eq:Hatano-Nelson}) as a function of $g$. When $g$ is small,
the spectrum is quite flat for the eigenvalues whose real parts are
small.}
\end{figure}

\begin{eqnarray*}
\lambda_{k} & = & 2\left\{ \cosh g\mbox{ }\cos\left(2\pi k/n\right)+i\mbox{ }\sinh g\mbox{ }\sin\left(2\pi k/n\right)\right\} \\
\mathbf{v}_{k}^{T} & = & \frac{1}{\sqrt{n}}\left[1,\omega_{k},\omega_{k}^{2},\dots,\omega_{k}^{n-1}\right],
\end{eqnarray*}
where $\omega_{k}=\exp\left(2\pi ik/n\right)$. See Fig. \ref{fig:HatanoNelsonVSg}
for the effect of the asymmetry parameter $g$ on the spectrum of
$H$. 

Previously we showed that $\overset{\centerdot}{\lambda_{i}}\approx0$
for sufficiently large $n$ (Eq. \ref{eq:lambda_dot_Circulant}).
Therefore the dynamics are governed by the acceleration $\overset{\centerdot\centerdot}{\lambda_{k}}$.

Even though $H$ is normal, $H\left(t\right)$ is not. However, we
observe that for small $g$, the eigenvalues rush to the real line
almost undeflected (see the left figure in Fig. \ref{fig:Dynamics_HN_different_g}).
Why are the eigenvalues of the Hatano-Nelson model dominated by an
attraction towards the real line for relatively large $t$ when $g$
is small?

When $0<g\ll1$, $\lambda_{k}\approx2\left\{ \cos\left(2\pi k/n\right)+ig\sin\left(2\pi k/n\right)\right\} $
and Eq. \ref{eq:lambda_dd_Circulant} becomes 
\[
\overset{\centerdot\centerdot}{\lambda_{k}}\approx\kappa^{2}\left\{ \frac{-i}{2\mbox{ Im}\left(\lambda_{k}\right)}+\sum_{j\ne\left\{ k,\bar{k}\right\} }\frac{1}{2\left[\cos\left(2\pi k/n\right)-\cos\left(2\pi j/n\right)\right]}-i\frac{g\mbox{ }\left[\sin\left(2\pi k/n\right)-\sin\left(2\pi j/n\right)\right]}{2\left[\cos\left(2\pi k/n\right)-\cos\left(2\pi j/n\right)\right]^{2}}\right\} 
\]
In the foregoing sum we calculate the contribution of the sum of two
terms given by $j$ and $\bar{j}$. Ignoring $\mathcal{O}\left(g^{2}\right)$,
we get
\begin{figure}
\includegraphics[scale=0.38]{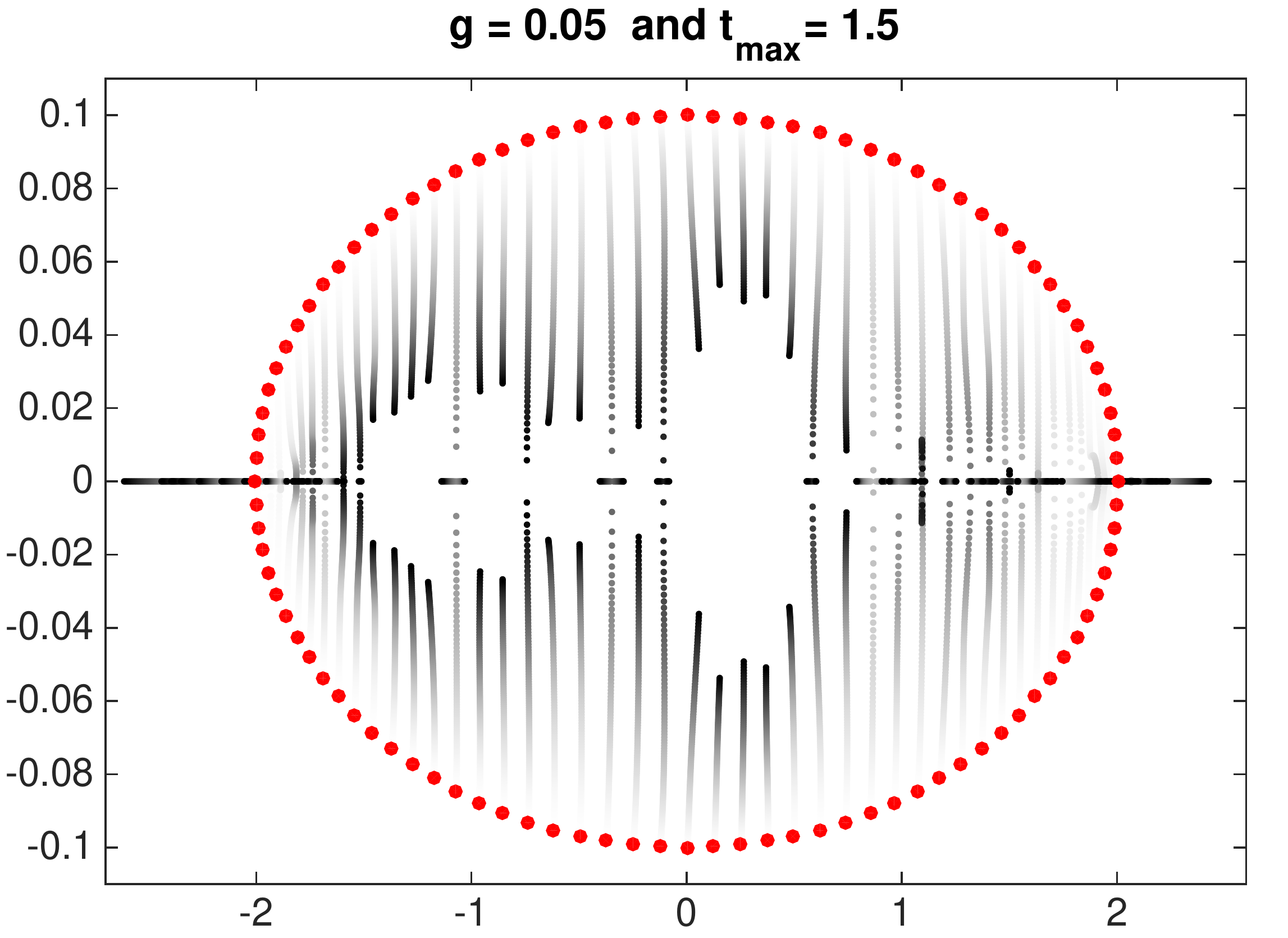}\includegraphics[scale=0.38]{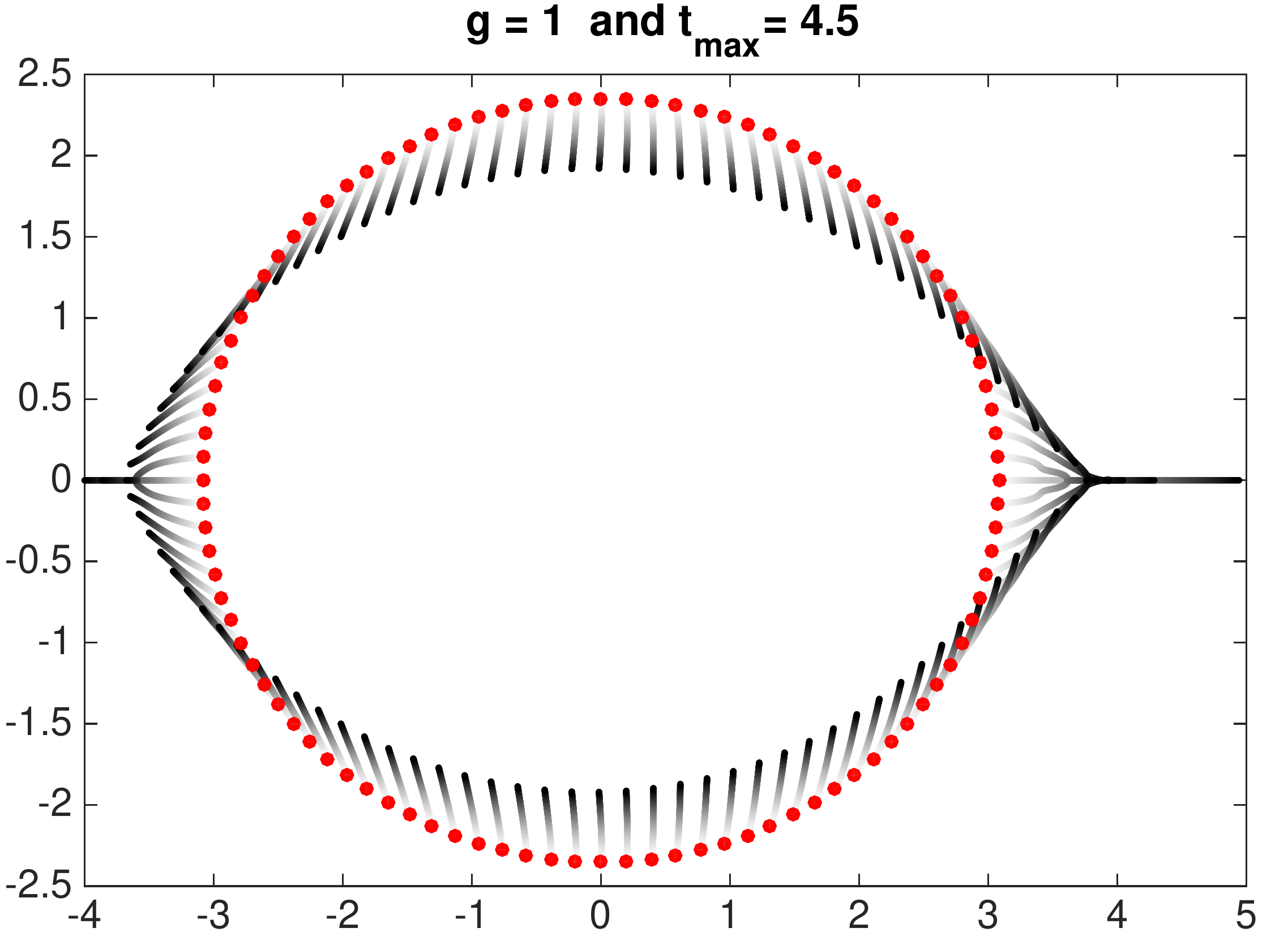}

\raggedright{}\caption{\label{fig:Dynamics_HN_different_g}Dynamics of eigenvalues for small
(left) $g$ and larger (right) $g$. Also note the difference in range
of the imaginary axis and the repulsion along the real axis in agreement
with the discussion of the previous section.}
\end{figure}

\[
\frac{1}{\left[\cos\left(2\pi k/n\right)-\cos\left(2\pi j/n\right)\right]}-i\frac{g\mbox{ }\sin\left(2\pi k/n\right)}{\left[\cos\left(2\pi k/n\right)-\cos\left(2\pi j/n\right)\right]^{2}}\quad.
\]
Since $\mbox{Im}\left(\lambda_{k}\right)\approx g\mbox{ }\sin\left(2\pi k/n\right)$,
the total force is
\begin{equation}
\frac{\overset{\centerdot\centerdot}{\lambda_{k}}}{\kappa^{2}}\approx\sum_{\begin{array}{c}
j\ne\left\{ k,\bar{k}\right\} \\
\mbox{Im}\left(\lambda_{j}\right)>0
\end{array}}\left\{ \frac{1}{\left[\cos\left(2\pi k/n\right)-\cos\left(2\pi j/n\right)\right]}\right\} -i\left\{ \frac{1}{2\mbox{ Im}\left(\lambda_{k}\right)}+\sum_{\begin{array}{c}
j\ne\left\{ k,\bar{k}\right\} \\
\mbox{Im}\left(\lambda_{j}\right)>0
\end{array}}\frac{\mbox{ Im}\left(\lambda_{k}\right)}{\left[\cos\left(2\pi k/n\right)-\cos\left(2\pi j/n\right)\right]^{2}}\right\} .\label{eq:HatanoNelson_acceleration}
\end{equation}
We conclude that there is a net pull of \textit{any} eigenvalue towards
the real line as $\mbox{ Im}\left(\lambda_{k}\right)>0$ gives a force
in the $-i$ direction and $\mbox{ Im}\left(\lambda_{k}\right)<0$
gives a force in $+i$ direction. Moreover, the real part of the force
of any eigenvalue on $\lambda_{k}$ is repulsive as can easily be
seen by analyzing the sign of the denominator of the first sum in
the foregoing equation. The eigenvalues rush towards the real line
and flow outwards away from the origin (See Fig. \ref{fig:Dynamics_HN_different_g}).

Comment: Take $k=n/4$, where $\lambda_{n/4}=i\mbox{ }2g$, it is
easy to see that $\mbox{Re}\left(\overset{\centerdot\centerdot}{\lambda}_{\frac{n}{4}}\right)=0$
by rewriting the real part of the sum (Eq. \ref{eq:HatanoNelson_acceleration})
as 
\[
\mbox{Re}\left(\frac{\overset{\centerdot\centerdot}{\lambda_{n/4}}}{\kappa^{2}}\right)=\sum_{\begin{array}{c}
j\ne\left\{ k,\bar{k}\right\} \\
\mbox{Im}\left(\lambda_{j}\right)>0
\end{array}}\frac{-1}{\cos\left(2\pi j/n\right)}=\sum_{\ell=1}^{n/4-1}\frac{-1}{\cos\left(2\pi\left(n/4+\ell\right)/n\right)}-\frac{1}{\cos\left(2\pi\left(n/4-\ell\right)/n\right)}=0\quad.
\]
So the net force on this eigenvalue is purely imaginary. In this limit
(i.e., $g\ll1$), the spectrum is an ellipse with semi-minor axis
$i\mbox{ }2g$ and semi-major axis $2$ (Figs. \ref{fig:HatanoNelsonVSg}
and \ref{fig:Dynamics_HN_different_g}). The spectrum has small imaginary
variations (i.e., is quite flat) for the eigenvalues with small real
parts. Therefore the force on these eigenvalues is approximately purely
imaginary explaining their almost undeflected rush towards the real
line (Fig. \ref{fig:Dynamics_HN_different_g}).

Although $H$ is a circulant matrix and hence normal, $H+\delta t\mbox{ }P$
is not, but the deviation from normality is mild for small $g$ as
the following shows. The following analysis combined with Corollary
\ref{cor:M_Normal}, implies that the expected force on any eigenvalue
in this model is only due to its complex conjugate even when $\delta t$
is not small. We quantify the degree of non-normality by looking at
$\left[H\left(t\right),H^{T}\left(t\right)\right]=\delta t\left\{ \left[H,P\right]-\left[H^{T},P\right]\right\} $,
which is calculated to be 

\[
\left[H\left(t\right),H^{T}\left(t\right)\right]=2\mbox{ }\delta t\mbox{ }\sinh g\left[\begin{array}{cccccc}
0 & p_{1}-p_{2} & 0 & \cdots & 0 & p_{n}-p_{1}\\
p_{1}-p_{2} & 0 & p_{2}-p_{3} &  &  & 0\\
0 & p_{2}-p_{3} & \ddots & \ddots &  & \vdots\\
\vdots &  & \ddots & 0 & p_{n-2}-p_{n-1} & 0\\
0 &  &  & p_{n-2}-p_{n-1} & 0 & p_{n-1}-p_{n}\\
p_{n}-p_{1} & 0 & \cdots & 0 & p_{n-1}-p_{n} & 0
\end{array}\right]\quad.
\]

Let us take each $p_{i}$ to be randomly distributed with mean $\mu$,
then the commutator matrix in an expectation sense (with respect to
entries) is the zero matrix. For example, if we take $p_{i}$ to be
normally distributed with mean $\mu$ and variance $\sigma^{2}$ (i.e.,
${\cal N}\left(\mu,\sigma^{2}\right)$), then $X_{i}\equiv\left(p_{i+1}-p_{i}\right)\sim{\cal N}\left(0,2\sigma^{2}\right)$
and each entry has mean zero and variance $4t^{2}\sinh^{2}g$. Therefore
fluctuations are small so long as $g$ is small where the variance
is approximately $4t^{2}g^{2}$. 
\begin{rem*}
We intentionally did not use the Frobenius norm to quantify normality
as it would not appreciate the entries of the commutator being zero
in an expected sense despite $\mathbb{E}\left(p\right)\ne0$. Take
as a measure of non-normality the ratio $\left\Vert \left[H\left(t\right),H^{T}\left(t\right)\right]\right\Vert _{F}/\left\Vert H\right\Vert _{F}^{2}$.
But $\left\Vert H\right\Vert _{F}^{2}=2n\mbox{ }\cosh2g$ and $\left\Vert \left[H\left(t\right),H^{T}\left(t\right)\right]\right\Vert _{F}^{2}=2t\mbox{ }\sinh g\left\{ 2\sum_{i}\left(p_{i+1}-p_{i}\right)^{2}\right\} $,
where $p_{n+1}=p_{1}$. and $\sum_{i}X_{i}^{2}\sim2\sigma^{2}\chi_{n-1}^{2}$,
where $\chi_{n-1}^{2}$ denotes a chi-square distribution with $n-1$
degrees of freedom and we have

\begin{align*}
\frac{\left\Vert \left[H\left(t\right),H^{T}\left(t\right)\right]\right\Vert _{F}}{\left\Vert H\right\Vert _{F}^{2}} & \sim\frac{\sigma\sqrt{2t\mbox{ }\sinh g}}{n\mbox{ }\cosh2g}\chi_{n-1}\quad.
\end{align*}
The mean of the $\chi-$distribution is $\sqrt{2}\Gamma\left[\left(n+1\right)/2\right]/\Gamma\left(n/2\right)$,
which for large $n$ is approximately $\sqrt{2n}$ and we have
\[
\mathbb{E}\left[\frac{\left\Vert \left[H\left(t\right),H^{T}\left(t\right)\right]\right\Vert _{F}}{\left\Vert H\right\Vert _{F}^{2}}\right]=\frac{2\sigma\sqrt{t\mbox{ }\sinh g}}{\sqrt{n}\mbox{ }\cosh2g}\quad.
\]

Comment: Consider $g\gg1$, where $H$ becomes proportional to a permutation
matrix. In this limit the foregoing expectation is zero. Second take
$g\ll1$, in which case the expectation is approximately $2\sigma\sqrt{\frac{tg}{n}}$;
in this limit the matrix is approximately symmetric.
\end{rem*}

\subsection{\label{sub:Further-illustrations} Further illustrations}

Suppose we perturb a $64\times64$ real orthogonal matrix, $M$, with
$\delta t\mbox{ }P$, where $P$ is a real random matrix and $0\le\delta t\le t_{max}$
with Gaussian entries. In Fig. \ref{fig:Orthogonal} (left) we show
the motion of the eigenvalues and indicate $t_{max}$ for each and
on the right we take $P$ to be a random $\pm1$ matrix 
\begin{figure}
\centering{}\includegraphics[scale=0.35]{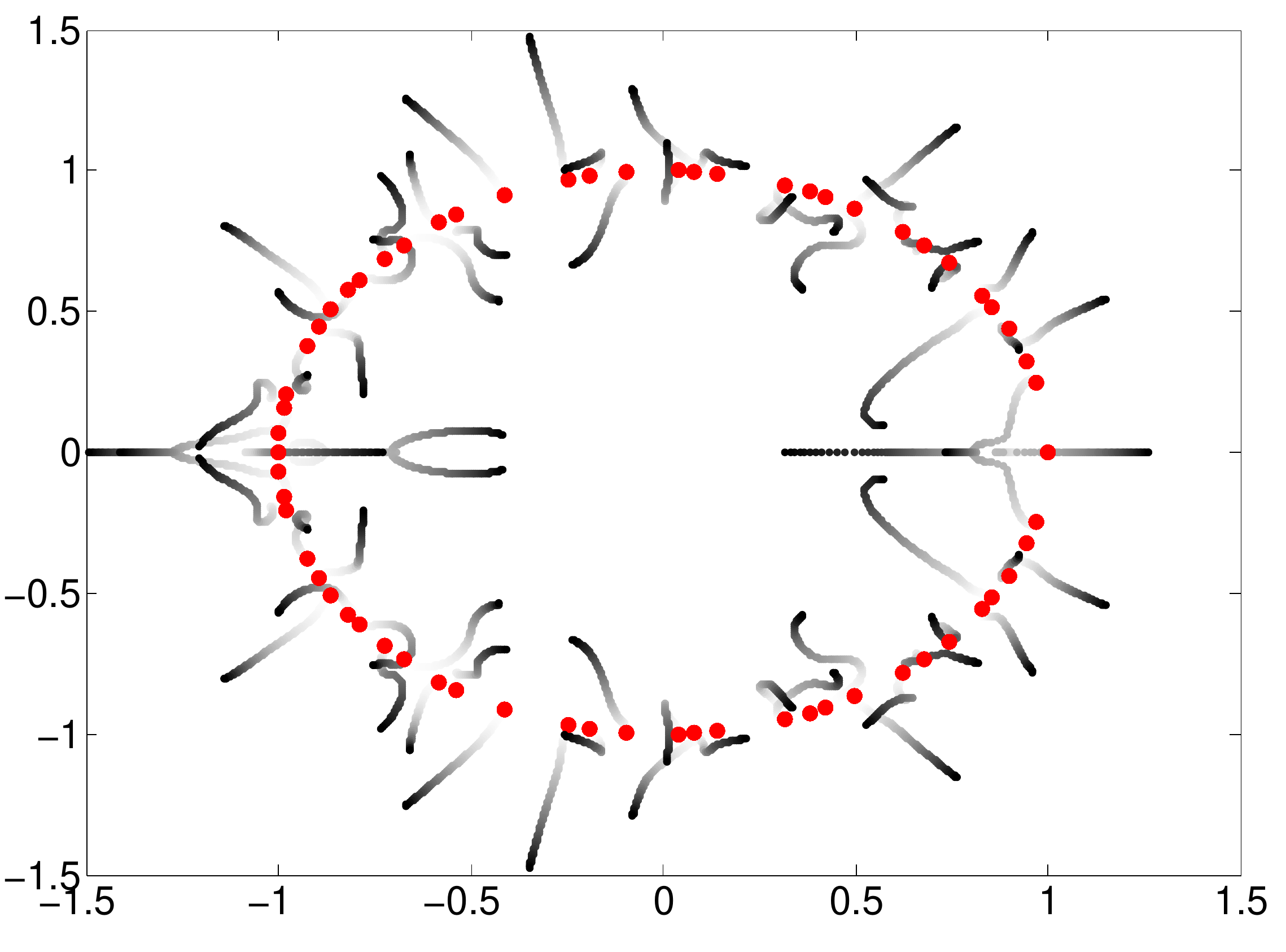}\includegraphics[scale=0.31]{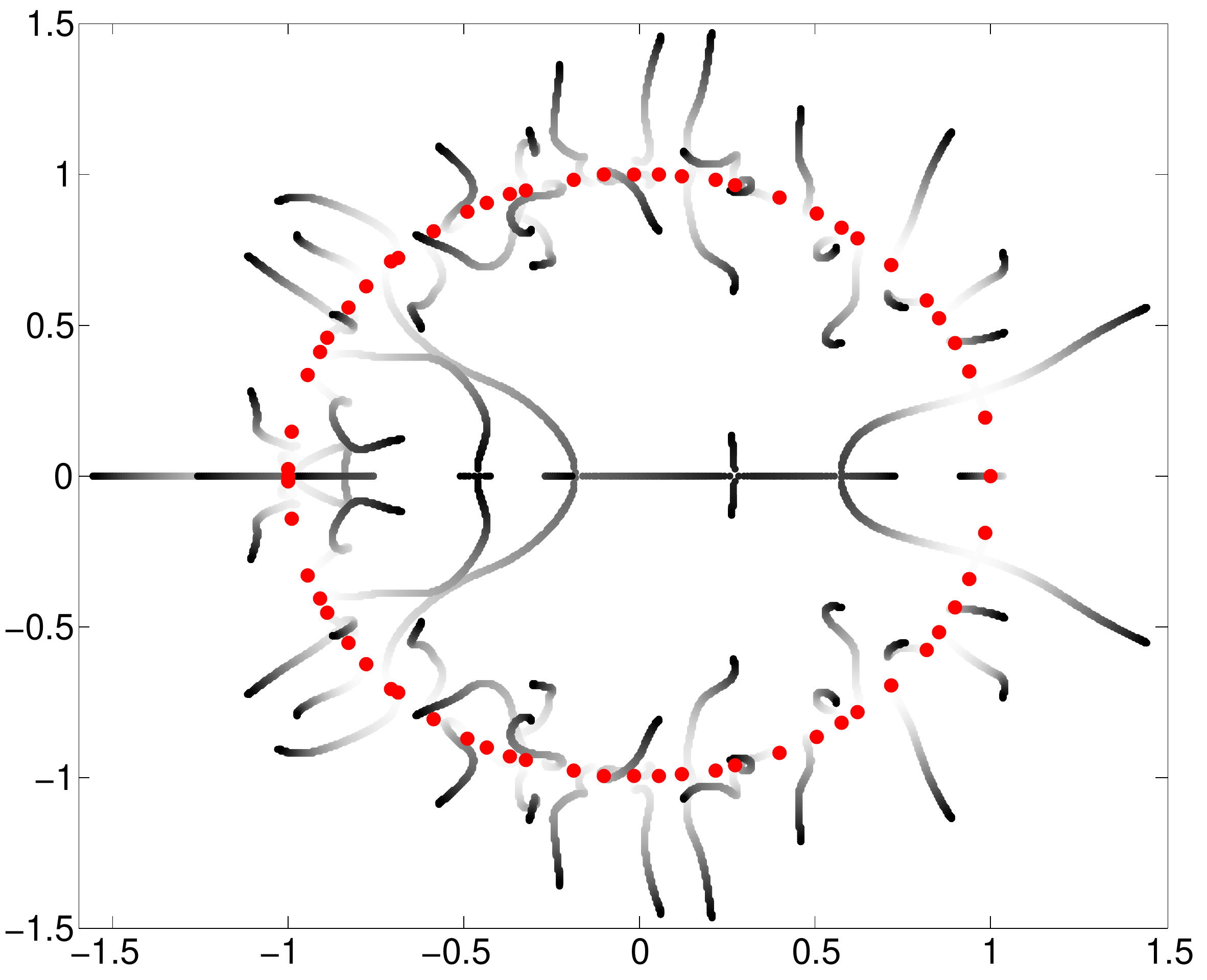}\caption{\label{fig:Orthogonal} Trajectories of the eigenvalues of $M+\delta t\mbox{ }P$.
Left: $M$ is an orthogonal matrix and $t_{max}=2$. $P$ is a real
random Gaussian matrix with norm $1$. Right: $M$ is an orthogonal
matrix and $P$ is a random $\pm1$ matrix whose norm is $2$ and
final time is $t_{max}=0.74$.}
\end{figure}

An application of this work is a better understanding of the origin
of real eigenvalues in the Hatano-Nelson model as discussed above
(Fig. \ref{fig:Hatano-Nelson}). In Fig. \ref{fig:Dynamics_HN_different_g}
one sees the formation of wings mentioned Sec. \ref{sec:Motivation}.
\begin{figure}
\includegraphics[scale=0.35]{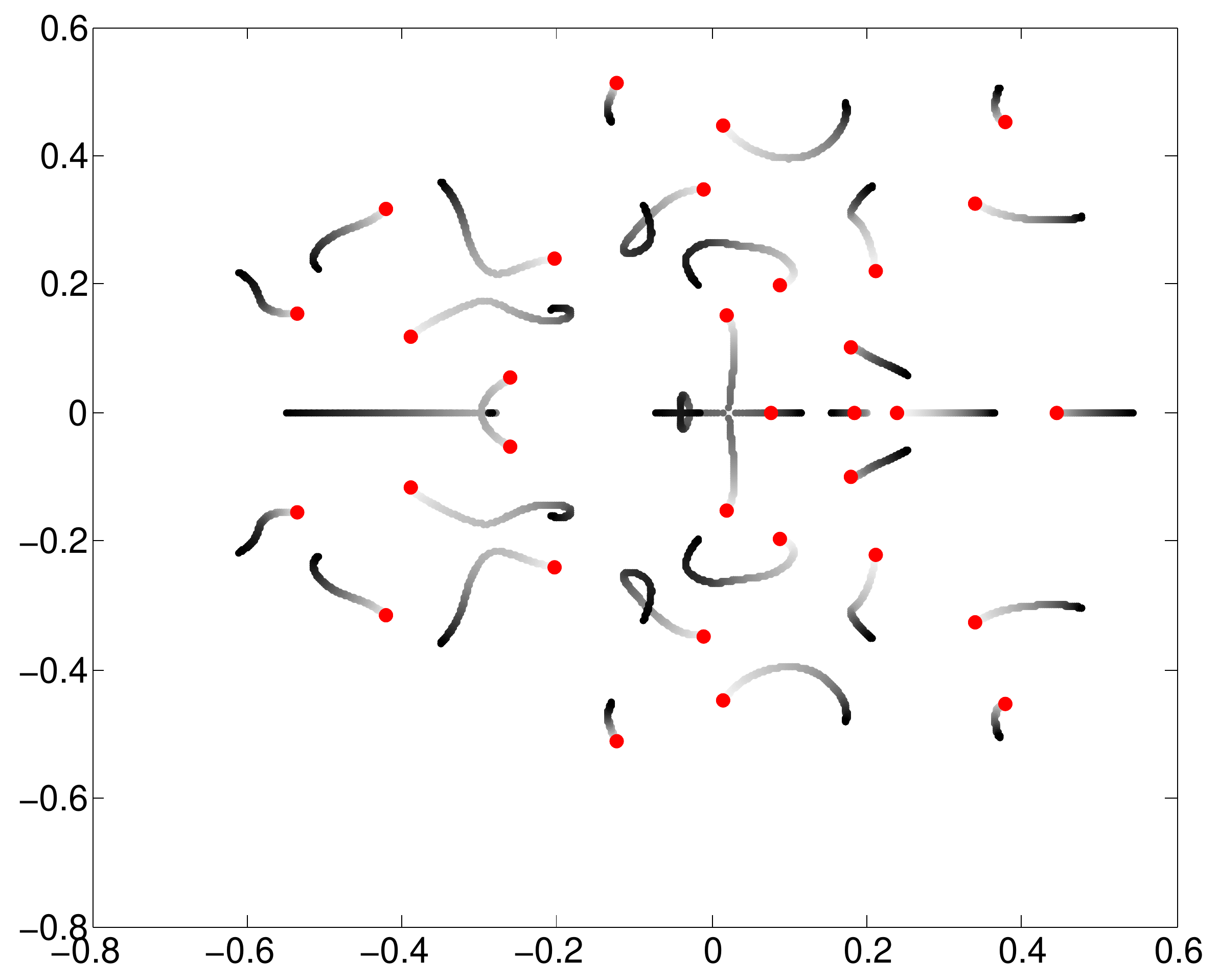}\caption{\label{fig:RealReal}Eigenvalues of $M+\delta t\mbox{ }P$, where
$M$ is a real Gaussian random matrix of size $32$ with unit $2-$norm.
$P$ is a real random Gaussian matrix with unit $2-$norm. We took
$t_{max}=0.5$. }
\end{figure}
Despite the base case being the same, the motion of the eigenvalues
of the Hatano-Nelson model in Figs. \ref{fig:Hatano-Nelson} (left)
and Fig. \ref{fig:Dynamics_HN_different_g} is much more uniform than
in Fig. \ref{fig:Stochastic-dynamics-of} (right). In the latter plot
in the course of the evolution there were $50$ intervals with a different
random $P$ acting in each. Although continuous, this makes the motion
jittery. 

In Fig. \ref{fig:RealReal} we take $M$ and $P$ to be two independent
$32\times32$ real random Gaussian matrices and normalize them to
have a unit $2-$norm. Here we also see cases where complex conjugate
pairs become real eigenvalues first and then as a result of yet another
encounter leave the real line by forming a complex conjugate pair
with the newly encountered eigenvalue.

\section{\label{sec:Stochatic-dynamics}Special case of stochastic dynamics
of the eigenvalues}

\subsection{A discrete stochastic process}

Suppose $M\left(t\right)$ is a discrete stochastically varying matrix
with $M\equiv M\left(0\right)$ being a fixed real $n\times n$ matrix.
We discretize time $0=t_{0}<t_{1}<\cdots$, and define the evolution
of $M\left(t\right)$ for any $t_{i}\le\delta t\le t_{i+1}$ by a
piece-wise linear stochastic process 
\begin{equation}
M\left(t_{i}+\delta t\right)=M\left(t_{i}\right)+\delta t\mbox{ }P\left(t_{i}\right)\quad,\label{eq:M_t_Discrete}
\end{equation}
where $\delta t\in\left[0,t_{i+1}-t_{i}\right)$ and each $P\left(t_{i}\right)$
is a random matrix \textit{impulse} whose entries are independent
and have mean zero i.e., $\mathbb{E}[p_{jk}]=0$ $\forall j,k$. The
eigenvalues of a continuous stochastic process are continuous in $t$
and within every interval $(t_{i},t_{i+1})$ the results of Subsection
\ref{sub:Random-perturbations-of} apply. 

Had we used $\sqrt{\delta t}$ in Eq. \ref{eq:M_t_Discrete}, and
defined the process such that $M\left(t_{i}+\delta t\right)-M\left(t_{i}\right)=\sqrt{\delta t}\mbox{ }P_{i}\sim N\left(0,\delta t\right)_{\mathbb{R}^{n}\times\mathbb{R}^{n}}$,
then $M\left(t\right)$ would define a \textit{discrete Wiener process,
}which is very special type of a stochastic process. The square root
of $\delta t$ is to satisfy the requirement that the variance grows
linearly with time. The natural geometry would then be a random walk
on the space of $n\times n$ real matrices \cite[Chapter 3]{TTao2012}.

\subsection{Smoothened discrete stochastic process}

In practice nothing develops infinitely fast and no impulse acts instantly.
It is more satisfactory to have a controlled smooth, albeit potentially
rapidly changing, formulation of the stochastic impulse. To this end,
in what follows we define a smooth version of the stochastic impulse,
denoted by $P_{\epsilon}\left(t\right)$, that in the limit of $\epsilon\rightarrow0$
becomes Eq. \ref{eq:M_t_Discrete}. Let 
\begin{equation}
P_{\epsilon}\left(t\right)=\sum_{i\ge0}P\left(t_{i}\right)\mbox{ }W_{\epsilon}\left(t;\mbox{ }t_{i},t_{i+1}\right),\label{eq:P_t_differentiable}
\end{equation}
 where each $P\left(t_{i}\right)$ is as before and we define the
\textit{window} function $W_{\epsilon}\left(t;\mbox{ }t_{i},t_{i+1}\right)$
to be (see Fig. \ref{fig:bump--for}) 

\[
W_{\epsilon}\left(t;\mbox{ }t_{i},t_{i+1}\right)=\left\{ \begin{array}{ccc}
\Theta\left(t_{i+1}\right)-\Theta\left(t_{i}\right) & \quad & t_{i}+\epsilon<t<t_{i+1}-\epsilon\\
B_{-}\left(t\mbox{ };\mbox{ }t_{i},t_{i+1}\right) & \quad & t\le t_{i}+\epsilon\\
B_{+}\left(t\mbox{ };\mbox{ }t_{i},t_{i+1}\right) & \quad & t\ge t_{i+1}-\epsilon
\end{array}\right.
\]
 with $\epsilon<\left(t_{i+1}-t_{i}\right)/2$, and $\Theta$ being
the Heaviside function. $B^{+}$ and $B^{-}$ are the right and left
sections of the modified bump function \cite{Loring2008} respectively
shown in Fig. \ref{fig:bump--for}, such that they reach zero at $t_{i}$
and $t_{i+1}$ and are scaled to have $1$ as their maxima (Fig. \ref{fig:bump--for}).
Mathematically, they are

\[
B_{-}(t\mbox{ };\mbox{ }t_{i},t_{i+1})=\left\{ \begin{array}{ccc}
e^{1-\frac{1}{1-\left[\left(t-t_{i}-\epsilon\right)/\epsilon\right]^{2}}} & \quad & t_{i}\le t\le t_{i}+\epsilon\\
0 &  & \mbox{otherwise}\mbox{ },
\end{array}\right.
\]
 and

\[
B_{+}(t\mbox{ };\mbox{ }t_{i},t_{i+1})=\left\{ \begin{array}{ccc}
e^{1-\frac{1}{1-\left[\left(t-t_{i+1}+\epsilon\right)/\epsilon\right]^{2}}} & \quad & t_{i+1}-\epsilon\le t\le t_{i+1}\\
0 &  & \mbox{otherwise}\mbox{ }.
\end{array}\right.
\]

We think of $B_{\pm}(t\mbox{ };\mbox{ }t_{i},t_{i+1})$ as equations
for the \textit{boundary layers} near every $t_{i}$. Moreover, the
desired independence of time intervals in the discrete stochastic
process is guaranteed by the independence of $P\left(t_{i}\right)$'s
and their confinement to $t_{i}\le t\le t_{i+1}$ by $W_{\epsilon}\left(t;\mbox{ }t_{i},t_{i+1}\right)$.
\begin{figure}
\centering{}\includegraphics[scale=0.45]{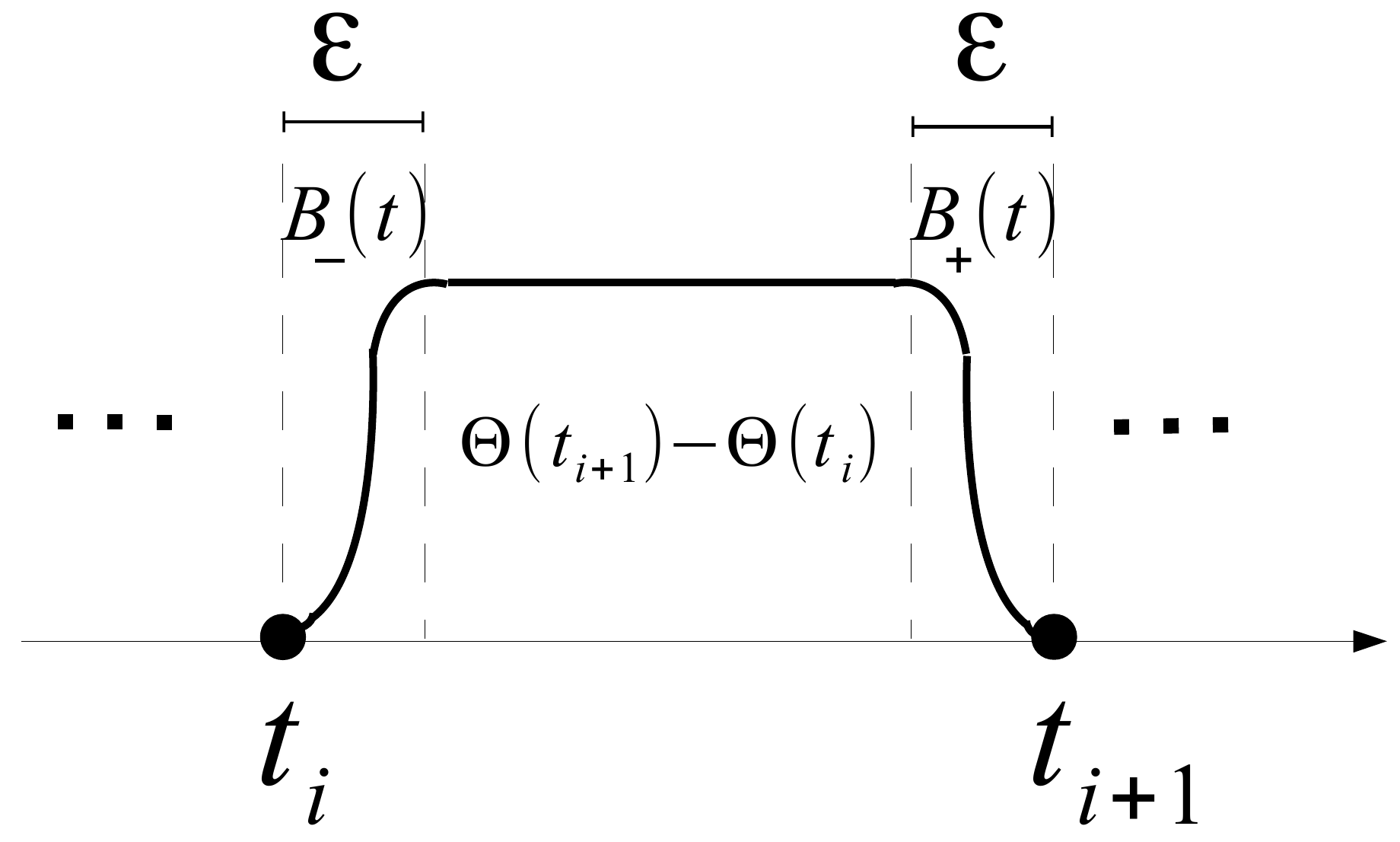}\caption{\label{fig:bump--for} The smooth window function $W_{\epsilon}\left(t;\mbox{ }t_{i},t_{i+1}\right)$
for $t_{i}\le t\le t_{i+1}$ . }
\end{figure}

It is easy to check that $P_{\epsilon}\left(t\right)$ is differentiable
everywhere \footnote{It is possible to construct $C^{\infty}$ versions of such window
functions such as the Planck-taper window function\cite[pp. 127-134]{Loring2008};
however, $W$ has a simple form with the basic differentiability properties
needed here.}. More specifically, for $\epsilon>0$, it is smooth ($C^{\infty}$)
everywhere but on a set of measure zero (i.e., all $t_{i}$), where
it is only $C^{1}$ . We shall need the differentiable property below.
With these definitions the smooth version of Eq. \ref{eq:M_t_Discrete}
reads

\begin{equation}
\overset{\centerdot}{M}_{\epsilon}\left(t\right)=P_{\epsilon}\left(t\right)\quad,\label{eq:M_t}
\end{equation}
 with the base case $M\equiv M_{\epsilon}\left(0\right)$ being a
fixed real $n\times n$ matrix. $M\left(t\right)$, which is not differentiable
at any $t_{i}$, is recovered by $\lim_{\epsilon\rightarrow0}M_{\epsilon}\left(t\right)$.

Since $P_{\epsilon}\left(t\right)$ is random and differentiable,
$\lambda\left(t\right)$'s are distinct with probability one. That
is $\lambda_{i}\left(t\right)$ has an open neighborhood around it
that does not contain any other eigenvalue and can be taken to be
a smooth function of $t$ \cite{TerryTao2009}. We shall investigate
the eigenvalues of $M\left(t\right)$.

The smoothing procedure is not necessary if one is only interested
in the spectral properties inside a single interval such as in Sec.
\ref{sub:Further-illustrations}. In this case, attraction holds for
$t\in\left(t_{i},t_{i+1}\right)$; i.e., outside the boundary layers.

From our derivations leading to Eq. \ref{eq:M_t}, it now becomes
obvious that, for small $\epsilon$, inside the boundary layers, i.e.,
times $\left(t_{i}-\epsilon,t_{i}+\epsilon\right)$, the dominant
force on any eigenvalue is the inertial force $\mathbf{u_{i}}^{*}\left(t\right)\mbox{ }\overset{\centerdot}{P}_{\epsilon}\left(t\right)\mbox{ }\mathbf{v_{i}}\left(t\right)$,
because

\begin{eqnarray}
\mathbf{u_{i}}^{*}\left(t\right)\mbox{ }\overset{\centerdot}{P}_{\epsilon}\left(t\right)\mbox{ }\mathbf{v_{i}}\left(t\right) & \sim & 1/\epsilon^{2}\qquad t_{j}<t<t_{j}+\epsilon\label{eq:BndryRATE}\\
\mathbf{u_{i}}^{*}\left(t\right)\mbox{ }\overset{\centerdot}{P}_{\epsilon}\left(t\right)\mbox{ }\mathbf{v_{i}}\left(t\right) & \sim & -1/\epsilon^{2}\qquad t_{j}-\epsilon<t<t_{j}\quad.\label{eq:BndryRate2}
\end{eqnarray}
For all other times this term is zero and the interaction of the eigenvalues,
given by the second term (Eq. \ref{eq:lambda_pp_finalRAMIS-1}), governs
the force. We will further discuss this and the continuum limit in
the next section.

Below to simplify notation, we let $\lambda_{i}\left(t\right)=\lambda_{i}$,
$\mathbf{v_{i}}\left(t\right)=\mathbf{v_{i}}$ , and $\mathbf{u_{i}}\left(t\right)=\mathbf{u_{i}}$,
whereby, Eq. \ref{eq:lambda_pp_final} reads

\begin{align}
\overset{\centerdot\centerdot}{\lambda_{i}}\left(t\right) & =(\mathbf{u}_{i}^{*}\mbox{ }\overset{\centerdot}{P}_{\epsilon}\left(t\right)\mbox{ }\mathbf{v_{i}})+2\sum_{j\ne i}\frac{c_{ij}\mbox{ }c_{ji}}{\lambda_{i}-\lambda_{j}}\label{eq:lamda_pp_Wiener}\\
 & \doteq\left(\mbox{Stochastic Force}\right)+\sum_{j\ne i}\left\{ \mbox{ Force of }\lambda_{j}\mbox{ on }\lambda_{i}\mbox{ }\right\} ,
\end{align}
where $c_{ij}=\mathbf{u_{i}^{*}}\mbox{ }P_{\epsilon}\left(t\right)\mbox{ }\mathbf{v_{j}}$. 
\begin{cor}
(\textbf{attraction}) Let $M\equiv M_{\epsilon}\left(0\right)$ be
a real matrix that evolves according to $\overset{\centerdot}{M}_{\epsilon}\left(t\right)=P_{\epsilon}\left(t\right)$,
where $P_{\epsilon}\left(t\right)$ is given by Eq. \ref{eq:P_t_differentiable}.
Then for all $t$, any complex conjugate pair of eigenvalues of $M_{\epsilon}\left(t\right)$
attract (as in Definition \ref{(Attraction-and-Repulsion)}). Moreover
the expected stochastic force is zero.\end{cor}
\begin{proof}
The attraction immediately follows from the previous proof of attraction
for small perturbations. Let us denote $P_{\epsilon}\left(t\right)$
and its components by $P_{\epsilon}$ and $p_{m\ell}$ respectively.
The first variation of the eigenvalues in Eq. \ref{eq:lambda_prime},
using index notation reads $\overset{\centerdot}{\lambda_{i}}=\mathbf{u}_{i}^{*,m}\mbox{ }p_{m\ell}\mbox{ }\mathbf{v}_{i}^{,\ell}$.

We comment that for small $t_{i}\le\delta t\le t_{i+1}$, $\mathbf{u_{i}}^{*}$
and $\mathbf{v_{i}}$ are taken to be eigenvectors of $M\left(t_{i}\right)$
which are independent of $P_{\epsilon}\left(t_{i}+\delta t\right)$.
Hence the right-hand sides, in the proof below, are accurate up to
$\mathcal{O}\left(\delta t\right)$. The first variation is 
\begin{equation}
\mathbb{E}[\overset{\centerdot}{\lambda_{i}}]=\bar{u}_{i}^{m}\mbox{ }\mathbb{E}[p_{m\ell}]\mbox{ }v_{i}^{\ell}=0\quad,\label{eq:E_vel_Ev}
\end{equation}
since $\mathbb{E}\left[p_{m\ell}\right]=0$ by assumption. From Eq.
\ref{eq:lambda_pp_final} we have 
\begin{eqnarray*}
\mathbb{E}[\overset{\centerdot\centerdot}{\lambda_{i}}] & = & \bar{u}_{i}^{m}\mbox{ }\mathbb{E}[(\overset{\centerdot}{P}_{\epsilon})_{m\ell}]\mbox{ }v_{i}^{\ell}+2\mathbb{E}\sum_{j\ne i}\frac{(\mathbf{u_{j}^{*}}\mbox{ }P_{\epsilon}\mbox{ }\mathbf{v_{i}})(\mathbf{u_{i}^{*}}\mbox{ }P_{\epsilon}\mbox{ }\mathbf{v_{j}})}{\lambda_{i}-\lambda_{j}}\quad.
\end{eqnarray*}
where $\mathbb{E}[(\overset{\centerdot}{P}_{\epsilon})_{m\ell}]=\sum_{i}\mathbb{E}[p_{m\ell}(t_{i})]\overset{\centerdot}{W}_{\epsilon}\left(t;\mbox{ }t_{i},t_{i+1}\right)=0$.
We have 
\begin{eqnarray}
\mathbb{E}[\overset{\centerdot\centerdot}{\lambda_{i}}] & = & -i\frac{\mathbb{E}[p^{2}]\left\Vert \mathbf{u_{i}}\right\Vert _{2}^{2}}{\mbox{Im}\left(\lambda_{i}\right)}+2\mathbb{E}[p^{2}]\sum_{j\ne\left\{ i,\bar{i}\right\} }\frac{\mathbf{\left(v_{i}^{T}v_{j}\right)}\mathbf{\left(u_{i}^{*}\overline{u_{j}}\right)}}{\lambda_{i}-\lambda_{j}}\label{eq:StochasticAttraction}
\end{eqnarray}
Therefore the expected force of attraction between complex conjugate
pairs is
\begin{equation}
\mathbb{E}[\mbox{force of }\overline{\lambda_{i}}\mbox{ on }\lambda_{i}]=-i\frac{\mathbb{E}[p^{2}]\left\Vert \mathbf{u_{i}}\right\Vert _{2}^{2}}{\mbox{Im}\left(\lambda_{i}\right)}\quad.\label{eq:Expected_cc_attraction}
\end{equation}
\end{proof}
\begin{rem}
Complex conjugate eigenvalues and eigenvectors that ultimately become
real or those that are initially real and eventually become a complex
conjugate pair must first become equal. Since the motion of the eigenvalues
is continuous and the matrix is real, the transition from a complex
conjugate pair to two real eigenvalues or vice versa requires that
they first become equal. This corollary may be obvious but perhaps
is interesting in that the degeneracy of eigenvalues is forced under
a generic evolution. \end{rem}
\begin{cor}
\label{cor:IllCondition}The expected force of attraction of the complex
conjugate eigenvalues is directly proportional to the square of the
$2-$norm of the left eigenvector and the variance of the entries
of the perturbation matrix. 
\end{cor}
This is an immediate consequence of Eq. \ref{eq:StochasticAttraction}.
The numerator in Eq. \ref{eq:StochasticAttraction} can change the
strength of interaction, most notably because of $\left\Vert \mathbf{u_{i}}^{*}\right\Vert _{2}^{2}$,
which for non-normal matrices can become quite large. 

Now suppose we want to define a \textit{continuous stochastic process
}where in the equations of motion we first take $\epsilon\rightarrow0$
and then $\delta t\rightarrow0$. The first limit will produce two
Dirac delta functions at each $t_{i}$ whereby $\lim_{t\rightarrow t_{i}^{<}}\overset{\centerdot}{P}\left(t\right)=-P\left(t_{i-1}\right)\delta\left(t_{i}\right)$
and $\lim_{t\rightarrow t_{i}^{>}}\overset{\centerdot}{P}\left(t\right)=P\left(t_{i}\right)\delta\left(t_{i}\right)$;
therefore the function is not differentiable at $t_{i}$. Lastly,
$\delta t\rightarrow0$ will ensure that the stochastic process is
nowhere differentiable as one expects from continuous Brownian motion
ideas. So what does this mean for eigenvalue attraction? The infinitesimally
close delta function impulses dominate the time evolution (Eqs. \ref{eq:BndryRATE}
and \ref{eq:BndryRate2}) with a zero mean force on any eigenvalue.
\begin{rem}
For a real stochastic process as before the requirement of $M_{\epsilon}\left(t\right)$
being normal for all times demands that $\left[M_{\epsilon}\left(t_{i}\right)-M_{\epsilon}^{T}\left(t_{i}\right),P_{\epsilon}\left(t_{i}\right)\right]=0$
which generally is not met for generic $P_{\epsilon}\left(t_{i}\right)$. 
\end{rem}

\section{\label{sec:Further-discussions}Further discussions and open problems}

Strongly attracting complex conjugate pairs ultimately coalesce on
the real line and scatter like billiard balls and move about on the
real line. Thereafter they can act like ``normal'' eigenvalues and
repel. In particular, the inevitability of collision between an eigenvalue
and its complex conjugate prevents any eigenvalue from crossing the
real line (changing the sign of its imaginary part), without a second
encounter.

The motion of the eigenvalues is constrained by the reality of the
matrix; the eigenvalue distribution remains symmetric about the real
axis. As can be seen in the Figures, an interesting scenario is when
two complex conjugate eigenvalues attract and coalesce on the real
line, after which they move in opposite directions on the real line
till one of them collides with another (real) eigenvalue. Subsequently,
the newly encountered eigenvalue and one of the original complex conjugate
eigenvalues can momentarily become equal, then move off the real line
as a new complex conjugate pair (see Fig. \ref{fig:RealReal}, the
left figure in Fig. \ref{fig:Example-3} and the right figure in Fig.
\ref{fig:Orthogonal} for examples). At times they simply repel each
other and remain real.

We emphasize that the proof of complex conjugate attraction is one
of the many forces and at any given instance the net force on any
eigenvalue (Eq. \ref{eq:lambda_pp_final} and Eq. \ref{eq:lamda_pp_Wiener})
is the result of the sum of forces of the remaining $n-1$ eigenvalues.
In particular, a random collision can take place in the complex plane
between eigenvalues that are not complex conjugates and cause a deviation
in the path of an eigenvalue that initially moved towards the real
line.

It should be clear that the attraction proved in this work does \textit{not}
imply that the long-time behavior is an aggregation of all the eigenvalues
on the real line. For a fixed $\delta t$, and over long times, the
process loses memory of the initial condition (i.e., $M\left(t=0\right)$)
and ultimately behaves like a random walk on the space of $\mathbb{R}^{n\times n}$
matrices. In fact, Edelman, Kostlan, and Shub \cite{Edelman1994}
showed that for an $n\times n$ real random matrix whose entries are
drawn from a standard normal distribution, the expected number of
real eigenvalues is approximately $\sqrt{\frac{2n}{\pi}}$. Later
Tao and Vu \cite{TaoVu2012} proved that matrices whose entries are
jointly independent, exponentially decaying, and whose moments match
the real Gaussian ensemble to fourth order have $\sqrt{\frac{2n}{\pi}}+o\left(\sqrt{n}\right)$
real eigenvalues.

It would be interesting to calculate relaxation times for real deterministic
matrices that evolve stochastically and see how long it takes for
the matrix to start acting ``typical'' whereafter the results just
mentioned determine the expected behavior.

For any simple eigenvalue $\lambda_{i}$, the condition number \cite[p. 474]{TrefethenEmbree2005}
is a function of the angle between the left and right eigenvectors
denoted by $\theta_{0}$ 
\[
\kappa_{i}=\frac{\left\Vert \mathbf{u_{i}}\right\Vert \left\Vert \mathbf{v_{i}}\right\Vert }{\left|\mathbf{u_{i}}^{*}\mathbf{v_{i}}\right|}=\left\Vert \mathbf{u_{i}}\right\Vert \equiv\frac{1}{\left|\cos\theta_{0}^{\ell}\right|},
\]
where we used the orthogonality condition (Eq. \ref{eq:orthogonality})
and the normality (unit length) of $\mathbf{v}_{i}$. By the Cauchy-Schwarz
inequality $\kappa_{i}\ge1$. An eigenvalue for which $\kappa_{i}=1$
is called a normal eigenvalue and is stable under perturbation, whereas
an ill-conditioned eigenvalue has $\kappa_{i}\gg1$ . The right and
left eigenvectors associated to an ill-conditioned eigenvalue can
become almost orthogonal implying $\kappa_{i}\gg1$ or equivalently
$\left\Vert \mathbf{u_{i}}\right\Vert \gg1$.

The eigenvalues of normal matrices (e.g., Hermitian, unitary) are
very stable under small perturbations. This is not generally the case
for non-normal (e.g., Toeplitz) matrices, where small perturbations
can change the spectrum significantly \cite{TrefethenEmbree2005,BottEmbreeSokolov}.
Therefore by Corollary \ref{cor:IllCondition}, complex conjugate
eigenvalues that are distant; i.e., $\mbox{Im}\left(\lambda_{i}\right)$
is not necessarily small, can attract strongly if they are ill-conditioned.

\subsection{\label{sub:Conjecture}Does the low density of eigenvalues near the
real line result from repulsion?}

It was previously argued that the relatively low density of eigenvalues
of real random matrices seen near the real axis results from a repulsion
of eigenvalues from the real line \cite{Mays2014}, \cite[Section 6.1]{EdelmanKostlan1995}(see
Fig. \ref{fig:Conjecture}). Preceding \cite{EdelmanKostlan1995},
Edelman derived the distribution of the eigenvalues for standard normal
random matrices and, interestingly, argued that one might think of
the real axis as attracting the nearby eigenvalues \cite[Section 2 following Theorem 6.2]{Edelman}
(preprint appeared in 1993).

One can conceive of a potentially more complete explanation, where
the interaction and dynamics of the eigenvalues take the center stage
and not a mysterious interaction with the real axis. To do so, one
might need to relate every instance of a real random matrix to the
limit of a stochastic process with a base case contained in the deformations
of the particular matrix (see below). 
\begin{figure}
\includegraphics[scale=0.35]{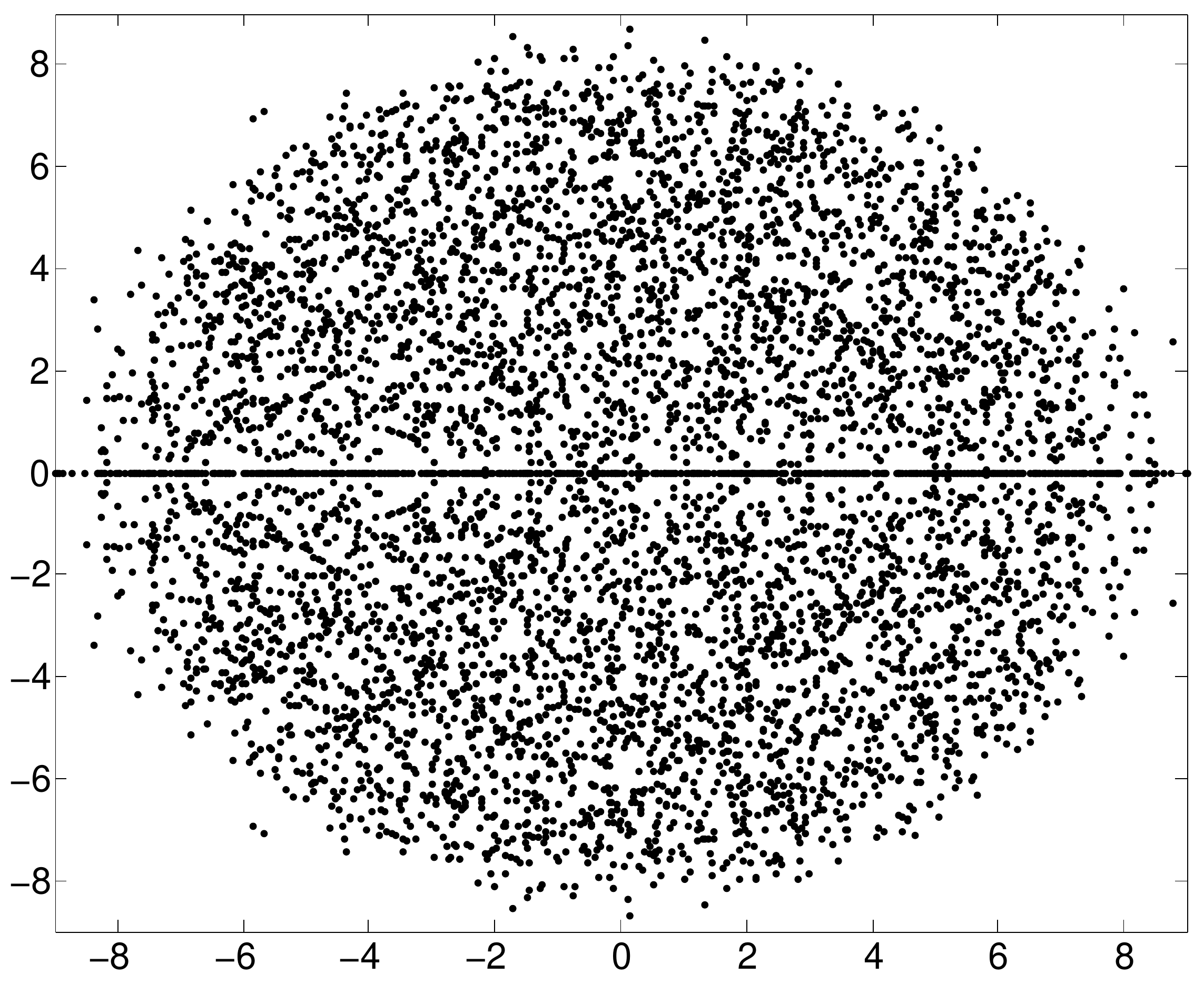}$\quad$\includegraphics[scale=0.35]{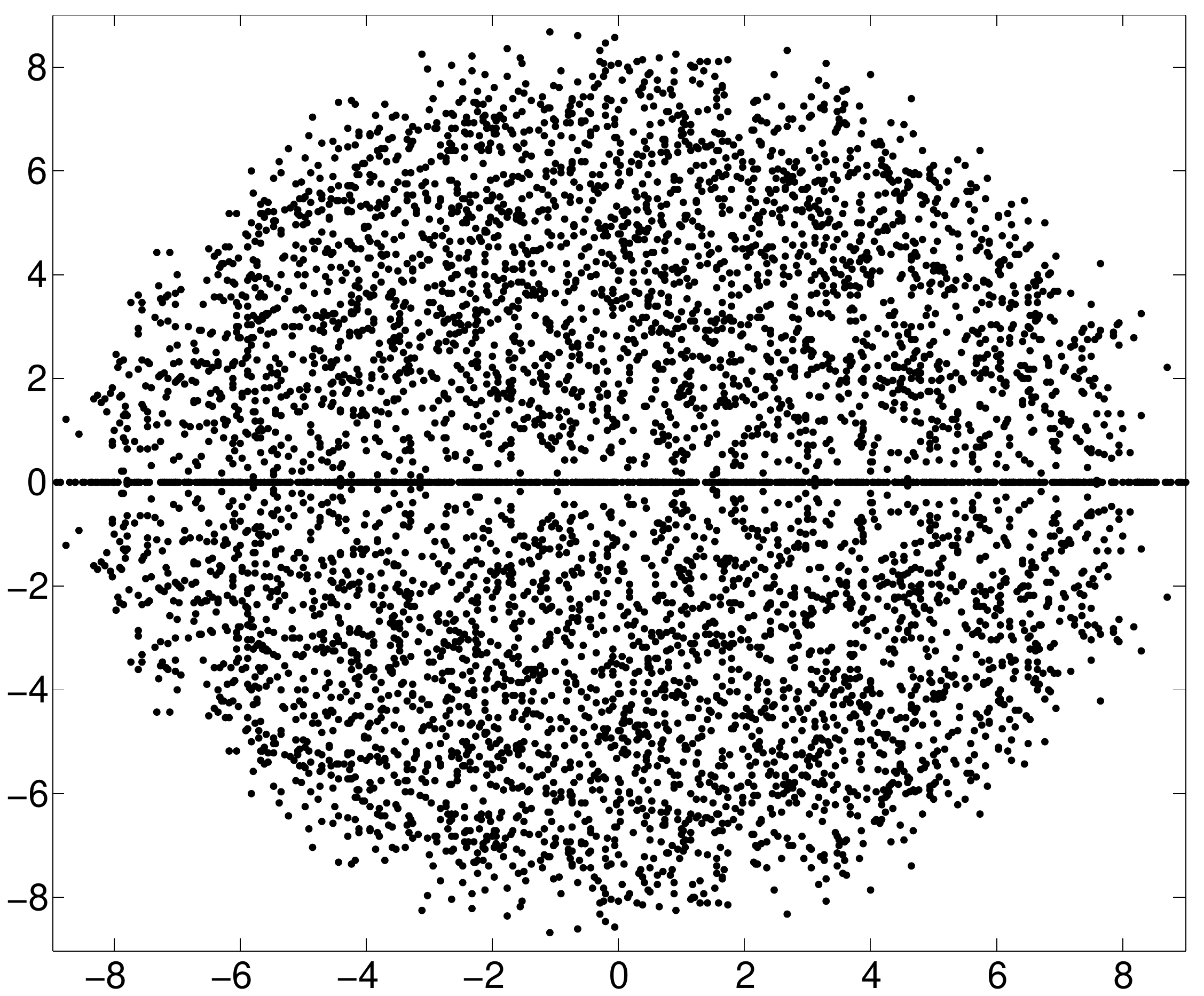}

\caption{\label{fig:Conjecture}Eigenvalues of $100$ instances of $64\times64$
real matrices. Left: Entries drawn from a standard normal distribution.
Right: Entries are random $\pm1$. Note the aggregation of the eigenvalues
on the real axis and their lower nearby density.}
\end{figure}

In the considerations above, the complex conjugate eigenvalues of
a real matrix move more rapidly towards the real line. Moreover, colliding
eigenvalues on the real line have a large acceleration when they shoot
off the real line. This is easily seen from Eq. \ref{eq:lambda_pp_finalRAMIS-1};
real eigenvalues have high accelerations as $\lambda_{i}-\lambda_{j}$
is small and their subsequent motion is either on the real axis or
into the complex plane.

In either case when there is an imaginary component to the acceleration,
its magnitude is quite high. The former corresponds to high accelerations
towards the real line which result in the complex conjugate eigenvalues
becoming real. The latter is a strong repulsion away from the real
line, shooting the eigenvalues into the complex plane away from the
real axis. Therefore, at any given time, on average, one expects a
smaller number of eigenvalues to be in the vicinity of the real axis.
For large times, when the \textit{equilibrium} is reached, $M\left(t\right)$
will have lost the memory of the initial conditions. At all subsequent
times, under the stochastic evolution, some eigenvalues become real
(because of attraction) and some move off the real line (because of
collisions) and on average about $\mathcal{O}\left(\sqrt{n}\right)$
of the eigenvalues will be found on the real line.

Large forces between nearby eigenvalues is in no way special to the
ones with small imaginary components. However, the reality of the
matrix causes an anisotropy-- the acceleration of the eigenvalues
in the imaginary direction becomes larger. 
\begin{conjecture*}
The low density of eigenvalues of real random matrices near the real
axis is the result of the large imaginary components of the acceleration
into (attraction of complex conjugate pair) and away from (colliding
real eigenvalues) the real axis. 
\end{conjecture*}
The relative lack of stability of eigenvalues may explain their aggregation
on the real line, as well as their low nearby density. In order to
settle this conjecture, a first step might be to construct any $n\times n$
real random matrix as a limit of a dynamical process like we did above.
In particular, deformations of a given random matrix are also random,
so in a way, one can conceive of the stochastic process to be the
deformations of a matrix in the neighborhood of the random matrix.
Then one can relate the expectation of finding an eigenvalue to the
expectation of the time it spends anywhere on the complex plane, which,
for real random matrices, would be lower in the vicinity of the real
axis.

It is our belief that the repulsion of eigenvalues away from the real
line is only part of the story in accounting for their relative low
density near the real axis.

\subsection{Further opportunities for future work}

A list of other open problems includes:
\begin{enumerate}
\item What is the probability of collision of eigenvalues on the the real
line?\footnote{This question was posed to us by Freeman Dyson.} The
answer to this question would be a helpful step in proving the conjecture
above.
\item Estimation of $c_{ij}$ would help quantify the direction and strength
of interaction between pairs of eigenvalues. 
\item One could give an estimate of the imaginary part of Eq. \ref{eq:HatanoNelson_acceleration}.
By doing so one can solve the differential equation to calculate the
time it takes for any eigenvalue, $\lambda_{k}$, to reach the real
axis, which is the time that $\lambda_{k}$ and $\overline{\lambda_{k}}$
collide and momentarily become degenerate. Indeed, let $u\equiv\mbox{Im}\left(\lambda_{k}\right)$.
The imaginary part of the differential equation (Eq. \ref{eq:HatanoNelson_acceleration})
is of the form $\overset{\centerdot\centerdot}{u}=f\left(u\right)$
, which can be solved by first multiplying both sides by $\overset{\centerdot}{u}$.
That is $\overset{\centerdot}{u}\overset{\centerdot\centerdot}{u}=\frac{1}{2}\frac{d}{dt}\left(\overset{\centerdot}{u}\right)^{2}=\overset{\centerdot}{u}f\left(u\right)$
and one has $d\left(\overset{\centerdot}{u}\right)^{2}=2f\left(u\right)\mbox{ }du$,
hence $\overset{\centerdot}{u}\left(t\right)=\sqrt{\left(\overset{\centerdot}{u}\left(0\right)\right)^{2}+2\int\mbox{ }f\left(u\right)\mbox{ }du}$,
which can be integrated once more to solve for $t$ when $u\left(t\right)=0$. 
\item Study of eigenvector localization, especially as the eigenvectors
become more real.
\item Toeplitz matrices provide excellent examples of matrices that can
be asymmetric and arise in various applications \cite{TrefethenEmbree2005}.
It would be interesting to better understand the role of the symbol
(e.g. its singularity) \cite{TrefethenEmbree2005} in connection with
this work.
\item Do new features appear in the operator limit?
\item Application of this work in other areas such as open quantum systems
\cite{ticozzi2012hamiltonian,ticozzi2012stabilizing}, PT-symmetric
material \cite{Ramezani2011} and biophysics \cite{Nelson2012}.
\end{enumerate}

\section*{Acknowledgements}

I thank Leo P. Kadanoff, Steven G. Johnson, Tony Iarrobino and Gil
Strang for discussions and the James Franck Institute at University
of Chicago and the Perimeter Institute Canada, for having hosted me
over the summer of 2013. I acknowledge the National Science Foundation's
support through grant DMS. 1312831.

 \bibliographystyle{apsrev4-1}
\bibliography{mybib}

\begin{thebibliography}{35}%
\makeatletter
\providecommand \@ifxundefined [1]{%
 \@ifx{#1\undefined}
}%
\providecommand \@ifnum [1]{%
 \ifnum #1\expandafter \@firstoftwo
 \else \expandafter \@secondoftwo
 \fi
}%
\providecommand \@ifx [1]{%
 \ifx #1\expandafter \@firstoftwo
 \else \expandafter \@secondoftwo
 \fi
}%
\providecommand \natexlab [1]{#1}%
\providecommand \enquote  [1]{``#1''}%
\providecommand \bibnamefont  [1]{#1}%
\providecommand \bibfnamefont [1]{#1}%
\providecommand \citenamefont [1]{#1}%
\providecommand \href@noop [0]{\@secondoftwo}%
\providecommand \href [0]{\begingroup \@sanitize@url \@href}%
\providecommand \@href[1]{\@@startlink{#1}\@@href}%
\providecommand \@@href[1]{\endgroup#1\@@endlink}%
\providecommand \@sanitize@url [0]{\catcode `\\12\catcode `\$12\catcode
  `\&12\catcode `\#12\catcode `\^12\catcode `\_12\catcode `\%12\relax}%
\providecommand \@@startlink[1]{}%
\providecommand \@@endlink[0]{}%
\providecommand \url  [0]{\begingroup\@sanitize@url \@url }%
\providecommand \@url [1]{\endgroup\@href {#1}{\urlprefix }}%
\providecommand \urlprefix  [0]{URL }%
\providecommand \Eprint [0]{\href }%
\providecommand \doibase [0]{http://dx.doi.org/}%
\providecommand \selectlanguage [0]{\@gobble}%
\providecommand \bibinfo  [0]{\@secondoftwo}%
\providecommand \bibfield  [0]{\@secondoftwo}%
\providecommand \translation [1]{[#1]}%
\providecommand \BibitemOpen [0]{}%
\providecommand \bibitemStop [0]{}%
\providecommand \bibitemNoStop [0]{.\EOS\space}%
\providecommand \EOS [0]{\spacefactor3000\relax}%
\providecommand \BibitemShut  [1]{\csname bibitem#1\endcsname}%
\let\auto@bib@innerbib\@empty
\bibitem [{\citenamefont {Trefethen}\ and\ \citenamefont
  {Embree}(2005)}]{TrefethenEmbree2005}%
  \BibitemOpen
  \bibfield  {author} {\bibinfo {author} {\bibfnamefont {L.~N.}\ \bibnamefont
  {Trefethen}}\ and\ \bibinfo {author} {\bibfnamefont {M.}~\bibnamefont
  {Embree}},\ }\href@noop {} {\emph {\bibinfo {title} {Spectra and
  Pseudospectra}}}\ (\bibinfo  {publisher} {Princeton University Press},\
  \bibinfo {year} {2005})\BibitemShut {NoStop}%
\bibitem [{\citenamefont {Tao}(2009)}]{TerryTao2009}%
  \BibitemOpen
  \bibfield  {author} {\bibinfo {author} {\bibfnamefont {T.}~\bibnamefont
  {Tao}},\ }\href@noop {} {\emph {\bibinfo {title} {Poincar{\'e}'s legacies:
  pages from year two of a mathematical blog}}}\ (\bibinfo  {publisher}
  {American Mathematical Society},\ \bibinfo {year} {2009})\BibitemShut
  {NoStop}%
\bibitem [{\citenamefont {Landau}\ and\ \citenamefont
  {Lifshitz}(1981)}]{LandauLifshitz1981}%
  \BibitemOpen
  \bibfield  {author} {\bibinfo {author} {\bibfnamefont {L.~D.}\ \bibnamefont
  {Landau}}\ and\ \bibinfo {author} {\bibfnamefont {E.~M.}\ \bibnamefont
  {Lifshitz}},\ }\href@noop {} {\emph {\bibinfo {title} {Quantum Mechanics}}}\
  (\bibinfo  {publisher} {Pergamon Press},\ \bibinfo {year} {1981})\BibitemShut
  {NoStop}%
\bibitem [{\citenamefont {Erd{\"o}s}\ \emph {et~al.}(2010)\citenamefont
  {Erd{\"o}s}, \citenamefont {Schlein},\ and\ \citenamefont {Yau}}]{Yau2010}%
  \BibitemOpen
  \bibfield  {author} {\bibinfo {author} {\bibfnamefont {L.}~\bibnamefont
  {Erd{\"o}s}}, \bibinfo {author} {\bibfnamefont {B.}~\bibnamefont {Schlein}},
  \ and\ \bibinfo {author} {\bibfnamefont {H.-T.}\ \bibnamefont {Yau}},\
  }\href@noop {} {\bibfield  {journal} {\bibinfo  {journal} {Int Math Res
  Notices}\ }\textbf {\bibinfo {volume} {3}},\ \bibinfo {pages} {436} (\bibinfo
  {year} {2010})}\BibitemShut {NoStop}%
\bibitem [{\citenamefont {Tao}(2012)}]{TTao2012}%
  \BibitemOpen
  \bibfield  {author} {\bibinfo {author} {\bibfnamefont {T.}~\bibnamefont
  {Tao}},\ }\href@noop {} {\emph {\bibinfo {title} {Topics in Random Matrix
  Theory (Graduate Studies in Mathematics)}}}\ (\bibinfo  {publisher} {American
  Mathematical Society},\ \bibinfo {year} {March 21, 2012})\BibitemShut
  {NoStop}%
\bibitem [{\citenamefont {Dyson}(1962)}]{Dyson1962}%
  \BibitemOpen
  \bibfield  {author} {\bibinfo {author} {\bibfnamefont {F.}~\bibnamefont
  {Dyson}},\ }\href@noop {} {\bibfield  {journal} {\bibinfo  {journal} {J.
  Mathematical Phys.}\ }\textbf {\bibinfo {volume} {3}},\ \bibinfo {pages}
  {1191 } (\bibinfo {year} {1962})}\BibitemShut {NoStop}%
\bibitem [{\citenamefont {Hatano}\ and\ \citenamefont
  {Nelson}(1997)}]{HatanoNelson1997}%
  \BibitemOpen
  \bibfield  {author} {\bibinfo {author} {\bibfnamefont {N.}~\bibnamefont
  {Hatano}}\ and\ \bibinfo {author} {\bibfnamefont {D.~R.}\ \bibnamefont
  {Nelson}},\ }\href@noop {} {\bibfield  {journal} {\bibinfo  {journal}
  {Physical Review B}\ }\textbf {\bibinfo {volume} {56}},\ \bibinfo {pages}
  {8651 } (\bibinfo {year} {1997})}\BibitemShut {NoStop}%
\bibitem [{\citenamefont {Br{\'e}zin}\ and\ \citenamefont
  {Zee}(1998)}]{BrezinZee1998}%
  \BibitemOpen
  \bibfield  {author} {\bibinfo {author} {\bibfnamefont {E.}~\bibnamefont
  {Br{\'e}zin}}\ and\ \bibinfo {author} {\bibfnamefont {A.}~\bibnamefont
  {Zee}},\ }\href@noop {} {\bibfield  {journal} {\bibinfo  {journal} {Nuclear
  Physics B}\ }\textbf {\bibinfo {volume} {509}},\ \bibinfo {pages} {599}
  (\bibinfo {year} {1998})}\BibitemShut {NoStop}%
\bibitem [{\citenamefont {Brouwer}\ \emph {et~al.}(1997)\citenamefont
  {Brouwer}, \citenamefont {Silvestrov},\ and\ \citenamefont
  {Beenakker}}]{Brouwer1997}%
  \BibitemOpen
  \bibfield  {author} {\bibinfo {author} {\bibfnamefont {P.}~\bibnamefont
  {Brouwer}}, \bibinfo {author} {\bibfnamefont {P.}~\bibnamefont {Silvestrov}},
  \ and\ \bibinfo {author} {\bibfnamefont {C.}~\bibnamefont {Beenakker}},\
  }\href@noop {} {\bibfield  {journal} {\bibinfo  {journal} {Physical Review
  B}\ }\textbf {\bibinfo {volume} {56}},\ \bibinfo {pages} {55} (\bibinfo
  {year} {1997})}\BibitemShut {NoStop}%
\bibitem [{\citenamefont {Feinberg}\ and\ \citenamefont
  {Zee}(1997)}]{FeinbergZee1997}%
  \BibitemOpen
  \bibfield  {author} {\bibinfo {author} {\bibfnamefont {J.}~\bibnamefont
  {Feinberg}}\ and\ \bibinfo {author} {\bibfnamefont {A.}~\bibnamefont {Zee}},\
  }\href@noop {} {\  (\bibinfo {year} {1997})},\ \Eprint
  {http://arxiv.org/abs/arXiv:9706218} {arXiv:9706218} \BibitemShut {NoStop}%
\bibitem [{\citenamefont {Widom}(1994)}]{Widom7}%
  \BibitemOpen
  \bibfield  {author} {\bibinfo {author} {\bibfnamefont {H.}~\bibnamefont
  {Widom}},\ }\href@noop {} {\bibfield  {journal} {\bibinfo  {journal}
  {Operator Theory: Advances and Applications}\ }\textbf {\bibinfo {volume}
  {71}},\ \bibinfo {pages} {1} (\bibinfo {year} {1994})}\BibitemShut {NoStop}%
\bibitem [{\citenamefont {Lin}\ \emph {et~al.}(2011)\citenamefont {Lin},
  \citenamefont {Ramezani}, \citenamefont {Eichelkraut}, \citenamefont
  {Kottos}, \citenamefont {Cao},\ and\ \citenamefont
  {Christodoulides}}]{Ramezani2011}%
  \BibitemOpen
  \bibfield  {author} {\bibinfo {author} {\bibfnamefont {Z.}~\bibnamefont
  {Lin}}, \bibinfo {author} {\bibfnamefont {H.}~\bibnamefont {Ramezani}},
  \bibinfo {author} {\bibfnamefont {T.}~\bibnamefont {Eichelkraut}}, \bibinfo
  {author} {\bibfnamefont {T.}~\bibnamefont {Kottos}}, \bibinfo {author}
  {\bibfnamefont {H.}~\bibnamefont {Cao}}, \ and\ \bibinfo {author}
  {\bibfnamefont {D.~N.}\ \bibnamefont {Christodoulides}},\ }\href@noop {}
  {\bibfield  {journal} {\bibinfo  {journal} {Phys. Rev. Lett.}\ }\textbf
  {\bibinfo {volume} {106}},\ \bibinfo {pages} {213901} (\bibinfo {year}
  {2011})}\BibitemShut {NoStop}%
\bibitem [{\citenamefont {Nelson}(2012)}]{Nelson2012}%
  \BibitemOpen
  \bibfield  {author} {\bibinfo {author} {\bibfnamefont {D.~R.}\ \bibnamefont
  {Nelson}},\ }\href@noop {} {\bibfield  {journal} {\bibinfo  {journal} {Annu.
  Rev. Biophys.}\ }\textbf {\bibinfo {volume} {41}},\ \bibinfo {pages} {371}
  (\bibinfo {year} {2012})}\BibitemShut {NoStop}%
\bibitem [{\citenamefont {Goldsheid}\ and\ \citenamefont
  {Khoruzhenko}(1998)}]{Goldsheid1998}%
  \BibitemOpen
  \bibfield  {author} {\bibinfo {author} {\bibfnamefont {I.~Y.}\ \bibnamefont
  {Goldsheid}}\ and\ \bibinfo {author} {\bibfnamefont {B.~A.}\ \bibnamefont
  {Khoruzhenko}},\ }\href@noop {} {\bibfield  {journal} {\bibinfo  {journal}
  {Physical Review Letters}\ }\textbf {\bibinfo {volume} {80}},\ \bibinfo
  {pages} {2897} (\bibinfo {year} {1998})}\BibitemShut {NoStop}%
\bibitem [{\citenamefont {Bloch}\ \emph {et~al.}(2012)\citenamefont {Bloch},
  \citenamefont {Bruckmann}, \citenamefont {Meyer},\ and\ \citenamefont
  {Schierenberg}}]{Bloch2012}%
  \BibitemOpen
  \bibfield  {author} {\bibinfo {author} {\bibfnamefont {J.}~\bibnamefont
  {Bloch}}, \bibinfo {author} {\bibfnamefont {F.}~\bibnamefont {Bruckmann}},
  \bibinfo {author} {\bibfnamefont {N.}~\bibnamefont {Meyer}}, \ and\ \bibinfo
  {author} {\bibfnamefont {S.}~\bibnamefont {Schierenberg}},\ }\href@noop {}
  {\bibfield  {journal} {\bibinfo  {journal} {Journal of High Energy Physics}\
  }\textbf {\bibinfo {volume} {8}},\ \bibinfo {pages} {1} (\bibinfo {year}
  {2012})}\BibitemShut {NoStop}%
\bibitem [{\citenamefont {Anderson}(1958)}]{Anderson58}%
  \BibitemOpen
  \bibfield  {author} {\bibinfo {author} {\bibfnamefont {P.~W.}\ \bibnamefont
  {Anderson}},\ }\href@noop {} {\bibfield  {journal} {\bibinfo  {journal}
  {Physical Review}\ }\textbf {\bibinfo {volume} {109}},\ \bibinfo {pages}
  {1492 } (\bibinfo {year} {1958})}\BibitemShut {NoStop}%
\bibitem [{\citenamefont {Stewart}\ and\ \citenamefont
  {Sun}(1990)}]{StewartSun}%
  \BibitemOpen
  \bibfield  {author} {\bibinfo {author} {\bibfnamefont {G.}~\bibnamefont
  {Stewart}}\ and\ \bibinfo {author} {\bibfnamefont {J.-G.}\ \bibnamefont
  {Sun}},\ }\href@noop {} {\emph {\bibinfo {title} {Matrix Perturbation
  Theory}}},\ \bibinfo {edition} {1st}\ ed.\ (\bibinfo  {publisher} {Academic
  Press},\ \bibinfo {year} {1990})\BibitemShut {NoStop}%
\bibitem [{\citenamefont {Dirac}(1958)}]{dirac1958principles}%
  \BibitemOpen
  \bibfield  {author} {\bibinfo {author} {\bibfnamefont {P.~A.~M.}\
  \bibnamefont {Dirac}},\ }\href@noop {} {\emph {\bibinfo {title} {The
  principles of quantum mechanics}}},\ Vol.~\bibinfo {volume} {4}\ (\bibinfo
  {publisher} {Clarendon Press Oxford},\ \bibinfo {year} {1958})\BibitemShut
  {NoStop}%
\bibitem [{\citenamefont {Kato}(1976)}]{kato1976perturbation}%
  \BibitemOpen
  \bibfield  {author} {\bibinfo {author} {\bibfnamefont {T.}~\bibnamefont
  {Kato}},\ }\href@noop {} {\emph {\bibinfo {title} {Perturbation theory for
  linear operators}}},\ Vol.\ \bibinfo {volume} {132}\ (\bibinfo  {publisher}
  {Springer Science \& Business Media},\ \bibinfo {year} {1976})\BibitemShut
  {NoStop}%
\bibitem [{\citenamefont {Wilkinson}(1965)}]{wilkinson1965algebraic}%
  \BibitemOpen
  \bibfield  {author} {\bibinfo {author} {\bibfnamefont {J.~H.}\ \bibnamefont
  {Wilkinson}},\ }\href@noop {} {\emph {\bibinfo {title} {The algebraic
  eigenvalue problem}}},\ Vol.~\bibinfo {volume} {87}\ (\bibinfo  {publisher}
  {Clarendon Press Oxford},\ \bibinfo {year} {1965})\BibitemShut {NoStop}%
\bibitem [{\citenamefont {Tao}\ and\ \citenamefont {Vu}(2011)}]{tao2011random}%
  \BibitemOpen
  \bibfield  {author} {\bibinfo {author} {\bibfnamefont {T.}~\bibnamefont
  {Tao}}\ and\ \bibinfo {author} {\bibfnamefont {V.}~\bibnamefont {Vu}},\
  }\href@noop {} {\bibfield  {journal} {\bibinfo  {journal} {Acta mathematica}\
  }\textbf {\bibinfo {volume} {206}},\ \bibinfo {pages} {127} (\bibinfo {year}
  {2011})}\BibitemShut {NoStop}%
\bibitem [{\citenamefont {Trefethen}\ and\ \citenamefont
  {III}(1997)}]{Trefethen}%
  \BibitemOpen
  \bibfield  {author} {\bibinfo {author} {\bibfnamefont {L.~N.}\ \bibnamefont
  {Trefethen}}\ and\ \bibinfo {author} {\bibfnamefont {D.~B.}\ \bibnamefont
  {III}},\ }\href@noop {} {\emph {\bibinfo {title} {Numerical Linear
  Algebra}}}\ (\bibinfo  {publisher} {Siam},\ \bibinfo {year}
  {1997})\BibitemShut {NoStop}%
\bibitem [{\citenamefont {Moro}\ \emph {et~al.}(1997)\citenamefont {Moro},
  \citenamefont {Burke},\ and\ \citenamefont {Overton}}]{moro1997lidskii}%
  \BibitemOpen
  \bibfield  {author} {\bibinfo {author} {\bibfnamefont {J.}~\bibnamefont
  {Moro}}, \bibinfo {author} {\bibfnamefont {J.~V.}\ \bibnamefont {Burke}}, \
  and\ \bibinfo {author} {\bibfnamefont {M.~L.}\ \bibnamefont {Overton}},\
  }\href@noop {} {\bibfield  {journal} {\bibinfo  {journal} {SIAM Journal on
  Matrix Analysis and Applications}\ }\textbf {\bibinfo {volume} {18}},\
  \bibinfo {pages} {793} (\bibinfo {year} {1997})}\BibitemShut {NoStop}%
\bibitem [{\citenamefont {Gray}(2006)}]{GrayReview}%
  \BibitemOpen
  \bibfield  {author} {\bibinfo {author} {\bibfnamefont {R.~M.}\ \bibnamefont
  {Gray}},\ }\href@noop {} {\emph {\bibinfo {title} {Toeplitz and Circulant
  Matrices: A review}}}\ (\bibinfo  {publisher} {Now Pub},\ \bibinfo {year}
  {2006})\BibitemShut {NoStop}%
\bibitem [{\citenamefont {Bourgade}\ and\ \citenamefont
  {Yau}(2013)}]{bourgade2013eigenvector}%
  \BibitemOpen
  \bibfield  {author} {\bibinfo {author} {\bibfnamefont {P.}~\bibnamefont
  {Bourgade}}\ and\ \bibinfo {author} {\bibfnamefont {H.-T.}\ \bibnamefont
  {Yau}},\ }\href@noop {} {\bibfield  {journal} {\bibinfo  {journal} {arXiv
  preprint arXiv:1312.1301}\ } (\bibinfo {year} {2013})}\BibitemShut {NoStop}%
\bibitem [{\citenamefont {Stoica}\ and\ \citenamefont
  {Moses}(1997)}]{stoica1997introduction}%
  \BibitemOpen
  \bibfield  {author} {\bibinfo {author} {\bibfnamefont {P.}~\bibnamefont
  {Stoica}}\ and\ \bibinfo {author} {\bibfnamefont {R.~L.}\ \bibnamefont
  {Moses}},\ }\href@noop {} {\emph {\bibinfo {title} {Introduction to spectral
  analysis}}},\ Vol.~\bibinfo {volume} {1}\ (\bibinfo  {publisher} {Prentice
  hall Upper Saddle River},\ \bibinfo {year} {1997})\BibitemShut {NoStop}%
\bibitem [{\citenamefont {Tu}(2008)}]{Loring2008}%
  \BibitemOpen
  \bibfield  {author} {\bibinfo {author} {\bibfnamefont {L.~W.}\ \bibnamefont
  {Tu}},\ }\href@noop {} {\emph {\bibinfo {title} {An Introduction to
  Manifold}}}\ (\bibinfo  {publisher} {New York, Springer},\ \bibinfo {year}
  {2008})\BibitemShut {NoStop}%
\bibitem [{\citenamefont {Edelman}\ \emph {et~al.}(1994)\citenamefont
  {Edelman}, \citenamefont {Kostlan},\ and\ \citenamefont
  {Shub}}]{Edelman1994}%
  \BibitemOpen
  \bibfield  {author} {\bibinfo {author} {\bibfnamefont {A.}~\bibnamefont
  {Edelman}}, \bibinfo {author} {\bibfnamefont {E.}~\bibnamefont {Kostlan}}, \
  and\ \bibinfo {author} {\bibfnamefont {M.}~\bibnamefont {Shub}},\ }\href@noop
  {} {\bibfield  {journal} {\bibinfo  {journal} {J. Amer. Math. Soc.}\ }\textbf
  {\bibinfo {volume} {7}},\ \bibinfo {pages} {247} (\bibinfo {year}
  {1994})}\BibitemShut {NoStop}%
\bibitem [{\citenamefont {Tao}\ and\ \citenamefont {Vu}(2012)}]{TaoVu2012}%
  \BibitemOpen
  \bibfield  {author} {\bibinfo {author} {\bibfnamefont {T.}~\bibnamefont
  {Tao}}\ and\ \bibinfo {author} {\bibfnamefont {V.}~\bibnamefont {Vu}},\
  }\href@noop {} {\  (\bibinfo {year} {2012})},\ \Eprint
  {http://arxiv.org/abs/arXiv:1206.1893 [math.PR]} {arXiv:1206.1893 [math.PR]}
  \BibitemShut {NoStop}%
\bibitem [{\citenamefont {Boettcher}\ \emph {et~al.}(2003)\citenamefont
  {Boettcher}, \citenamefont {Embree},\ and\ \citenamefont
  {Sokolov}}]{BottEmbreeSokolov}%
  \BibitemOpen
  \bibfield  {author} {\bibinfo {author} {\bibfnamefont {A.}~\bibnamefont
  {Boettcher}}, \bibinfo {author} {\bibfnamefont {M.}~\bibnamefont {Embree}}, \
  and\ \bibinfo {author} {\bibfnamefont {V.~I.}\ \bibnamefont {Sokolov}},\
  }\href@noop {} {\bibfield  {journal} {\bibinfo  {journal} {Math. Comp.}\
  }\textbf {\bibinfo {volume} {72}},\ \bibinfo {pages} {1329 } (\bibinfo {year}
  {2003})}\BibitemShut {NoStop}%
\bibitem [{\citenamefont {Mays}(2014)}]{Mays2014}%
  \BibitemOpen
  \bibfield  {author} {\bibinfo {author} {\bibfnamefont {A.}~\bibnamefont
  {Mays}},\ }\href@noop {} {\bibfield  {journal} {\bibinfo  {journal} {Asia
  Pacific Mathematics Newsletter}\ }\textbf {\bibinfo {volume} {4}},\ \bibinfo
  {pages} {8} (\bibinfo {year} {2014})}\BibitemShut {NoStop}%
\bibitem [{\citenamefont {Edelman}\ and\ \citenamefont
  {Kostlan}(1995)}]{EdelmanKostlan1995}%
  \BibitemOpen
  \bibfield  {author} {\bibinfo {author} {\bibfnamefont {A.}~\bibnamefont
  {Edelman}}\ and\ \bibinfo {author} {\bibfnamefont {E.}~\bibnamefont
  {Kostlan}},\ }\href@noop {} {\bibfield  {journal} {\bibinfo  {journal}
  {Bulletin (new series) of the American Mathematical Society}\ }\textbf
  {\bibinfo {volume} {32}} (\bibinfo {year} {1995})}\BibitemShut {NoStop}%
\bibitem [{\citenamefont {Edelman}(1997)}]{Edelman}%
  \BibitemOpen
  \bibfield  {author} {\bibinfo {author} {\bibfnamefont {A.}~\bibnamefont
  {Edelman}},\ }\href@noop {} {\bibfield  {journal} {\bibinfo  {journal}
  {Journal of Multivariate Analysis}\ }\textbf {\bibinfo {volume} {60}},\
  \bibinfo {pages} {203} (\bibinfo {year} {1997})}\BibitemShut {NoStop}%
\bibitem [{\citenamefont {Ticozzi}\ \emph {et~al.}(2012)\citenamefont
  {Ticozzi}, \citenamefont {Lucchese}, \citenamefont {Cappellaro},\ and\
  \citenamefont {Viola}}]{ticozzi2012hamiltonian}%
  \BibitemOpen
  \bibfield  {author} {\bibinfo {author} {\bibfnamefont {F.}~\bibnamefont
  {Ticozzi}}, \bibinfo {author} {\bibfnamefont {R.}~\bibnamefont {Lucchese}},
  \bibinfo {author} {\bibfnamefont {P.}~\bibnamefont {Cappellaro}}, \ and\
  \bibinfo {author} {\bibfnamefont {L.}~\bibnamefont {Viola}},\ }\href@noop {}
  {\bibfield  {journal} {\bibinfo  {journal} {Automatic Control, IEEE
  Transactions on}\ }\textbf {\bibinfo {volume} {57}},\ \bibinfo {pages} {1931}
  (\bibinfo {year} {2012})}\BibitemShut {NoStop}%
\bibitem [{\citenamefont {Ticozzi}\ and\ \citenamefont
  {Viola}(2012)}]{ticozzi2012stabilizing}%
  \BibitemOpen
  \bibfield  {author} {\bibinfo {author} {\bibfnamefont {F.}~\bibnamefont
  {Ticozzi}}\ and\ \bibinfo {author} {\bibfnamefont {L.}~\bibnamefont
  {Viola}},\ }\href@noop {} {\bibfield  {journal} {\bibinfo  {journal}
  {Philosophical Transactions of the Royal Society A: Mathematical, Physical
  and Engineering Sciences}\ }\textbf {\bibinfo {volume} {370}},\ \bibinfo
  {pages} {5259} (\bibinfo {year} {2012})}\BibitemShut {NoStop}%
\end{thebibliography}%

\newpage{}

\section{Appendix: Matlab code}

\begin{algorithm}[H]
\begin{centering}
\includegraphics[scale=0.9]{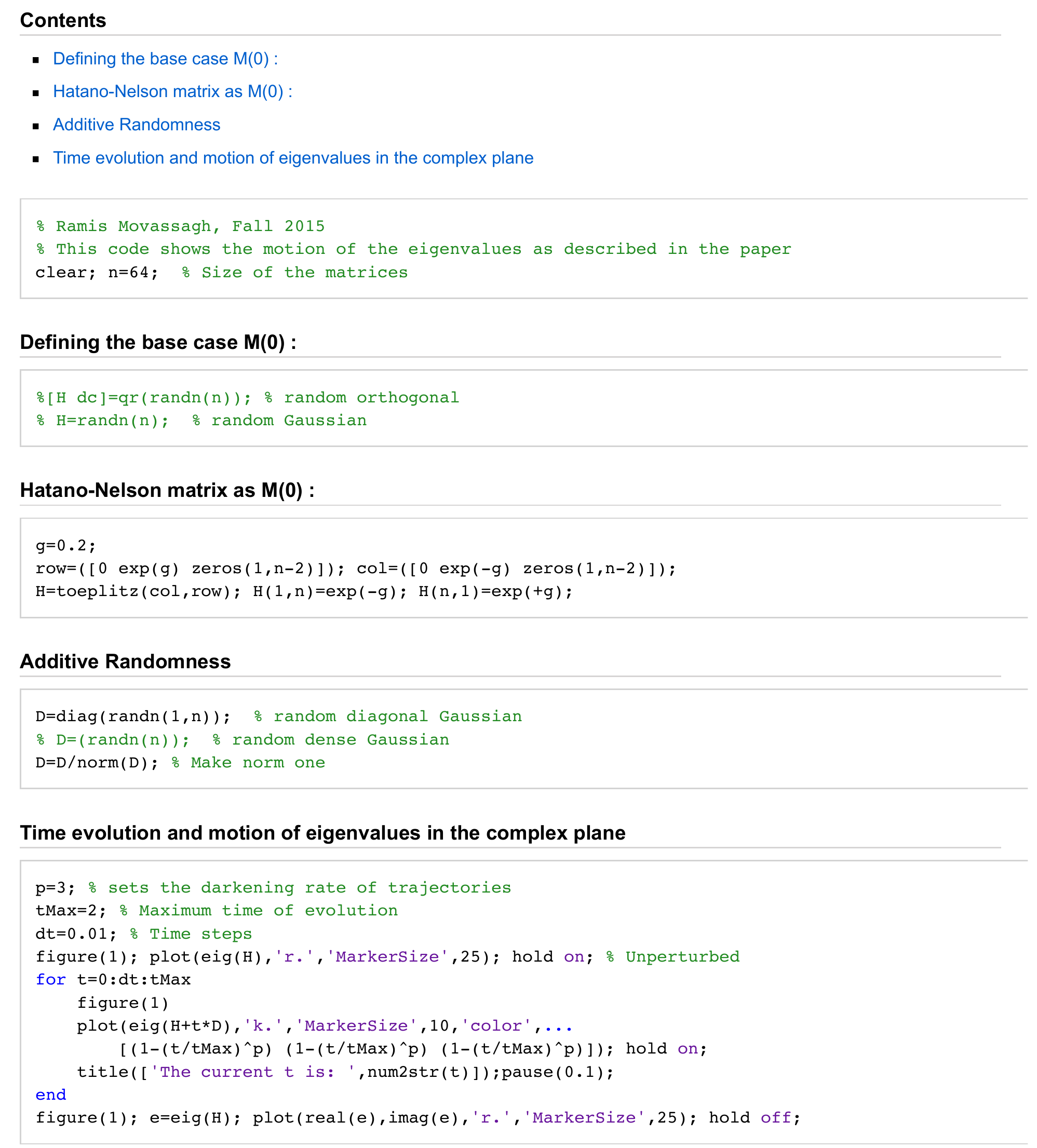} 
\par\end{centering}

\caption{Code for plots shown in Sections \ref{sec:Motivation} and \ref{sub:Further-illustrations}.
The base case can be changed by (un)commenting.}
\end{algorithm}

\end{document}